\documentclass[10pt,reqno]{amsart}
\usepackage{charter}
\usepackage{mathrsfs}
\usepackage{enumitem}
\usepackage{amsmath, amsfonts, amssymb, amsthm, amscd}
\usepackage{microtype}%make paper fine to read
\usepackage[all]{xy}
\usepackage[pagebackref]{hyperref}

\usepackage[misc]{ifsym}%add Letter symbol

\usepackage{array} %control height of the arrays
\usepackage{booktabs} %as package array
\usepackage{multirow}%control height of rows in arrays and tubulars
\usepackage{listings}%insert program text.

\numberwithin{equation}{section}
\newtheorem{theorem} {Theorem} [section]
\newtheorem{proposition}[theorem]{Proposition}
\newtheorem{corollary}  [theorem]     {Corollary}
\newtheorem{lemma}  [theorem]     {Lemma}

\newtheorem{remark}  [theorem]     {Remark}

\newcommand{\comment}[1]{}

\usepackage{fancyhdr}
\allowdisplaybreaks[2]%divide the multi-row equations up to pages, the number [n],n=0,1,2,3,4, represents the extent of the division. 0 means no divide, 4 means divide as much as possible.

\begin{document}
\title{Extension formulae on almost complex manifolds}
\author{Jixiang Fu}
\address{Institute of Mathematics, Fudan University, Shanghai 200433, China}
\email{majxfu@fudan.edu.cn}
\author{Haisheng Liu$^\ast$}
\thanks{This paper has been accepted for publication in SCIENCE CHINA Mathematics.}
%\thanks{32G07,53C15,13D10,14D15}
\thanks{\today}
\address{School of Mathematical Sciences, Fudan University, Shanghai 200433, China}
\email{haishengliu2014@163.com}
%\\
%Center of Mathematical Sciences, Zhejiang University,\\ Hangzhou 310027, China.\\
%haishengliu2014@163.com}
%\date{\today}
%\maketitle

%\vspace{-20pt}

\begin{abstract}
We give the extension formulae on almost complex manifolds and give decompositions of the extension formulae. As applications, we study $(n,0)$-forms, the $(n,0)$-Dolbeault cohomology group and $(n,q)$-forms on almost complex manifolds.
\end{abstract}
\maketitle
%\parskip=5pt
%\baselineskip=15pt
%\pagebreak

%\vspace{-20pt}

%\setcounter{tocdepth}{1}
%\tableofcontents

%%%%%%%%%%%%%%%%%%%%%%%%%%%%%%%%%%%%%%%%%%%%%%%%%%%%%%%%%%%%%%%%%%%%%%%%%%%

%\baselineskip{12pt}

%\setcounter{section}{-1}

%\email{haishengliu2014@163.com}
%\date{\today}
%

\tableofcontents
\section{Introduction}
We would like to introduce some background knowledge in complex geometry. Let $(M,J)$ be a compact complex manifold of complex dimension $n$. Vector-valued differential forms in $A^{0,1}(M,T^{1,0}M)$ are usually called the Beltrami differentials. A Beltrami differential $\phi$ defines a linear map
\begin{equation*}
\phi:A^{1,0}(M)\rightarrow A^{0,1}(M).
\end{equation*}
Let $\{\theta^{i}\}_{i=1}^{n}$ be a local basis of $A^{1,0}(M)$. If $I+\phi$ is an isomorphism, we know that
\begin{equation*}
\bigl\{(I+\phi)\theta^{i},\ \overline{(I+\phi)\theta^{i}}\bigr\}_{i=1}^{n}
\end{equation*}
forms a local basis of $A^{1}(M)$, and meanwhile $\phi$ defines an almost complex structure $J_{\phi}$ over $M$. $J_{\phi}$ is integrable (i.e. $J_{\phi}$ is induced by some complex structure on $M$) if and only if $\phi$ satisfies the well-known Maurer-Cartan equation
\begin{equation*}
\bar\partial\phi=\frac{1}{2}[\phi,\phi],
\end{equation*}
where $[\ ,\ ]$ is the Fr\"olicher-Nijenhuis bracket (see (\ref{fnbracket})).

The exponential operator for the contraction operator $i_{\phi}$ of a Beltrami differential $\phi$ is defined as
\begin{equation*}
e^{i_{\phi}}:=\sum_{k=0}^{\infty}\frac{1}{k!}i_{\phi}^{k},
\end{equation*}
where $i_{\phi}^{k}=\underbrace{i_{\phi}\circ\cdots\circ i_{\phi}}_{k}$ and $i_{\phi}^{0}=I$. The sum is actually finite, since the dimension of $M$ is finite. The application of this operator to the deformation problem can be least traced back to Todorov \cite{todorov1989weil} as far as the authors know. One may also see \cite{clemens2005geometry} for more discussion on this operator.

Liu-Rao-Yang \cite[Theorem 3.4]{MR3302118} proved a so-called extension formula
\begin{equation*}
e^{i_{\phi}}\circ\nabla\circ e^{i_{\phi}}=\nabla-\mathcal{L}^{\nabla}_{\phi}-i_{\frac{1}{2}[\phi,\phi]}
\end{equation*}
for the Chern connection $\nabla$ on a hermitian holomorphic vector bundle $(E,h)$ of $M$, where $\mathcal{L}^{\nabla}_{\phi}$ is the generalized Lie derivative (see (\ref{eq301})).
By this formula, the authors in \cite{MR3302118} constructed relations of the $\bar\partial$-operators and studied $(n,0)$-forms under varying complex structures.

To deal with $(p,q)$-forms, Rao-Zhao \cite{zhao2015extension} introduced an extended exponential operator $e^{i_{\phi}|i_{\bar\phi}}$,
\begin{equation*}
e^{i_{\phi}|i_{\bar\phi}}:A^{p,q}_{J}(M)\rightarrow A^{p,q}_{\phi}(M).
\end{equation*}
Rao-Zhao \cite{rao2018several} obtained the extension formula for the extended exponential operator. Moreover, the authors in \cite{rao2018several} used the extension formula to study the properties of Hodge numbers along the small deformation of complex structures.

The extension formula on complex manifolds plays an important role in the deformation of complex structures, especially in describing the decomposition of complex differential forms on a complex manifold with respect to different complex structures.
In this paper, we generalize the extension formula from complex manifolds to almost complex manifolds and give two kinds of decompositions. We expect that our extension formulae will also play an important role in the deformation of almost complex structures.

The problem of deformations of almost complex structures is interesting for its own sake.
Let $M$ be a differentiable manifold.
An interesting problem is how (or under what conditions) one can deform two given almost complex structures $J_{1}$ and $J_{2}$ on $M$.
This problem becomes more interesting if one of these two almost complex structures is integrable.
Another interesting problem is what one can say about the geometry (and topology) of the moduli space of all almost complex structures on $M$. McDuff \cite{McDuff2000almostcomplexstructures} obtained many interesting results on $S^{2}\times S^{2}$.

This paper is arranged as follows.

In Section \ref{sec3}, we give a description of Beltrami differentials on almost complex manifolds.

In Section \ref{sec1}, we recall some basic knowledge on contraction operators and generalized Lie derivatives.

In Section \ref{sec2}, firstly we give and prove the commutator of the connection $\nabla$ and the multi-contraction $i^{k}_{\phi}$.
\begin{proposition}\textup{(= Proposition \ref{proposition1})}
Let $(M,J)$ be an almost complex manifold, $\phi\in A^{0,1}_{J}(M,T^{1,0}M)$ be a Beltrami differential,
$E$ be a vector bundle on $M$ and $\nabla$ be a connection of $E$.  Then
\begin{equation*}
[\nabla,i^{k}_{\phi}]
=ki^{k-1}_{\phi}\circ[\nabla,i_{\phi}]
-\frac{k(k-1)}{2}i^{k-2}_{\phi}\circ i_{[\phi,\phi]}
-\frac{k(k-1)(k-2)}{3!}i^{k-3}_{\phi}\circ i_{(i_{[\phi,\phi]}\phi-i_{\phi}[\phi,\phi])}.
\end{equation*}
\end{proposition}
By this lemma, we give the extension formula in the almost complex background.
\begin{theorem}\textup{(= Theorem \ref{theorem1})}
Let $(M,J)$ be an almost complex manifold and $\phi\in A_{J}^{0,1}(M,T^{1,0}M)$. Let $E$ be a vector bundle over $M$ and $\nabla$ be a connection of $E$. Then on $A_{J}^{\ast,\ast}(M,E)$,
\begin{equation*}
e^{-i_{\phi}}\circ\nabla\circ e^{i_{\phi}}
=\nabla-\mathcal{L}^{\nabla}_{\phi}-i_{\frac{1}{2}[\phi,\phi]}-i_{\frac{1}{3!}(i_{[\phi,\phi]}\phi-i_{\phi}[\phi,\phi])}.
\end{equation*}
\end{theorem}
In \cite{xia2019derivation}, Xia tried to prove similar result using a different method only without noting Lemma \ref{lemma17}. In the newest version of \cite{xia2018derivation}, Xia gives a correction of the results in \cite{xia2019derivation}.

As a special case of Theorem \ref{theorem1}, when the vector bundle $E$ is the trivial bundle and the connection is the ordinary exterior differential operator $d$, we have the following corollary.
\begin{corollary}\textup{(= Corollary \ref{corollary3})}
Let $(M,J)$ be an almost complex manifold. Then
%$d$ be the ordinary exterior differential operator over $M$.
\begin{equation*}
e^{-i_{\phi}}\circ d\circ e^{i_{\phi}}
=d-\mathcal{L}_{\phi}-i_{\frac{1}{2}[\phi,\phi]}-i_{\frac{1}{3!}(i_{[\phi,\phi]}\phi-i_{\phi}[\phi,\phi])}.
\end{equation*}
\end{corollary}
A key observation in this paper is the following lemma.
\begin{lemma}\textup{(= Lemma \ref{lemma17})}
If $\varphi=\rho\otimes X,\psi=\tau\otimes Y\in A^{0,1}(M,T^{1,0}M)$, then
\begin{equation*}
[\varphi,\psi]=
\rho\wedge\tau\otimes[X,Y]+\rho\wedge\mathcal{L}_{X}\tau\otimes Y-\mathcal{L}_{Y}\rho\wedge\tau\otimes X,
\end{equation*}
and
\begin{equation*}%\label{decomposition}
[\varphi,\psi]\in A^{0,2}(M,T^{1,0}M)\oplus A^{1,1}(M,T^{1,0}M)\oplus A^{0,2}(M,T^{0,1}M).
\end{equation*}
\end{lemma}
We denote the $A^{0,2}(M,T^{1,0}M)$-component of $[\varphi,\psi]$ by $\mathscr{A}(\varphi,\psi)$, the $A^{1,1}(M,T^{1,0}M)$-component by $\mathscr{B}(\varphi,\psi)$ and the $A^{0,2}(M,T^{0,1}M)$-component by $\mathscr{C}(\varphi,\psi)$.

At first we introduce the following \emph{almost complex Maurer-Cartan equation} (see (\ref{mcequation})) on almost complex manifolds.
\begin{equation*}
\mathrm{MC}(\phi):=\bar\partial\phi-\frac{1}{2}\mathscr{A}(\phi,\phi)-\frac{1}{3!}(i_{[\phi,\phi]}\phi-i_{\phi}[\phi,\phi])=0.
\end{equation*}
The almost complex Maurer-Cartan equation reduces to the ordinary Maurer-Cartan equation on complex manifolds.
Next, we deal with the extended exponential operator $e^{i_{\phi}|i_{\bar\phi}}$. As preparations, we give the following analogy for almost complex structures of the fundamental lemmas proved by Rao-Zhao on complex manifolds in \cite{rao2018several}.
\begin{lemma}\textup{(= Lemma \ref{lemma6})}\cite[Lemma 2.12]{rao2018several}
\begin{equation*}
e^{i_{\phi}}((I-\bar\phi\cdot\phi+\bar\phi)\lrcorner\,\theta^{k})
=(I+\phi)\lrcorner\,\theta^{k}=e^{i_{\phi}}(\theta^{k}),
\end{equation*}
and
\begin{equation*}
e^{i_{\phi}}((I-\bar\phi\cdot\phi+\bar\phi)\lrcorner\,\theta^{\bar k})
=(I+\bar\phi)\lrcorner\,\theta^{\bar k}=e^{i_{\bar\phi}}(\theta^{\bar k}).
\end{equation*}
\end{lemma}
\begin{lemma}\textup{(= Lemma \ref{lemma12})}\cite[Lemma 2.12]{rao2018several}
\begin{equation*}
e^{-i_{\phi}}\circ e^{i_{\phi}|i_{\bar\phi}}=(I-\bar\phi\cdot\phi+\bar\phi)\Finv\, ,
\end{equation*}
\begin{equation*}
e^{-i_{\phi}|-i_{\bar\phi}}\circ e^{i_{\phi}}
=((I^{\prime\prime}-\bar\phi\cdot\phi)^{-1}-\bar\phi\cdot(I^{\prime}-\phi\cdot\bar\phi)^{-1})\Finv\,.
\end{equation*}
\end{lemma}
The \emph{simultaneous contraction operator} $\Finv\,$ is introduced in \cite[Equation 2.8]{rao2016power} (see (\ref{neweq12})). By these two lemmas, we can prove the following theorem. Let $I:A^{1}(M)\rightarrow A^{1}(M)$, $I^{\prime}:A^{1,0}(M)\rightarrow A^{1,0}(M)$ and $I^{\prime\prime}:A^{0,1}(M)\rightarrow A^{0,1}(M)$ be the identity maps of the corresponding spaces. Clearly $I=I^{\prime}+I^{\prime}$.
\begin{theorem}\textup{(= Theorem \ref{theorem3})}
Let $e^{i_{\phi}|i_{\bar\phi}}$ be the extended exponential operator. Then for any $\alpha\in A^{p,q}_{J}(M)$,
\begin{align}
d(e^{i_{\phi}|i_{\bar\phi}}(\alpha))
=&\,e^{i_{\phi}|i_{\bar\phi}}
\bigl\{(I^{\prime}+(I^{\prime\prime}-\bar\phi\cdot\phi)^{-1}-\bar\phi\cdot(I^{\prime}-\phi\cdot\bar\phi)^{-1})\nonumber\\
&\Finv\,(d+[\mu,i_{\phi}]+[\partial,i_{\phi}]-i_{\frac{1}{2}(\mathscr{B}(\phi,\phi)+\mathscr{C}(\phi,\phi))}
+i_{\mathrm{MC}(\phi)})\circ(I-\bar\phi\cdot\phi+\bar\phi)\Finv\,\alpha\bigl\}.\nonumber
\end{align}
\end{theorem}
In Section \ref{sec4}, we give the decompositions of the extension formula on almost complex manifolds. Precisely, we prove the following theorem.
\begin{theorem}\textup{(= Theorem \ref{theorem4})}
Let $(M,J)$ be an almost complex manifold, $\phi\in A^{0,1}_{J}(M,T^{1,0}M)$ be a Beltrami differential, $E\rightarrow M$ be a vector bundle on $M$ and $\nabla$ be a connection of $E$. We have
\begin{align*}
e^{-i_{\phi}}\circ\mu\circ e^{i_{\phi}}
&=\mu-\mathcal{L}^{\mu}_{\phi}
-i_{\frac{1}{2}(\mathscr{B}(\phi,\phi)+\mathscr{C}(\phi,\phi))}
-i_{\frac{1}{3!}(i_{\mathscr{C}(\phi,\phi)}\phi-i_{\phi}\mathscr{B}(\phi,\phi))},\\
e^{-i_{\phi}}\circ\nabla^{1,0}\circ e^{i_{\phi}}
&=\nabla^{1,0}-\mathcal{L}^{\nabla^{1,0}}_{\phi}-i_{\frac{1}{2}\mathscr{A}(\phi,\phi)},\\
e^{-i_{\phi}}\circ\nabla^{0,1}\circ e^{i_{\phi}}
&=\nabla^{0,1}-\mathcal{L}^{\nabla^{0,1}}_{\phi},\\
e^{-i_{\phi}}\circ\bar\mu\circ e^{i_{\phi}}
&=\bar\mu-\mathcal{L}^{\bar\mu}_{\phi}.
\end{align*}
\end{theorem}
The exterior differential $d$ on an almost complex manifold has a natural decomposition according to types,
\begin{equation}\label{eq67}
d=\mu+\partial+\bar\partial+\bar\mu.
\end{equation}
The symbols $\mu$ and $\bar\mu$ denote the $(2,-1)$-part and the $(-1,2)$-part of $d$ respectively (as in \cite{cirici2018dolbeault}). For more details on the operators $\mu$ and $\bar\mu$, see Section \ref{sec3}. Considering (\ref{eq67}), we have the following corollary.
\begin{corollary}\textup{(= Corollary \ref{corollary5})}
Let $(M,J)$ be an almost complex manifold and $\phi\in A_{J}^{0,1}(M,T^{1,0}M)$ be a Beltrami differential. We have
\begin{align*}
e^{-i_{\phi}}\circ\mu\circ e^{i_{\phi}}
&=\mu-\mathcal{L}^{\mu}_{\phi}-i_{\frac{1}{2}(\mathscr{B}(\phi,\phi)+\mathscr{C}(\phi,\phi))}
-i_{\frac{1}{3!}(i_{\mathscr{C}(\phi,\phi)}\phi-i_{\phi}\mathscr{B}(\phi,\phi))},\\
e^{-i_{\phi}}\circ\partial\circ e^{i_{\phi}}
&=\partial-\mathcal{L}^{\partial}_{\phi}-i_{\frac{1}{2}\mathscr{A}(\phi,\phi)},\\
e^{-i_{\phi}}\circ\bar\partial\circ e^{i_{\phi}}
&=\bar\partial-\mathcal{L}^{\bar\partial}_{\phi},\\
e^{-i_{\phi}}\circ\bar\mu\circ e^{i_{\phi}}
&=\bar\mu-\mathcal{L}^{\bar\mu}_{\phi}.
\end{align*}
\end{corollary}
By comparing types with respect to $J$ and $J_{\phi}$ and noting that the extended exponential operator $e^{i_{\phi}|i_{\bar\phi}}$ preserves types (i.e. $e^{i_{\phi}|i_{\bar\phi}}:A^{p,q}_{J}(M)\rightarrow A^{p,q}_{J_{\phi}}(M)$), we can prove the following theorem.
\begin{theorem} \textup{(= Theorem \ref{theorem2})}
Let $(M,J)$ be an almost complex manifold and $d=\mu+\partial+\bar\partial+\bar\mu$ be the decomposition of the exterior differential operator with respect to $J$. Let $\phi\in A^{0,1}_{J}(M,T^{1,0}M)$ be a Beltrami differential on $M$ that generates a new almost complex structure $J_{\phi}$ on $M$ and $d=\mu_{\phi}+\partial_{\phi}+\bar\partial_{\phi}+\bar\mu_{\phi}$ be the decomposition of the exterior differential operator with respect to $J_{\phi}$. Then we have the following equations.
\begin{align*}
\mu_{\phi}\circ e^{i_{\phi}|i_{\bar\phi}}
=&\,e^{i_{\phi}|i_{\bar\phi}}((O_{1}\circ O_{2}\circ O_{3})_{J}^{2,-1}),\\
\partial_{\phi}\circ e^{i_{\phi}|i_{\bar\phi}}
=&\,e^{i_{\phi}|i_{\bar\phi}}((O_{1}\circ O_{2}\circ O_{3})_{J}^{1,0}),\\
\bar\partial_{\phi}\circ e^{i_{\phi}|i_{\bar\phi}}
=&\,e^{i_{\phi}|i_{\bar\phi}}((O_{1}\circ O_{2}\circ O_{3})_{J}^{0,1}),\\
\bar\mu_{\phi}\circ e^{i_{\phi}|i_{\bar\phi}}
=&\,e^{i_{\phi}|i_{\bar\phi}}((O_{1}\circ O_{2}\circ O_{3})_{J}^{-1,2}).
\end{align*}
\end{theorem}
The operators $O_{1}$, $O_{2}$ and $O_{3}$ in the above theorem are
\begin{align*}
O_{1}&=I^{\prime}+(I^{\prime\prime}-\bar\phi\cdot\phi)^{-1}-\bar\phi\cdot(I^{\prime}-\phi\cdot\bar\phi)^{-1},\\
O_{2}&=d+[\mu,i_{\phi}]+[\partial,i_{\phi}]-i_{\frac{1}{2}(\mathscr{B}(\phi,\phi)+\mathscr{C}(\phi,\phi))}
+i_{\mathrm{MC}(\phi)},\\
O_{3}&=I-\bar\phi\cdot\phi+\bar\phi.
\end{align*}
In Section \ref{sec6}, we give some applications of the extension formulae.  We get the following results.
\begin{proposition}\textup{(= Proposition \ref{proposition2})}
Let $(M,J)$ be an almost complex manifold and $\phi$ be a Beltrami differential that satisfies the almost complex Maurer-Cartan equation. For any smooth $J^{n,0}$-form $\Omega$, $e^{i_{\phi}|i_{\bar\phi}}(\Omega)$ is a $\bar\partial_{\phi}$-closed $J_{\phi}^{n,0}$-form if and only if
\begin{equation*}
(-1)^{n}(\bar\partial\Omega+\partial(\phi\lrcorner\Omega)-\frac{1}{2}\mathscr{B}(\phi,\phi)\lrcorner\Omega)
-\bar\phi\cdot(I^{\prime}-\phi\cdot\bar\phi)^{-1}\lrcorner\,\bar\mu\Omega=0.
\end{equation*}
\end{proposition}
To distinguish the case of almost complex manifolds and complex manifolds, in this manuscript, we denote an integrable almost complex manifold (i.e. a complex manifold) by $(M,J,N_{J}=0)$. In the integrable case, Proposition \ref{proposition2} reduces to the following known proposition.
\begin{proposition}\cite[Proposition 5.1]{MR3302118}
Let $(M,J,N_{J}=0)$ be a complex manifold and $\phi$ be a Beltrami differential that satisfies the Maurer-Cartan equation. For any smooth $J^{n,0}$-form $\Omega$, $e^{i_{\phi}|i_{\bar\phi}}(\Omega)$ is a $\bar\partial_{\phi}$-closed $J_{\phi}^{n,0}$-form if and only if
\begin{equation*}
\bar\partial\Omega+\partial(\phi\lrcorner\,\Omega)=0.
\end{equation*}
\end{proposition}
As an application of the decomposition formulas, we get the following result concerning the $(n,0)$-Dolbeault cohomology on almost complex manifolds.
\begin{theorem}\textup{(= Theorem \ref{theorem5})}
Let $(M,J)$ be an almost complex manifold and $\phi\in A^{0,1}_{J}(M,T^{1,0}M)$ be a Beltrami differential which generates a new almost complex structure $J_{\phi}$ on $M$. Assume that $\phi$ satisfies the almost complex Maurer-Cartan equation. Then for any $[[\Omega]]\in H^{n,0}_{\mathrm{Dol},J}(M)$, $[[e^{i_{\phi}|i_{\bar\phi}}(\Omega)]]\in H^{n,0}_{\mathrm{Dol},J_{\phi}}(M)$ if and only if
\begin{equation*}
\partial(\phi\lrcorner\Omega)-\frac{1}{2}\mathscr{B}(\phi,\phi)\lrcorner\,\Omega=0.
\end{equation*}
\end{theorem}
For $(n,q)$-forms, we have the following theorem.
\begin{theorem}\textup{(= Theorem \ref{theorem6})}
Let $(M,J)$ be an almost complex manifold and $\phi$ be a Beltrami differential that induces a new almost complex structure. For any smooth $\Xi\in A^{n,q}_{J}(M)$, $e^{i_{\phi}|i_{\bar\phi}}(\Xi)\in A^{n,q}_{\phi}(M)$ is $\bar\partial_{\phi}$-closed if and only if
\begin{small}
\begin{equation*}
(I^{\prime}+(I^{\prime\prime}-\bar\phi\cdot\phi)^{-1})\Finv
\bigl((\bar\partial+[\partial,i_{\phi}]-i_{\frac{1}{2}(\mathscr{B}(\phi,\phi)+\mathscr{C}(\phi,\phi))}
-i_{\bar\phi\cdot(I^{\prime}-\phi\cdot\bar\phi)^{-1}}\circ(\bar\mu+i_{\mathrm{MC}(\phi)}))(I-\bar\phi\cdot\phi)\Finv\Xi\bigr)=0.
\end{equation*}
\end{small}
\end{theorem}
In Section \ref{sec7}, for the readers' convenience, we recall some basic knowledge on the Dolbeault cohomology of almost complex manifolds from \cite{cirici2018dolbeault}.
\section{Beltrami differential on almost complex manifolds}\label{sec3}
Let $(M,J)$ be an almost complex manifold of real dimension $2n$. The complexified tangent bundle of $M$ has the decomposition
\begin{equation*}
T^{\mathbb{C}}M\cong T^{1,0}_{J}M\oplus T^{0,1}_{J}M.
\end{equation*}
The complex cotangent bundle also has similar decomposition
\begin{equation*}
T^{\ast}_{\mathbb{C}}M\cong (T^{\ast}_{J}M)^{1,0}\oplus (T^{\ast}_{J}M)^{0,1}.
\end{equation*}
The space of $(p,q)$-forms with respect to $J$ over $M$ (or $J^{p,q}$-forms) is denoted as $A^{p,q}_{J}(M)$. For any almost complex manifold $(M,J)$, the space of $T^{1,0}_{J}M$-valued $(p,q)$-form should be denoted as $A^{p,q}_{J}(M,T^{1,0}_{J}M)$, especially when there is more than one almost complex structure on $M$ concerned. However, here we denote $A^{p,q}_{J}(M,T^{1,0}_{J}M)$ as $A^{p,q}_{J}(M,T^{1,0}M)$ for a little simplicity. Let $\{\theta^{i}\}_{i=1}^{n}$ be a local basis of $A^{1,0}_{J}(M)$ and $\{e_{i}\}_{i=1}^{n}$ be the dual basis of $\{\theta^{i}\}_{i=1}^{n}$ in $T^{1,0}_{J}M$. Denote the complex conjugate of $\{\theta^{i}\}_{i=1}^{n}$ and $\{e_{i}\}_{i=1}^{n}$ as $\bigl\{\theta^{\bar i}\bigr\}_{i=1}^{n}$ and $\{e_{\bar i}\}_{i=1}^{n}$ respectively. The Nijenhuis tensor with respect to $J$ is defined as
\begin{equation*}
N(X,Y):=[JX,JY]-J[JX,Y]-J[X,JY]-[X,Y].
\end{equation*}
One can easily check that
\begin{align*}
N(e_{\bar j},e_{\bar k})
=&\,[Je_{\bar j},Je_{\bar k}]-J[Je_{\bar j},e_{\bar k}]-J[e_{\bar j},Je_{\bar k}]-[e_{\bar j},e_{\bar k}]\\
=&\,-[e_{\bar j},e_{\bar k}]+iJ[e_{\bar j},e_{\bar k}]+iJ[e_{\bar j},e_{\bar k}]-[e_{\bar j},e_{\bar k}]\\
=&\,-2([e_{\bar j},e_{\bar k}]-iJ[e_{\bar j},e_{\bar k}])\\
=&\,-4[e_{\bar j},e_{\bar k}]^{1,0}.
\end{align*}
For the readers' convenience, here we would like to give some explanations on the decomposition (\ref{eq67}). For $\theta^{i}\in A_{J}^{1,0}(M)$, (\ref{eq67}) implies that
\begin{align*}
\bar\mu\theta^{i}(e_{\bar j},e_{\bar k})
=&d\theta^{i}(e_{\bar j},e_{\bar k})
=e_{\bar j}(\theta^{i}(e_{\bar k}))-e_{\bar k}(\theta^{i}(e_{\bar j}))-\theta^{i}([e_{\bar j},e_{\bar k}])\\
=&-\theta^{i}([e_{\bar j},e_{\bar k}]^{1,0})
=\frac{1}{4}\theta^{i}(N_{\bar j\bar k}^{\ell}e_{\ell})
=\frac{1}{4}N_{\bar j\bar k}^{\ell}\theta^{i}(e_{\ell})\\
=&\frac{1}{4}N\lrcorner\,\theta^{i}(e_{\bar j},e_{\bar k}).
\end{align*}
Thus we know that
\begin{equation*}
\bar\mu=\frac{1}{8}N_{\bar j\bar k}^{i}\theta^{\bar j}\wedge\theta^{\bar k}\otimes e_{k},
\end{equation*}
or equivalently
\begin{equation*}
\bar\mu_{\bar j\bar k}^{i}=\frac{1}{4}N^{i}_{\bar j\bar k}.
\end{equation*}
Similarly we have
\begin{equation*}
\mu_{jk}^{\bar i}=\frac{1}{4}N^{\bar i}_{jk}.
\end{equation*}
Any $\phi=\phi^{i}_{\bar j}\theta^{\bar j}\otimes e_{i}\in A^{0,1}_{J}(M,T^{1,0}M)$ defines a map
\begin{equation*}
i_{\phi}:A^{1,0}_{J}(M)\rightarrow A^{0,1}_{J}(M).
\end{equation*}
Denote $\theta^{i}_{\phi}:=\theta^{i}+i_{\phi}\theta^{i}$, or equivalently
\begin{displaymath}
\left(\begin{array}{c}
\theta_{\phi}^{1}\\
\vdots\\
\theta_{\phi}^{n}\\
\bar\theta_{\phi}^{1}\\
\vdots\\
\bar\theta_{\phi}^{n}
\end{array}\right)
=\left(\begin{array}{ccc|ccc}
1 &0 & & & & \\
0 &\ddots &0&&\phi^{i}_{\bar j}& \\
 & 0&1&&&  \\
\hline
 &&&1&0&\\
 &\phi^{\bar i}_{j}&&0&\ddots&0\\
 &&&&0&1
\end{array}\right)
\left(\begin{array}{c}
\theta^{1}\\
\vdots\\
\theta^{n}\\
\bar\theta^{1}\\
\vdots\\
\bar\theta^{n}
\end{array}\right)
\triangleq\Phi\cdot
\left(\begin{array}{c}
\theta^{1}\\
\vdots\\
\theta^{n}\\
\bar\theta^{1}\\
\vdots\\
\bar\theta^{n}
\end{array}\right),
\end{displaymath}
where $\phi^{\bar i}_{j}\triangleq\overline{\phi^{i}_{\bar j}}$, $i,j=1,2,\ldots,n$. The superscript of $\Phi^{i}_{j}$ is the row index, and from now on, we will always denote the row index as the superscript. When there is no danger of confusion, we denote $\phi=\bigl(\phi^{i}_{\bar j}\bigr)_{n\times n}$ and $\bar\phi=\bigl(\phi^{\bar i}_{j}\bigr)_{n\times n}$. Under this notation, the above equation can be shorted as
\begin{equation}\label{eq68}
\left(\begin{array}{c}
\theta_{\phi}\\
\bar\theta_{\phi}
\end{array}\right)
=\left(\begin{array}{cc}
I^{\prime} & \phi\\
\bar\phi& I^{\prime\prime}
\end{array}\right)
\left(\begin{array}{c}
\theta\\
\bar\theta
\end{array}\right),
\end{equation}
where $\theta_{\phi}\triangleq(\theta^{1}_{\phi},\cdots,\theta^{n}_{\phi})^{T}$, $\theta\triangleq(\theta^{1},\cdots,\theta^{n})^{T}$, $\bar\theta_{\phi}\triangleq(\theta^{\bar 1}_{\phi},\cdots,\theta^{\bar n}_{\phi})^{T}$, $\bar\theta\triangleq(\theta^{\bar 1},\cdots,\theta^{\bar n})^{T}$ and
\begin{equation*}
I=I^{\prime}+I^{\prime\prime}=\sum_{i=1}^{n}\theta^{i}\otimes e_{i}+\sum_{i=1}^{n}\theta^{\bar i}\otimes e_{\bar i},
\quad I^{\prime}=\sum_{i=1}^{n}\theta^{i}\otimes e_{i},I^{\prime\prime}=\sum_{i=1}^{n}\theta^{\bar i}\otimes e_{\bar i}.
\end{equation*}
Note that locally $I$ can be considered as a $2n\times 2n$ identity matrix and $I^{\prime},I^{\prime\prime}$ as $n\times n$ identity matrices. From (\ref{eq68}) we can see that $\{\theta_{\phi},\bar\theta_{\phi}\}$ forms a local basis of $A^{1}(M;\mathbb{C})$ if and only if
\begin{equation*}
\det
\left(\begin{array}{cc}
I^{\prime} & \phi\\
\bar\phi& I^{\prime}
\end{array}\right)\neq0.
\end{equation*}
Linear algebra calculations show that
\begin{equation}\label{det1}
\det
\left(\begin{array}{cc}
I^{\prime} & \phi\\
\bar\phi& I^{\prime\prime}
\end{array}\right)=\det(I^{\prime}-\phi\cdot\bar\phi).
\end{equation}
Here we add a dot ``$\cdot$" in $\phi\cdot\bar\phi$ to emphasize that we takes the matrix multiplication. Note that
\begin{align}
i_{\phi}\bar\phi
=&\,\phi^{i}_{\bar j}\theta^{\bar j}\otimes e_{i}\lrcorner\left(\phi^{\bar p}_{q}\theta^{q}\otimes e_{\bar p}\right)\nonumber\\
=&\,\phi^{i}_{\bar j}\phi^{\bar p}_{q}\delta^{q}_{i}\theta^{\bar j}\otimes e_{\bar p}\nonumber\\
=&\,\phi^{\bar p}_{i}\phi^{i}_{\bar j}\theta^{\bar j}\otimes e_{\bar p}\nonumber\\
\sim&\,\bar\phi\cdot\phi,\label{eq201}
\end{align}
and similarly
\begin{equation}
i_{\bar\phi}\phi\sim\phi\cdot\bar\phi.\label{eq202}
\end{equation}
Here the symbol ``$\sim$" relates vector-valued differential forms with their coefficients.
The readers can also see (\ref{eq201}) and (\ref{eq202}) in the paragraph below \cite[Lemma 2.4]{rao2018several}.

By (\ref{det1}) we see that $\{\theta_{\phi},\bar\theta_{\phi}\}$ forms a local basis of $A^{1}(M;\mathbb{C})$ if and only if
\begin{equation*}
\det(I^{\prime}-\phi\cdot\bar\phi)\neq0.
\end{equation*}
When $\det(I^{\prime}-\phi\cdot\bar\phi)\neq0$, one can check that the inverse of the matrix
$\Phi=\left(\begin{array}{cc}
I^{\prime} & \phi\\
\bar\phi& I^{\prime\prime}
\end{array}\right)$ is
\begin{displaymath}
\Phi^{-1}=\left(\begin{array}{cc}
(I^{\prime}-\phi\cdot\bar\phi)^{-1} & -\phi\cdot(I^{\prime\prime}-\bar\phi\cdot\phi)^{-1}\\
&\\
-\bar\phi\cdot(I^{\prime}-\phi\cdot\bar\phi)^{-1}& (I^{\prime\prime}-\bar\phi\cdot\phi)^{-1}
\end{array}\right)_{2n\times 2n}.
\end{displaymath}
Actually, we have
\begin{align*}
\Phi\cdot\Phi^{-1}=&
\left(\begin{array}{cc}
I^{\prime} & \phi\\
\bar\phi& I^{\prime\prime}
\end{array}\right)\cdot
\left(\begin{array}{cc}
(I^{\prime}-\phi\cdot\bar\phi)^{-1} & -\phi\cdot(I^{\prime\prime}-\bar\phi\cdot\phi)^{-1}\\
&\\
-\bar\phi\cdot(I^{\prime}-\phi\cdot\bar\phi)^{-1}& (I^{\prime\prime}-\bar\phi\cdot\phi)^{-1}
\end{array}\right)\\
=&\left(\begin{array}{cc}
(I^{\prime}-\phi\cdot\bar\phi)^{-1}-\phi\cdot\bar\phi\cdot(I^{\prime}-\phi\cdot\bar\phi)^{-1}
&0\\
0&-\bar\phi\cdot\phi\cdot(I^{\prime\prime}-\bar\phi\cdot\phi)^{-1}+(I^{\prime\prime}-\bar\phi\cdot\phi)^{-1}
\end{array}\right)\\
=&\left(\begin{array}{cc}
(I^{\prime}-\phi\cdot\bar\phi)(I^{\prime}-\phi\cdot\bar\phi)^{-1}
&0\\
0&(I^{\prime\prime}-\bar\phi\cdot\phi)(I^{\prime\prime}-\bar\phi\cdot\phi)^{-1}
\end{array}\right)\\
=&\left(\begin{array}{cc}
I^{\prime}&0\\
0&I^{\prime\prime}
\end{array}\right).
\end{align*}
By the uniqueness of the inverse matrix, we know that $\Phi^{-1}\cdot\Phi=I$, i.e. we have
\begin{align*}
\Phi^{-1}\cdot\Phi
=&\left(\begin{array}{cc}
(I^{\prime}-\phi\cdot\bar\phi)^{-1} & -\phi\cdot(I^{\prime\prime}-\bar\phi\cdot\phi)^{-1}\\
&\\
-\bar\phi\cdot(I^{\prime}-\phi\cdot\bar\phi)^{-1}& (I^{\prime\prime}-\bar\phi\cdot\phi)^{-1}
\end{array}\right)\cdot
\left(\begin{array}{cc}
I^{\prime} & \phi\\
\bar\phi& I^{\prime\prime}
\end{array}\right)\\
=&\left(\begin{array}{cc}
(I^{\prime}-\phi\cdot\bar\phi)^{-1}-\phi\cdot(I^{\prime\prime}-\bar\phi\cdot\phi)^{-1}\cdot\bar\phi
&(I^{\prime}-\phi\cdot\bar\phi)^{-1}\cdot\phi-\phi\cdot(I^{\prime\prime}-\bar\phi\cdot\phi)^{-1}\\
-\bar\phi\cdot(I^{\prime}-\phi\cdot\bar\phi)^{-1}+(I^{\prime\prime}-\bar\phi\cdot\phi)^{-1}\cdot\bar\phi
&-\bar\phi\cdot(I^{\prime}-\phi\cdot\bar\phi)^{-1}\cdot\phi+(I^{\prime\prime}-\bar\phi\cdot\phi)^{-1}
\end{array}\right)\\
=&\left(\begin{array}{cc}
I^{\prime}&0\\
0&I^{\prime\prime}
\end{array}\right).
\end{align*}
By the above equation, we know that
\begin{align*}
(I^{\prime}-\phi\cdot\bar\phi)^{-1}-\phi\cdot(I^{\prime\prime}-\bar\phi\cdot\phi)^{-1}\cdot\bar\phi&=I^{\prime},\\
(I^{\prime}-\phi\cdot\bar\phi)^{-1}\cdot\phi-\phi\cdot(I^{\prime\prime}-\bar\phi\cdot\phi)^{-1}&=0,\\
-\bar\phi\cdot(I^{\prime}-\phi\cdot\bar\phi)^{-1}+(I^{\prime\prime}-\bar\phi\cdot\phi)^{-1}\cdot\bar\phi&=0,\\
-\bar\phi\cdot(I^{\prime}-\phi\cdot\bar\phi)^{-1}\cdot\phi+(I^{\prime\prime}-\bar\phi\cdot\phi)^{-1}&=I^{\prime\prime}.
\end{align*}
We summarize these equations as the following remark.
\begin{remark}
Let $\phi\in A^{0,1}(M,T^{1,0}M)$ be a Beltrami differential. If
\begin{equation*}
\det(I^{\prime}-\phi\cdot\bar\phi)\neq0,
\end{equation*}
then we have
\begin{equation}\label{neq3}
(I^{\prime}-\phi\cdot\bar\phi)^{-1}\cdot\phi=\phi\cdot(I^{\prime\prime}-\bar\phi\cdot\phi)^{-1},
\end{equation}
\begin{equation}\label{neq2}
(I^{\prime}-\phi\cdot\bar\phi)^{-1}-\phi\cdot(I^{\prime\prime}-\bar\phi\cdot\phi)^{-1}\cdot\bar\phi=I^{\prime}.
\end{equation}
The complex conjugate of (\ref{neq3}) and (\ref{neq2}) also hold.
\end{remark}
When $\{\theta_{\phi},\bar\theta_{\phi}\}$ forms a local basis of $A^{1}(M)$, we can define a new almost complex structure $J_{\phi}$ on $M$, such that $\{\theta^{i}_{\phi}\}_{i=1}^{n}$ is a local basis of $A^{1,0}_{J_{\phi}}(M)$ and $\{\theta^{\bar i}_{\phi}\}_{i=1}^{n}$ is a local basis of $A^{0,1}_{J_{\phi}}(M)$. Denote the dual basis of $\{\theta^{i}_{\phi}\}_{i=1}^{n}$ and $\{\theta^{\bar i}_{\varphi}\}_{i=1}^{n}$ as $\{e_{\phi,i}\}_{i=1}^{n}$ and $\{e_{\phi,\bar i}\}_{i=1}^{n}$ respectively.

As a sum of the above discussion, we give the following theorem.
\begin{theorem}
Let $(M,J)$ be an almost complex manifold. Assume that $J_{\phi}$ is another almost complex structure on $M$ induced by a Beltrami differential $\phi\in A^{0,1}_{J}(M,T^{1,0}M)$. Let $\bigl\{\theta^{i}\bigr\}_{i=1}^{n}$ be a basis of $A^{1,0}_{J}(M)$, and $\{e_{i}\}_{i=1}^{n}$ its dual basis in $T^{1,0}_{J}M$. By definition $\bigl\{\theta^{i}_{\phi}=\theta^{i}+i_{\phi}\theta^{i}\bigr\}_{i=1}^{n}$ is a basis of $A^{1,0}_{J_{\phi}}(M)$. Then the dual basis of $\theta^{i}_{\phi}$ in $T^{1,0}_{J_{\phi}}M$ is expressed as
\begin{equation*}
e_{\phi,i}=(\Phi^{-1})_{i}^{j}e_{j}+(\Phi^{-1})_{i}^{\bar j}\bar e_{j},
\end{equation*}
or equivalently
\begin{displaymath}
\left(\begin{array}{c}
e_{\phi,1}\\
\vdots\\
e_{\phi,n}\\
\bar e_{\phi,1}\\
\vdots\\
\bar e_{\phi,n}
\end{array}\right)
=\Phi^{-T}\cdot
\left(\begin{array}{c}
e_{1}\\
\vdots\\
e_{n}\\
\bar e_{1}\\
\vdots\\
\bar e_{n}
\end{array}\right),
\quad\quad
\Phi^{T}\cdot\left(\begin{array}{c}
e_{\phi,1}\\
\vdots\\
e_{\phi,n}\\
\bar e_{\phi,1}\\
\vdots\\
\bar e_{\phi,n}
\end{array}\right)
=\left(\begin{array}{c}
e_{1}\\
\vdots\\
e_{n}\\
\bar e_{1}\\
\vdots\\
\bar e_{n}
\end{array}\right),
\end{displaymath}
where
\begin{displaymath}
\Phi=\left(\begin{array}{cc}
I^{\prime} & \phi\\
\bar\phi& I^{\prime\prime}
\end{array}\right)_{2n\times 2n},
\end{displaymath}
and
\begin{displaymath}
\Phi^{-1}=
\left(\begin{array}{cc}
(I^{\prime}-\phi\cdot\bar\phi)^{-1} & -\phi\cdot(I^{\prime\prime}-\bar\phi\cdot\phi)^{-1}\\
&\\
-\bar\phi\cdot(I^{\prime}-\phi\cdot\bar\phi)^{-1}& (I^{\prime\prime}-\bar\phi\cdot\phi)^{-1}
\end{array}\right)_{2n\times 2n}.
\end{displaymath}
Also we have
\begin{equation}\label{neq1}
\left(\begin{array}{c}
\theta^{1}\\
\vdots\\
\theta^{n}\\
\bar\theta^{1}\\
\vdots\\
\bar\theta^{n}
\end{array}\right)
=\left(\begin{array}{cc}
&\\
(I^{\prime}-\phi\cdot\bar\phi)^{-1} & -\phi\cdot(I^{\prime\prime}-\bar\phi\cdot\phi)^{-1}\\
&\\
&\\
-\bar\phi\cdot(I^{\prime}-\phi\cdot\bar\phi)^{-1}& (I^{\prime\prime}-\bar\phi\cdot\phi)^{-1}\\
&
\end{array}\right)
\left(\begin{array}{c}
\theta_{\phi}^{1}\\
\vdots\\
\theta_{\phi}^{n}\\
\bar\theta_{\phi}^{1}\\
\vdots\\
\bar\theta_{\phi}^{n}
\end{array}\right).
\end{equation}
\end{theorem}
\begin{remark}\label{remark1}
The matrix
\begin{equation*}
\Phi^{-1}=\left(\begin{array}{cc}
(I^{\prime}-\phi\cdot\bar\phi)^{-1} & -\phi\cdot(I^{\prime\prime}-\bar\phi\cdot\phi)^{-1}\\
&\\
-\bar\phi\cdot(I^{\prime}-\phi\cdot\bar\phi)^{-1}& (I^{\prime\prime}-\bar\phi\cdot\phi)^{-1}
\end{array}\right),
\end{equation*}
could be considered as a contraction operator, i.e.
\begin{equation*}
\Phi^{-1}=(\Phi^{-1})^{\alpha}_{\beta}\theta_{\phi}^{\beta}\otimes e_{\phi,\alpha},
\end{equation*}
where $\alpha,\beta\in\{1,\ldots,n,\bar 1,\ldots,\bar n\}$.
\end{remark}
Note that
\begin{align}
\Phi^{-1}=&\,(\Phi^{-1})^{\alpha}_{\beta}\theta_{\phi}^{\beta}\otimes e_{\phi,\alpha}\nonumber\\
=&\,(\Phi^{-1})^{\alpha}_{\beta}\Phi^{\beta}_{\lambda}(\Phi^{-1})^{\mu}_{\alpha}
\theta^{\lambda}\otimes e_{\mu}\nonumber\\
=&\,\delta^{\alpha}_{\lambda}(\Phi^{-1})^{\mu}_{\alpha}\theta^{\lambda}\otimes e_{\mu}\nonumber\\
=&\,(\Phi^{-1})^{\mu}_{\lambda}\theta^{\lambda}\otimes e_{\mu}.\label{eq53}
\end{align}
So when we say that $\Phi^{-1}$ is a contraction operator, we do not need to specify the basis we use between $\{\theta^{\alpha}_{\phi}\}$ and $\{\theta^{\alpha}\}$. Using similar methods, readers can check that similar properties hold for $\Phi$. The coefficients are
\begin{align*}
(\Phi^{-1})^{i}_{j}=&\,((I-\phi\cdot\bar\phi)^{-1})^{i}_{j},\\
(\Phi^{-1})^{i}_{\bar j}=&\,(-\phi\cdot(I^{\prime\prime}-\bar\phi\cdot\phi)^{-1})^{i}_{\bar j}
=(-(I^{\prime\prime}-\phi\cdot\bar\phi)^{-1}\cdot\phi)^{i}_{\bar j}
=\bigl(-(I^{\prime\prime}-\phi\cdot\bar\phi)^{-1}\bigr)^{i}_{k}\phi^{k}_{\bar j},\\
(\Phi^{-1})^{\bar i}_{j}=&\,(-\bar\phi\cdot(I-\phi\cdot\bar\phi)^{-1})^{\bar i}_{j}
=(-(I^{\prime\prime}-\bar\phi\cdot\phi)^{-1}\cdot\bar\phi)^{\bar i}_{j}
=(-(I^{\prime\prime}-\bar\phi\cdot\phi)^{-1})^{\bar i}_{\bar k}\bar\phi^{\bar k}_{j}
,\\
(\Phi^{-1})^{\bar i}_{\bar j}=&\,((I^{\prime\prime}-\bar\phi\cdot\phi)^{-1})^{\bar i}_{\bar j}.
\end{align*}
From the above arguments, one can see that $(I^{\prime}-\phi\cdot\bar\phi)^{-1}$, $-(I^{\prime}-\phi\cdot\bar\phi)^{-1}\cdot\phi$, $(I^{\prime\prime}-\bar\phi\cdot\phi)^{-1}\cdot\bar\phi$ and $(I^{\prime\prime}-\bar\phi\cdot\phi)^{-1}$ can also be seen as contractions.
By (\ref{neq1}) or (\ref{eq53}), we know that
\begin{equation*}
\Phi^{-1}\lrcorner\,\theta_{\phi}^{i}=(\Phi^{-1})^{i}_{\beta}\theta_{\phi}^{\beta}
=(\Phi^{-1})^{i}_{\beta}\Phi^{\beta}_{\alpha}\theta^{\alpha}=\theta^{i}.
\end{equation*}
Moreover, using the fact that contraction operators are point-wisely linear, we have
\begin{equation*}
\Phi^{-1}\lrcorner\,\theta^{i}=\Phi^{-1}\bigl((\Psi^{-1})^{i}_{\beta}\theta_{\phi}^{\beta}\bigr)
=(\Psi^{-1})^{i}_{\beta}\Phi^{-1}\bigl(\theta_{\phi}^{\beta}\bigr)
=(\Psi^{-1})^{i}_{\beta}\theta^{\beta}.
\end{equation*}
An interesting operator is $i_{(I^{\prime\prime}-\bar\phi\cdot\phi)^{-1}\cdot\bar\phi}\circ i_{(I^{\prime\prime}-\bar\phi\cdot\phi)}$. For any $\theta^{\bar i}\in A^{0,1}_{J}(M)$, we have
\begin{align}
&(I^{\prime\prime}-\bar\phi\cdot\phi)^{-1}\cdot\bar\phi\lrcorner\bigl((I^{\prime\prime}-\bar\phi\cdot\phi)\lrcorner\theta^{\bar i}\bigr)\nonumber\\
=&\,(I^{\prime\prime}-\bar\phi\cdot\phi)^{-1}\cdot\bar\phi\lrcorner\bigl((I-\bar\phi\cdot\phi)^{\bar i}_{\bar j}\theta^{\bar j}\bigr)\nonumber\\
=&\,(I^{\prime\prime}-\bar\phi\cdot\phi)^{\bar i}_{\bar j}(I^{\prime\prime}-\bar\phi\cdot\phi)^{-1}\cdot\bar\phi\lrcorner(\theta^{\bar j})\quad\textup{by the point-wisely linear property}\nonumber\\
=&\,(I^{\prime\prime}-\bar\phi\cdot\phi)^{\bar i}_{\bar j}((I^{\prime\prime}-\bar\phi\cdot\phi)^{-1}\cdot\bar\phi)^{\bar j}_{ k}\theta^{k}\nonumber\\
=&\,(I^{\prime\prime}-\bar\phi\cdot\phi)^{\bar i}_{\bar j}((I-\bar\phi\cdot\phi)^{-1})^{\bar j}_{ \bar\ell}\bar\phi^{\bar\ell}_{k}\theta^{k}\nonumber\\
=&\,\bar\phi^{\bar i}_{k}\theta^{k}\nonumber\\
=&\,\bar\phi\lrcorner\,\theta^{\bar i}.\nonumber
\end{align}
For the definition of ``$\lrcorner$", see (\ref{contraction1}) and (\ref{contraction3}). Note that here we have used the fact that $e^{i_{\phi}|i_{\bar\phi}}$ is a point-wisely linear map.

Let $(M,J,g)$ be an almost hermitian manifold, i.e., $J$ is compatible with $g$. If $J_{\phi}$ is also compatible with $g$, or equivalently if $J_{\phi}$ and $J$ belong to \textbf{the same compatible class}, then
\begin{align}
\bigl\langle \theta^{i}_{\phi},\theta^{\bar j}_{\phi}\bigr\rangle_{\mathbb{C}}
=&\,\bigl\langle \theta^{i}+\phi^{i}_{\bar k}\theta^{\bar k},
\theta^{\bar j}+\phi^{\bar j}_{\ell}\theta^{\ell}\bigr\rangle_{\mathbb{C}}\nonumber\\
=&\,\bigl\langle\theta^{i},\phi^{\bar j}_{\ell}\theta^{\ell}\bigr\rangle_{\mathbb{C}}
+\bigl\langle\phi^{i}_{\bar k}\theta^{\bar k},\theta^{\bar j}\bigr\rangle_{\mathbb{C}}\nonumber\\
=&\,\phi^{j}_{\bar\ell}\bigl\langle\theta^{i},\theta^{\ell}\bigr\rangle_{\mathbb{C}}
+\phi^{i}_{\bar k}\bigl\langle\theta^{\bar k},\theta^{\bar j}\bigr\rangle_{\mathbb{C}}\nonumber\\
=&\,\phi^{j}_{\bar\ell}\delta^{i\ell}+\phi^{i}_{\bar k}\delta^{kj}\nonumber\\
=&\,\phi^{j}_{\bar i}+\phi^{i}_{\bar j}\nonumber\\
=&\,0,\label{eq91}
\end{align}
where $\langle-,-\rangle_{\mathbb{C}}$ denotes the complex conjugate extension of $g$ over the complexified tangent bundle $T^{\mathbb{C}}M$. (\ref{eq91}) implies that $\phi^{j}_{\bar i}=-\phi^{i}_{\bar j}$, or equivalently
\begin{equation*}
\phi=-\phi^{T}.
\end{equation*}
Then we see that $\bar\phi=-\bar\phi^{T}=-\phi^{\ast}$ and thus
\begin{equation*}
\det(I^{\prime}-\phi\cdot\bar\phi)=\det(I^{\prime}+\phi\cdot\phi^{\ast}).
\end{equation*}
A matrix $A$ is hermitian if and only if there is a unitary $U$ and a real diagonal $\Lambda$ such that $A=U\Lambda U^{\ast}$. See for example \cite[Theorem 4.1.5, p.229]{horn1990matrix}.
Since $\phi\cdot\phi^{\ast}$ is hermitian, its eigenvalues are all real. Assume that $\phi\cdot\phi^{\ast}=U\Lambda U^{\ast}$, then
\begin{align*}
\det(I^{\prime}-\phi\cdot\bar\phi)
=&\,\det(I^{\prime}+\phi\cdot\phi^{\ast})\\
=&\,\det(I^{\prime}+U\Lambda U^{\ast})\\
=&\,\det U(I^{\prime}+\Lambda)U^{\ast}\\
=&\,\det (I^{\prime}+\Lambda).
\end{align*}
Now we give the following remark.
\begin{remark}
If $J$ and $J_{\phi}$ are both compatible with $g$, then $-1$ is not an eigenvalue of $\phi\cdot\phi^{\ast}$.
\end{remark}
For more detailed information of Beltrami differentials on complex manifolds, one may try \cite{kodaira1981complex,morrow2006complex}.
\section{Generalized Lie derivative}\label{sec1}
Let $M$ be a differentiable manifold. For $X,Y_{1},\ldots,Y_{r}\in\Gamma(M,TM)$, $\varphi\in A^{r+1}(M)$, the contraction operator $\lrcorner\,:A^{r+1}(M)\rightarrow A^{r}(M)$ is defined as
\begin{equation}\label{contraction3}
(X\lrcorner\,\varphi)(Y_{1},\ldots,Y_{r}):=\varphi(X,Y_{1},\ldots,Y_{r}).
\end{equation}
Usually, we also denote $X\lrcorner\,\varphi$ as $i_{X}\varphi$. For $\rho\in A^{p}(M,TM)$, the contraction operator can be extended to
\begin{equation*}
\rho\lrcorner\,:A^{r}(M)\rightarrow A^{r+p-1}(M).
\end{equation*}
Explicitly, for $\rho=\eta\otimes X$ with $\eta\in A^{p}(M)$,
\begin{equation}\label{contraction1}
\rho\lrcorner\,\varphi:=\eta\wedge(X\lrcorner\,\varphi).
\end{equation}
The contraction operator ``$\lrcorner$" is a (anti-)derivation of the exterior algebra $A^{\ast}(M)$, i.e.
\begin{equation*}
X\lrcorner(\varphi\wedge\psi)=X\lrcorner\varphi\wedge\psi+(-1)^{|\varphi|}\varphi\wedge X\lrcorner\psi,\quad
\varphi,\psi\in A^{\ast}(M).
\end{equation*}
Moreover, if $\tilde X:=\langle-,X\rangle\in T^{\ast}M$, the operator $X\lrcorner-$ is the adjoint map of $\tilde X\wedge-$, i.e.
\begin{equation*}
\langle X\lrcorner\,\varphi,\psi\rangle=\langle \varphi,\tilde X\wedge \psi\rangle,\quad \varphi,\psi\in A^{\ast}(M).
\end{equation*}
For a more general description of the derivation theory, one may consult an excellent book \cite{kolar1999natural}.
\begin{lemma}\label{commutator7}\cite[Lemma 3.1]{MR3302118}
For $\varphi\in A^{0,q}(M,T^{1,0}M)$, $\psi\in A^{0,s}(M,T^{1,0}M)$, we have
\begin{equation*}
i_{\varphi}\circ i_{\psi}=(-1)^{(q+1)(s+1)}i_{\psi}\circ i_{\varphi}.
\end{equation*}
\end{lemma}
Generally, for any $\varphi\in A^{k}(M,TM)$, the contraction is defined as
\begin{align}\label{contraction2}
&\varphi\lrcorner\,\omega(X_{1},\ldots,X_{|\varphi|+|\omega|+1})=\frac{1}{|\varphi|!(|\omega|-1)!}\bullet\\
&\sum_{\sigma\in\mathscr{S}(|\varphi|+|\omega|-1)}\mathrm{sign}\sigma\cdot
\omega(\varphi(X_{\sigma(1)},\ldots,X_{\sigma(|\varphi|)}),X_{\sigma(|\varphi|+1)},\ldots,X_{\sigma(|\varphi|+|\omega|-1)})
.\nonumber
\end{align}
Note that (\ref{contraction1}) and (\ref{contraction2}) are actually the same.

Let $M$ be a differentiable manifold, $E\rightarrow M$ be a vector bundle and $\nabla$ be a connection of $E$. The generalized definition of the ordinary Lie derivative is
\begin{equation}\label{eq301}
\mathcal{L}^{\nabla}_{\rho}:=[i_{\rho},\nabla]=i_{\rho}\circ\nabla-(-1)^{|i_{\rho}|\cdot|\nabla|}\nabla\circ i_{\rho}=i_{\rho}\circ\nabla+(-1)^{|\rho|}\nabla\circ i_{\rho},
\end{equation}
where $\rho\in A^{\ast}(M,TM)$.

Liu-Rao \cite{MR2904916} and Liu-Rao-Yang \cite{MR3302118} got some results of the generalized Lie derivative for holomorphic vector bundles over complex manifolds, with a hermitian metric and the Chern connection.
Following \cite[p.70, Lemma 8.6]{kolar1999natural}, we can prove the following identity. One can also find another proof in \cite[Lemma 3.9]{MR923350}.
\begin{lemma}\label{lemma7}
Assume that $\varphi\in A^{k}(M,TM)$ and $\psi\in A^{\ell+1}(M,TM)$. For any $\alpha\in A^{p,q}(M,E)$, we have
\begin{equation*}
[\mathcal{L}^{\nabla}_{\varphi},i_{\psi}]
=i_{[\varphi,\psi]}-(-1)^{|\varphi|(|\psi|-1)}\mathcal{L}^{\nabla}_{i_{\psi}\varphi}.
\end{equation*}
\end{lemma}
\begin{proof}
For simplicity and without loss of generality, we may choose $\varphi=\rho\otimes X$ and $\psi=\tau\otimes Y$. For any $\alpha\in A^{p,q}(M,E)$,
the operator $i_{\varphi}$ is defined as
\begin{equation}\label{eq3}
i_{\varphi}\alpha=\rho\wedge(X\lrcorner\alpha).
\end{equation}
Note that
\begin{equation}\label{eq1}
\nabla(i_{\varphi}\alpha)
=\nabla(\rho\wedge X\lrcorner\alpha)
=d\rho\wedge (X\lrcorner\alpha)+(-1)^{\rho}\rho\wedge\nabla(X\lrcorner\alpha),
\end{equation}
and
\begin{equation}\label{eq2}
i_{\varphi}\nabla\alpha=\rho\wedge(X\lrcorner\nabla\alpha).
\end{equation}
By (\ref{eq1}) and (\ref{eq2}), we get
\begin{align}\label{liederivative2}
\mathcal{L}^{\nabla}_{\varphi}\alpha
=&\,i_{\varphi}\nabla\alpha+(-1)^{|\rho|}\nabla(i_{\varphi}\alpha)\nonumber\\
=&\,\rho\wedge X\lrcorner\nabla\alpha+(-1)^{\rho}(d\rho\wedge X\lrcorner\alpha+(-1)^{\rho}\rho\wedge\nabla(X\lrcorner\alpha))\nonumber\\
=&\,\rho\wedge X\lrcorner\nabla\alpha+(-1)^{\rho}d\rho\wedge X\lrcorner\alpha+\rho\wedge\nabla(X\lrcorner\alpha)\nonumber\\
=&\,\rho\wedge(X\lrcorner\nabla\alpha+\nabla(X\lrcorner\alpha))+(-1)^{\rho}d\rho\wedge X\lrcorner\alpha\nonumber\\
=&\,\rho\wedge\mathcal{L}^{\nabla}_{X}\alpha+(-1)^{\rho}d\rho\wedge X\lrcorner\alpha.
\end{align}
Using (\ref{eq3}) and (\ref{liederivative2}) we get
\begin{align}\label{eq4}
&\mathcal{L}^{\nabla}_{\varphi}(i_{\psi}\alpha)\nonumber\\
=&\,\mathcal{L}^{\nabla}_{\varphi}(\tau\wedge(Y\lrcorner\alpha))\nonumber\\
=&\,\mathcal{L}^{d}_{\varphi}\tau\wedge(Y\lrcorner\alpha)
+(-1)^{\tau\rho}\tau\wedge\mathcal{L}^{\nabla}_{\varphi}(Y\lrcorner\alpha)\nonumber\\
=&\,(\rho\wedge\mathcal{L}^{d}_{X}\tau+(-1)^{\rho}d\rho\wedge X\lrcorner\tau)\wedge(Y\lrcorner\alpha)
+(-1)^{\tau\rho}\tau\wedge\bigl(\rho\wedge\mathcal{L}^{\nabla}_{X}(Y\lrcorner\alpha)
+(-1)^{\rho}d\rho\wedge X\lrcorner(Y\lrcorner\alpha)\bigr)\nonumber\\
=&\,\rho\wedge\mathcal{L}^{d}_{X}\tau\wedge(Y\lrcorner\alpha)
+(-1)^{\rho}d\rho\wedge(X\lrcorner\tau)\wedge(Y\lrcorner\alpha)
+(-1)^{\tau\rho}\tau\wedge\rho\wedge\mathcal{L}^{\nabla}_{X}(Y\lrcorner\alpha)
+(-1)^{\rho+\tau\rho}\tau\wedge d\rho\wedge(X\lrcorner Y\lrcorner\alpha)\nonumber\\
=&\,\rho\wedge\mathcal{L}^{d}_{X}\tau\wedge(Y\lrcorner\alpha)
+(-1)^{\rho}d\rho\wedge(X\lrcorner\tau)\wedge(Y\lrcorner\alpha)
+\rho\wedge\tau\wedge\mathcal{L}^{\nabla}_{X}(Y\lrcorner\alpha)
+(-1)^{\rho+\tau}d\rho\wedge\tau\wedge(X\lrcorner Y\lrcorner\alpha),
\end{align}
and
\begin{align}\label{eq5}
&i_{\psi}(\mathcal{L}^{\nabla}_{\varphi}\alpha)\nonumber\\
=&\tau\wedge Y\lrcorner(\mathcal{L}^{\nabla}_{\varphi}\alpha)\nonumber\\
=&\tau\wedge Y\lrcorner(\rho\wedge\mathcal{L}^{\nabla}_{X}\alpha+(-1)^{\rho}d\rho\wedge X\lrcorner\alpha)\nonumber\\
=&\tau\wedge Y\lrcorner(\rho\wedge\mathcal{L}^{\nabla}_{X}\alpha)+(-1)^{\rho}\tau\wedge Y\lrcorner(d\rho\wedge X\lrcorner\alpha)\nonumber\\
=&\tau\wedge(Y\lrcorner\rho)\wedge\mathcal{L}^{\nabla}_{X}\alpha
+(-1)^{\rho}\tau\wedge\rho\wedge(Y\lrcorner\mathcal{L}^{\nabla}_{X}\alpha)
+(-1)^{\rho}\tau\wedge(Y\lrcorner d\rho)\wedge(X\lrcorner\alpha)
-\tau\wedge d\rho\wedge(Y\lrcorner X\lrcorner\alpha).
\end{align}
Using (\ref{eq4}) and (\ref{eq5}), we get the commutator of $\mathcal{L}^{\nabla}_{\varphi}$ and $i_{\psi}$,
\begin{small}
\begin{align}\label{eq6}
[\mathcal{L}^{\nabla}_{\varphi},i_{\psi}]\alpha
=&\mathcal{L}^{\nabla}_{\varphi}(i_{\psi}\alpha)
-(-1)^{|\mathcal{L}^{\nabla}_{\varphi}|\cdot|i_{\psi}|}i_{\psi}(\mathcal{L}^{\nabla}_{\varphi}\alpha)\nonumber\\
=&\mathcal{L}^{\nabla}_{\varphi}(i_{\psi}\alpha)
-(-1)^{|\varphi|\cdot(|\psi|-1)}i_{\psi}(\mathcal{L}^{\nabla}_{\varphi}\alpha)\nonumber\\
=&\rho\wedge\mathcal{L}^{d}_{X}\tau\wedge(Y\lrcorner\alpha)
+(-1)^{\rho}d\rho\wedge(X\lrcorner\tau)\wedge(Y\lrcorner\alpha)
+(-1)^{\tau\rho}\tau\wedge\rho\wedge\mathcal{L}^{\nabla}_{X}(Y\lrcorner\alpha)
+(-1)^{\rho+\tau\rho}\tau\wedge d\rho\wedge(X\lrcorner Y\lrcorner\alpha)\nonumber\\
&-(-1)^{\rho(\tau-1)}\tau\wedge(Y\lrcorner\rho)\wedge\mathcal{L}^{\nabla}_{X}\alpha
-(-1)^{\rho(\tau-1)}(-1)^{\rho}\tau\wedge\rho\wedge(Y\lrcorner\mathcal{L}^{\nabla}_{X}\alpha)\nonumber\\
&-(-1)^{\rho(\tau-1)}(-1)^{\rho}\tau\wedge(Y\lrcorner d\rho)\wedge(X\lrcorner\alpha)
+(-1)^{\rho(\tau-1)}\tau\wedge d\rho\wedge(Y\lrcorner X\lrcorner\alpha)\nonumber\\
=&\rho\wedge\mathcal{L}^{d}_{X}\tau\wedge(Y\lrcorner\alpha)
+(-1)^{\rho}d\rho\wedge(X\lrcorner\tau)\wedge(Y\lrcorner\alpha)
+\rho\wedge\tau\wedge\mathcal{L}^{\nabla}_{X}(Y\lrcorner\alpha)
+(-1)^{\rho+\tau}d\rho\wedge\tau\wedge(X\lrcorner Y\lrcorner\alpha)\nonumber\\
&-(-1)^{\rho+\tau}(Y\lrcorner\rho)\wedge\tau\wedge\mathcal{L}^{\nabla}_{X}\alpha
-\rho\wedge\tau\wedge(Y\lrcorner\mathcal{L}^{\nabla}_{X}\alpha)
-(Y\lrcorner d\rho)\wedge\tau\wedge(X\lrcorner\alpha)
+(-1)^{\rho+\tau}d\rho\wedge\tau\wedge(Y\lrcorner X\lrcorner\alpha)\nonumber\\
=&\rho\wedge\mathcal{L}^{d}_{X}\tau\wedge(Y\lrcorner\alpha)
+(-1)^{\rho}d\rho\wedge(X\lrcorner\tau)\wedge(Y\lrcorner\alpha)
-(-1)^{\rho+\tau}(Y\lrcorner\rho)\wedge\tau\wedge\mathcal{L}^{\nabla}_{X}\alpha
-(Y\lrcorner d\rho)\wedge\tau\wedge(X\lrcorner\alpha)\nonumber\\
&+\rho\wedge\tau\wedge\mathcal{L}^{\nabla}_{X}(Y\lrcorner\alpha)
-\rho\wedge\tau\wedge(Y\lrcorner\mathcal{L}^{\nabla}_{X}\alpha)\nonumber\\
=&\rho\wedge\mathcal{L}_{X}\tau\wedge(Y\lrcorner\alpha)
+(-1)^{\rho}d\rho\wedge(X\lrcorner\tau)\wedge(Y\lrcorner\alpha)
-(-1)^{\rho+\tau}(Y\lrcorner\rho)\wedge\tau\wedge\mathcal{L}^{\nabla}_{X}\alpha\nonumber\\
&-(Y\lrcorner d\rho)\wedge\tau\wedge(X\lrcorner\alpha)
+\rho\wedge\tau\wedge[X,Y]\lrcorner\alpha.
\end{align}
\end{small}
The (Fr\"olicher-Nijenhuis) bracket of $\varphi=\rho\otimes X\in A^{\ast}(M,TM)$ and $\psi=\tau\otimes Y\in A^{\ast}(M,TM)$ is defined as (\cite[p.70, Equation (6)]{kolar1999natural})
\begin{align}\label{fnbracket}
&[\varphi,\psi]\nonumber\\
=&[\rho\otimes X,\tau\otimes Y]\nonumber\\
=&\rho\wedge\tau\otimes[X,Y]+\rho\wedge\mathcal{L}_{X}\tau\otimes Y-\mathcal{L}_{Y}\rho\wedge\tau\otimes X
+(-1)^{|\rho|}d\rho\wedge(X\lrcorner\tau)\otimes Y+(-1)^{|\rho|}Y\lrcorner\rho\wedge d\tau\otimes X.
\end{align}
As a contraction operator, $i_{[\varphi,\psi]}$ is
\begin{small}
\begin{align}\label{eq7}
&i_{[\varphi,\psi]}\alpha\nonumber\\
=&\,(\rho\wedge\tau\otimes[X,Y]+\rho\wedge\mathcal{L}_{X}\tau\otimes Y-\mathcal{L}_{Y}\rho\wedge\tau\otimes X
+(-1)^{|\rho|}d\rho\wedge(X\lrcorner\tau)\otimes Y
+(-1)^{|\rho|}Y\lrcorner\rho\wedge d\tau\otimes X)\lrcorner\alpha\nonumber\\
=&\,\rho\wedge\tau\wedge[X,Y]\lrcorner\alpha
+\rho\wedge\mathcal{L}_{X}\tau\wedge Y\lrcorner\alpha
-\mathcal{L}_{Y}\rho\wedge\tau\wedge X\lrcorner\alpha
+(-1)^{|\rho|}d\rho\wedge(X\lrcorner\tau)\wedge Y\lrcorner\alpha
+(-1)^{|\rho|}Y\lrcorner\rho\wedge d\tau\wedge X\lrcorner\alpha.
\end{align}
\end{small}
Again, using (\ref{liederivative2}), we have
\begin{align}\label{eq8}
\mathcal{L}^{\nabla}_{i_{\psi}\varphi}\alpha
=&\,[i_{\psi}\varphi,\nabla]\alpha=[\tau\wedge (Y\lrcorner\rho)\otimes X,\nabla]\alpha\nonumber\\
=&\,[\tau\wedge(Y\lrcorner\rho)\otimes X\lrcorner,\nabla]\alpha\nonumber\\
=&\,\tau\wedge(Y\lrcorner\rho)\wedge (X\lrcorner\nabla\alpha)+(-1)^{\tau+\rho-1}\nabla(\tau\wedge(Y\lrcorner\rho)\wedge X\lrcorner\alpha)\nonumber\\
=&\,\tau\wedge(Y\lrcorner\rho)\wedge (X\lrcorner\nabla\alpha)+(-1)^{\tau+\rho-1}d(\tau\wedge(Y\lrcorner\rho))\wedge( X\lrcorner\alpha)
+\tau\wedge(Y\lrcorner\rho)\wedge\nabla(X\lrcorner\alpha)\nonumber\\
=&\,\tau\wedge(Y\lrcorner\rho)\wedge (X\lrcorner\nabla\alpha)
+(-1)^{\tau+\rho-1}d\tau\wedge(Y\lrcorner\rho)\wedge( X\lrcorner\alpha)\nonumber\\
&+(-1)^{\rho-1}\tau\wedge d(Y\lrcorner\rho)\wedge( X\lrcorner\alpha)
+\tau\wedge(Y\lrcorner\rho)\wedge\nabla(X\lrcorner\alpha)\nonumber\\
=&\,\tau\wedge(Y\lrcorner\rho)\wedge\mathcal{L}^{\nabla}_{X}\alpha
+(-1)^{\tau+\rho-1}d\tau\wedge(Y\lrcorner\rho)\wedge( X\lrcorner\alpha)
+(-1)^{\rho-1}\tau\wedge d(Y\lrcorner\rho)\wedge( X\lrcorner\alpha).
\end{align}
Comparing the items of (\ref{eq6}) (\ref{eq7}) and (\ref{eq8}), we have proved the lemma.
\end{proof}
Noting (\ref{fnbracket}), we have the following lemma.
\begin{lemma}\label{lemma17}
If $\varphi=\rho\otimes X,\psi=\tau\otimes Y\in A^{0,1}(M,T^{1,0}M)$, then
\begin{equation*}
[\varphi,\psi]=
\rho\wedge\tau\otimes[X,Y]+\rho\wedge\mathcal{L}_{X}\tau\otimes Y-\mathcal{L}_{Y}\rho\wedge\tau\otimes X,
\end{equation*}
and
\begin{equation*}%\label{decomposition}
[\varphi,\psi]\in A^{0,2}(M,T^{1,0}M)\oplus A^{1,1}(M,T^{1,0}M)\oplus A^{0,2}(M,T^{0,1}M).
\end{equation*}
\end{lemma}
By comparing types, we have a decomposition of $[\varphi,\psi]$ as
\begin{align*}
[\varphi,\psi]
=&\,\rho\wedge\tau\otimes[X,Y]
+\rho\wedge(X\lrcorner\mu\tau+X\lrcorner\partial\tau)\otimes Y
-(Y\lrcorner\mu\rho+Y\lrcorner\partial\rho)\wedge\tau\otimes X&\\
=&\,\rho\wedge\tau\otimes[X,Y]^{1,0}
+\rho\wedge(X\lrcorner\partial\tau)\otimes Y
-(Y\lrcorner\partial\rho)\wedge\tau\otimes X&\in A^{0,2}_{J}(M,T^{1,0}M)\\
&+\rho\wedge(X\lrcorner\mu\tau)\otimes Y
-(Y\lrcorner\mu\rho)\wedge\tau\otimes X &\in A^{1,1}_{J}(M,T^{1,0}M)\\
&+\rho\wedge\tau\otimes[X,Y]^{0,1}.&\in A^{0,2}_{J}(M,T^{0,1}M)
\end{align*}
We denote the decomposition of $[\varphi,\psi]$ as
\begin{align*}
\mathscr{A}(\varphi,\psi)&=\rho\wedge\tau\otimes[X,Y]^{1,0}+\rho\wedge(X\lrcorner\partial\tau)\otimes Y-(Y\lrcorner\partial\rho)\wedge\tau\otimes X,\\
\mathscr{B}(\varphi,\psi)&=\rho\wedge(X\lrcorner\mu\tau)\otimes Y-(Y\lrcorner\mu\rho)\wedge\tau\otimes X,\\
\mathscr{C}(\varphi,\psi)&=\rho\wedge\tau\otimes[X,Y]^{0,1}.
\end{align*}
Without loss of generality, from now on, we will denote the decomposition of $[\varphi,\psi]$ as $\mathscr{A}(\varphi,\psi)$, $\mathscr{B}(\varphi,\psi)$ and $\mathscr{C}(\varphi,\psi)$ for general $\varphi,\psi\in A^{0,1}_{J}(M,T^{1,0}M)$.

As a special case of Lemma \ref{lemma7}, we get the following corollary.
\begin{corollary}\label{corollary2}
If $\varphi,\psi\in A^{0,1}_{J}(M,T^{1,0}M)$, then
\begin{equation*}
[\mathcal{L}^{\nabla}_{\varphi},i_{\psi}]=i_{[\varphi,\psi]}.
\end{equation*}
\end{corollary}
Corollary \ref{corollary2} implies that for $\varphi,\psi\in A^{0,1}_{J}(M,T^{1,0}M)$,
\begin{align}\label{eq52}
i_{[\varphi,\psi]}\alpha
=&\,\mathcal{L}^{\nabla}_{\varphi}(i_{\psi}\alpha)-i_{\psi}\mathcal{L}^{\nabla}_{\varphi}\alpha\nonumber\\
=&\,[i_{\varphi},\nabla]\circ i_{\psi}\alpha-i_{\psi}\circ[i_{\varphi},\nabla]\alpha\nonumber\\
=&\,\varphi\lrcorner\nabla(\psi\lrcorner\alpha)-\nabla(\varphi\lrcorner\psi\lrcorner\alpha)
-\psi\lrcorner\varphi\lrcorner\nabla\alpha+\psi\lrcorner\nabla(\varphi\lrcorner\alpha).
\end{align}
Equivalently in another version of (\ref{eq52}) we have
\begin{equation*}%\label{eq21}
[\nabla,i_{\varphi}]\circ i_{\psi}=i_{\psi}\circ[\nabla,i_{\varphi}]-i_{[\varphi,\psi]},
\end{equation*}
i.e. we get
\begin{equation}\label{eq10}
[\varphi,\psi]\lrcorner\,\alpha
=\varphi\lrcorner\,\nabla(\psi\lrcorner\,\alpha)-\nabla(\varphi\lrcorner\,\psi\lrcorner\,\alpha)
-\psi\lrcorner\,\varphi\lrcorner\,\nabla\alpha+\psi\lrcorner\,\nabla(\varphi\lrcorner\,\alpha).
\end{equation}
(\ref{eq10}) can be seen as a generalization of \cite[(3.3), (3.4)]{MR3302118} from the complex case to the almost complex case. Also see \cite[Corollary 4.5, 4.6]{MR2904916}.
A connection $\nabla$ on an almost complex manifold has a decomposition
\begin{equation*}
\nabla=\mu+\nabla^{1,0}+\nabla^{0,1}+\bar\mu.
\end{equation*}
By comparing types, we get the following lemma.
\begin{lemma}
\begin{align}
i_{\mathscr{A}(\varphi,\psi)}&=i^{(-1,2)}_{[\varphi,\psi]}
=-\nabla^{1,0}\circ i_{\varphi} i_{\psi}
-i_{\psi}i_{\varphi}\circ\nabla^{1,0}
+i_{\psi}\circ\nabla^{1,0}\circ i_{\varphi}
+i_{\varphi}\circ\nabla^{1,0}\circ i_{\psi}.\label{neweq10}\\
i_{(\mathscr{B}(\varphi,\psi)+\mathscr{C}(\varphi,\psi))}&=i^{(0,1)}_{[\varphi,\psi]}
=-\mu\circ i_{\varphi}i_{\psi}
-i_{\psi}i_{\varphi}\circ\mu
+i_{\psi}\circ\mu\circ i_{\varphi}
+i_{\varphi}\circ\mu\circ i_{\psi}.\label{neweq11}\\
0
&=-\nabla^{0,1}\circ i_{\varphi} i_{\psi}
-i_{\psi} i_{\varphi}\circ\nabla^{0,1}
+i_{\psi}\circ\nabla^{0,1}\circ i_{\varphi}
+i_{\varphi}\circ\nabla^{0,1}\circ i_{\psi}.\label{eq11}\\
0
&=-\bar\mu\circ i_{\varphi} i_{\psi}
-i_{\psi}i_{\varphi}\circ\bar\mu
+i_{\psi}\circ\bar\mu\circ i_{\varphi}
+i_{\varphi}\circ\bar\mu\circ i_{\psi}.\nonumber
\end{align}
\end{lemma}
At the same time, we get the following lemma.
\begin{lemma}
If $\varphi,\psi\in A^{0,1}_{J}(M,T^{1,0}_{M})$, then
\begin{align*}
i_{[\varphi,\psi]}=
&-\nabla^{1,0}\circ i_{\varphi} i_{\psi}
-i_{\psi}i_{\varphi}\circ\nabla^{1,0}
+i_{\psi}\circ\nabla^{1,0}\circ i_{\varphi}
+i_{\varphi}\circ\nabla^{1,0}\circ i_{\psi}\\
&-\mu\circ i_{\varphi}i_{\psi}
-i_{\psi}i_{\varphi}\circ\mu
+i_{\psi}\circ\mu\circ i_{\varphi}
+i_{\varphi}\circ\mu\circ i_{\psi}.
\end{align*}
\end{lemma}
The Cauchy-Riemannian operator $\bar\partial$ defined in \cite[(2.7.2)]{gauduchon1997hermitian} is
\begin{equation*}
\bar\partial_{X}Y:=[X,Y]^{1,0},\quad X\in T^{0,1}M, Y\in T^{1,0}M.
\end{equation*}
Note that by definition, under local basis $\{e_{i},e_{\bar i}\}_{i=1}^{n}$ we have
\begin{equation*}
\bar\partial_{e_{\bar j}}e_{i}=[e_{\bar j},e_{i}]^{1,0}=([e_{\bar j},e_{i}]^{1,0})^{k}e_{k}
\end{equation*}
and equivalently
\begin{equation}\label{neweq1}
\bar\partial e_{i}=([e_{\bar j},e_{i}]^{1,0})^{k}\theta^{\bar j}\otimes e_{k}
\end{equation}
For the dual basis $\theta^{i}$ and the $(0,1)$-part of the ordinary exterior differential $d$ we have
\begin{align*}
\bar\partial\theta^{k}(e_{i},e_{\bar j})
=&\,d\theta^{k}(e_{i},e_{\bar j})\\
=&\,e_{i}(\theta^{k}(e_{\bar j}))-e_{\bar j}(\theta^{k}(e_{i}))-\theta^{k}([e_{i},e_{\bar j}])\\
=&\,\theta^{k}([e_{\bar j},e_{i}]^{1,0})\\
=&\,([e_{\bar j},e_{i}]^{1,0})^{k},
\end{align*}
and thus
\begin{equation}\label{neweq2}
\bar\partial\theta^{k}=([e_{\bar j},e_{i}]^{1,0})^{k}\theta^{i}\wedge\theta^{\bar j}.
\end{equation}
Considering (\ref{neweq1}) and (\ref{neweq2}), we see the Cauchy-Riemannian operator and the $(0,1)$-part of the ordinary exterior differential as the same operator. We can naturally extend the operator $\bar\partial$ to $A^{0,p}(M,T^{1,0}M)$ as
\begin{equation*}
\bar\partial(\rho\otimes Y):=\bar\partial\rho\otimes Y+(-1)^{p}\rho\wedge\bar\partial Y,
\end{equation*}
where $\rho_{1}\wedge(\rho_{2}\otimes Y):=(\rho_{1}\wedge\rho_{2})\otimes Y$ for $\rho_{1},\rho_{2}\in A^{0,\ast}(M)$. The extended $\bar\partial$ could be the $(0,1)$-part of some connection on almost complex manifolds, for example the (almost) Chern connection. (In \cite{gauduchon1997hermitian}, this connection is called the second canonical connection. In \cite{MR2018338}, this connection is called the canonical almost hermitian connection. Here we may call it the \emph{almost} Chern connection to emphasise that it is the Chern connection on an \emph{almost} complex manifold.) If we choose $\nabla^{0,1}=\bar\partial$, then (\ref{eq11}) becomes
\begin{equation}\label{eq12}
0=-\bar\partial(\varphi\lrcorner\,\psi\lrcorner\,\alpha)
-\psi\lrcorner\,\varphi\lrcorner\,\bar\partial\alpha
+\psi\lrcorner\,\bar\partial(\varphi\lrcorner\,\alpha)+\varphi\lrcorner\,\bar\partial(\psi\lrcorner\,\alpha).
\end{equation}
(\ref{eq12}) is a generalization of \cite[(3.4)]{MR3302118}.

As a special case, if $\Omega\in A^{n,0}_{J}(M)$, we get a generalization of the so-called Tian-Todorov lemma (see \cite{tian1987smoothness,todorov1989weil}) in the almost complex case.
\begin{lemma}
Let $\varphi,\psi\in A^{0,1}_{J}(M,T^{1,0}M)$, $\Omega\in A^{n,0}_{J}(M)$ and $\mu\in A^{2,0}_{J}(M,T^{0,1}M)$. Then we have
\begin{align*}
[\varphi,\psi]\lrcorner\,\Omega
=-\mu(\varphi\lrcorner\,\psi\lrcorner\,\Omega)
+\varphi\lrcorner\, \partial(\psi\lrcorner\,\Omega)-\partial(\varphi\lrcorner\,\psi\lrcorner\,\Omega)
+\psi\lrcorner\, \partial(\varphi\lrcorner\,\Omega).
\end{align*}
\end{lemma}
\begin{proof}
A direct consequence of (\ref{neweq10}) and (\ref{neweq11}).
\end{proof}
By Lemma \ref{lemma7}, $[\varphi,\psi]\in A^{0,2}(M,T^{0,1})\oplus A^{0,2}(M,T^{1,0})\oplus A^{1,1}(M,T^{1,0})$. We know that $[\varphi,\psi]\lrcorner\,\Omega=[\varphi,\psi]^{1,0}\lrcorner\,\Omega$. And we get the following lemma.
\begin{lemma}
Let $\Omega\in A^{n,0}(M)$, then
\begin{align*}
\mathscr{C}(\varphi,\psi)\lrcorner\,\Omega&=0,\\
\mathscr{A}(\varphi,\psi)\lrcorner\,\Omega&=\varphi\lrcorner\, \partial(\psi\lrcorner\,\Omega)-\partial(\varphi\lrcorner\,\psi\lrcorner\,\Omega)
+\psi\lrcorner\,\partial(\varphi\lrcorner\,\Omega),\\
\mathscr{B}(\varphi,\psi)\lrcorner\,\Omega&=-\mu(\varphi\lrcorner\,\psi\lrcorner\,\Omega).
\end{align*}
\end{lemma}
\section{Exponential operator and extension formula}\label{sec2}
\subsection{Exponential operator and extension formula}
By Corollary \ref{corollary2}, for $\phi\in A^{0,1}(M,T^{1,0}M)$, we have
\begin{align}
[\nabla,i_{\phi}]i_{\phi}\alpha
=&\,-\mathcal{L}^{\nabla}_{\phi}i_{\phi}\alpha\nonumber\\
=&\,-i_{\phi}\mathcal{L}^{\nabla}_{\phi}\alpha-i_{[\phi,\phi]}\alpha\nonumber\\
=&\,i_{\phi}[\nabla,i_{\phi}]\alpha-i_{[\phi,\phi]}\alpha.\label{eq50}\\
[\nabla,i^{2}_{\phi}]\alpha
=&\,\nabla(i_{\phi}i_{\phi}\alpha)-i_{\phi}i_{\phi}\nabla\alpha\nonumber\\
=&\,\nabla\circ i_{\phi}(i_{\phi}\alpha)+i_{\phi}([\nabla,i_{\phi}]\alpha-\nabla(i_{\phi}\alpha))\nonumber\\
=&\,[\nabla,i_{\phi}](i_{\phi}\alpha)+i_{\phi}[\nabla,i_{\phi}]\alpha\nonumber\\
=&\,i_{\phi}[\nabla,i_{\phi}]\alpha-i_{[\phi,\phi]}\alpha+i_{\phi}[\nabla,i_{\phi}]\alpha\nonumber\\
=&\,2i_{\phi}[\nabla,i_{\phi}]\alpha-i_{[\phi,\phi]}\alpha.\label{eq48}\\
[\nabla,i^{3}_{\phi}]\alpha
=&\,\nabla\circ i^{3}_{\phi}\alpha-i^{3}_{\phi}\nabla\alpha\nonumber\\
=&\,\nabla\circ i^{2}_{\phi}(i_{\phi}\alpha)+i^{2}_{\phi}([\nabla,i_{\phi}]\alpha-\nabla i_{\phi}\alpha)\nonumber\\
=&\,[\nabla,i^{2}_{\phi}](i_{\phi}\alpha)+i^{2}_{\phi}\circ [\nabla,i_{\phi}]\alpha\nonumber\\
=&\,2i_{\phi}[\nabla,i_{\phi}]i_{\phi}\alpha-i_{[\phi,\phi]}i_{\phi}\alpha+i^{2}_{\phi}[\nabla,i_{\phi}]\alpha\nonumber\\
=&\,2i_{\phi}i_{\phi}[\nabla,i_{\phi}]\alpha-2i_{\phi}i_{[\phi,\phi]}\alpha
-i_{[\phi,\phi]}i_{\phi}\alpha+i^{2}_{\phi}[\nabla,i_{\phi}]\alpha\nonumber\\
=&\,3i^{2}_{\phi}[\nabla,i_{\phi}]\alpha-2i_{\phi}i_{[\phi,\phi]}\alpha
-i_{[\phi,\phi]}i_{\phi}\alpha.\label{eq44}
\end{align}
Now we need the relation of $i_{\phi}i_{[\phi,\phi]}$ and $i_{[\phi,\phi]}i_{\phi}$.

Consider a more general case. Let $\varphi,\psi,\xi\in A^{0,1}(M,T^{1,0}M)$. By Lemma \ref{lemma7}, we know that $[\varphi,\psi]$ has three components. Now we will discuss according to types.

(1) For $A=A_{\bar i\bar j}^{k}\bar\theta^{i}\wedge\bar\theta^{j}\otimes e_{k}$, $\xi=\xi^{q}_{\bar p}\bar\theta^{p}\otimes e_{q}$, we have
\begin{align*}
i_{A}i_{\xi}\alpha
=&\,A_{\bar i\bar j}^{k}\bar\theta^{i}\wedge\bar\theta^{j}\wedge e_{k}\lrcorner\,(\xi^{q}_{\bar p}\bar\theta^{p}\wedge e_{q}\lrcorner\alpha)\\
=&\,-A_{\bar i\bar j}^{k}\bar\theta^{i}\wedge\bar\theta^{j}\wedge\xi^{q}_{\bar p}\bar\theta^{p}\wedge (e_{k}\lrcorner e_{q}\lrcorner\alpha)\\
=&\,-A_{\bar i\bar j}^{k}\xi^{q}_{\bar p}\bar\theta^{i}\wedge\bar\theta^{j}\wedge\bar\theta^{p}\wedge (e_{k}\lrcorner e_{q}\lrcorner\alpha),\\
i_{\xi}i_{A}\alpha
=&\,\xi^{q}_{\bar p}\bar\theta^{p}\wedge e_{q}\lrcorner(A_{\bar i\bar j}^{k}\bar\theta^{i}\wedge\bar\theta^{j}\wedge e_{k}\lrcorner\alpha)\\
=&\,\xi^{q}_{\bar p}\bar\theta^{p}\wedge A_{\bar i\bar j}^{k}\bar\theta^{i}\wedge\bar\theta^{j}\wedge (e_{q}\lrcorner e_{k}\lrcorner\alpha)\\
=&\,-A_{\bar i\bar j}^{k}\xi^{q}_{\bar p}\bar\theta^{i}\wedge\bar\theta^{j}\wedge\bar\theta^{p}\wedge ( e_{k}\lrcorner e_{q}\lrcorner\alpha).
\end{align*}
Thus
\begin{equation*}
i_{A}i_{\xi}=i_{\xi}i_{A},
\end{equation*}
for $A\in A^{0,2}(M,T^{1,0}M)$ and $\xi\in A^{0,1}(M,T^{1,0}M)$.

(2) For $B=B_{i\bar j}^{k}\theta^{i}\wedge\bar\theta^{j}\otimes e_{k}$, we have
\begin{align*}
i_{B}i_{\xi}\alpha
=&\,B_{i\bar j}^{k}\theta^{i}\wedge\bar\theta^{j}\wedge e_{k}\lrcorner\,(\xi^{q}_{\bar p}\bar\theta^{p}\wedge e_{q}\lrcorner\alpha)\\
=&\,-B_{i\bar j}^{k}\xi^{q}_{\bar p}\theta^{i}\wedge\bar\theta^{j}\wedge\bar\theta^{p}\wedge (e_{k}\lrcorner e_{q}\lrcorner\alpha),\\
i_{\xi}i_{B}\alpha
=&\,\xi^{q}_{\bar p}\bar\theta^{p}\wedge e_{q}\lrcorner\,(B_{i\bar j}^{k}\theta^{i}\wedge\bar\theta^{j}\wedge e_{k}\lrcorner\alpha)\\
=&\,B_{i\bar j}^{k}\xi^{q}_{\bar p}\bar\theta^{p}\wedge \delta^{i}_{q}\bar\theta^{j}\wedge( e_{k}\lrcorner\alpha)
-B_{i\bar j}^{k}\xi^{q}_{\bar p}\theta^{i}\wedge\bar\theta^{j}\wedge\bar\theta^{p}\wedge ( e_{k}\lrcorner e_{q}\lrcorner\alpha)\\
=&\,B_{i\bar j}^{k}\xi^{i}_{\bar p}\bar\theta^{p}\wedge\bar\theta^{j}\wedge( e_{k}\lrcorner\alpha)
+i_{B}i_{\xi}\alpha\\
=&\,i_{i_{\xi}B}\alpha+i_{B}i_{\xi}\alpha.
\end{align*}
Thus
\begin{equation*}
i_{\xi}i_{B}\alpha-i_{B}i_{\xi}\alpha=i_{i_{\xi}B}\alpha.
\end{equation*}

(3) For $C=C_{\bar i\bar j}^{\bar k}\bar\theta^{i}\wedge\bar\theta^{j}\otimes \bar e_{k}$, we have
\begin{align*}
i_{\xi}i_{C}\alpha
=&\,\xi^{q}_{\bar p}\bar\theta^{p}\wedge e_{q}\lrcorner\,(C_{\bar i\bar j}^{\bar k}\bar\theta^{i}\wedge\bar\theta^{j}\wedge \bar e_{k}\lrcorner\alpha)\\
=&\,C_{\bar i\bar j}^{\bar k}\xi^{q}_{\bar p}\bar\theta^{p}\wedge\bar\theta^{i}\wedge\bar\theta^{j}\wedge (e_{q}\lrcorner\bar e_{k}\lrcorner\alpha)\\
=&\,-C_{\bar i\bar j}^{\bar k}\xi^{q}_{\bar p}\bar\theta^{i}\wedge\bar\theta^{j}\wedge\bar\theta^{p}\wedge (\bar e_{k}\lrcorner e_{q}\lrcorner\alpha),\\
i_{C}i_{\xi}\alpha
=&\,C_{\bar i\bar j}^{\bar k}\bar\theta^{i}\wedge\bar\theta^{j}\wedge \bar e_{k}\lrcorner\,(\xi^{q}_{\bar p}\bar\theta^{p}\wedge e_{q}\lrcorner\alpha)\\
=&\,C_{\bar i\bar j}^{\bar k}\xi^{q}_{\bar p}\bar\theta^{i}\wedge\bar\theta^{j}\delta^{p}_{k}\wedge (e_{q}\lrcorner\alpha)
-C_{\bar i\bar j}^{\bar k}\xi^{q}_{\bar p}\bar\theta^{i}\wedge\bar\theta^{j}\wedge\bar\theta^{p}\wedge (\bar e_{k}\lrcorner e_{q}\lrcorner\alpha)\\
=&\,C_{\bar i\bar j}^{\bar k}\xi^{q}_{\bar k}\bar\theta^{i}\wedge\bar\theta^{j}\wedge (e_{q}\lrcorner\alpha)
-C_{\bar i\bar j}^{\bar k}\xi^{q}_{\bar p}\bar\theta^{i}\wedge\bar\theta^{j}\wedge\bar\theta^{p}\wedge (\bar e_{k}\lrcorner e_{q}\lrcorner\alpha)\\
=&\,i_{i_{C}\xi}\alpha+i_{\xi}i_{C}\alpha.
\end{align*}
Thus
\begin{equation*}
i_{C}i_{\xi}\alpha-i_{\xi}i_{C}\alpha=i_{i_{C}\xi}\alpha.
\end{equation*}
\begin{lemma}\label{lemma5}
For any $A\in A^{0,2}(M,T^{1,0}M)$, $B\in A^{1,1}(M,T^{1,0}M)$, $C\in A^{0,2}(M,T^{0,1}M)$ and $\xi\in A^{0,1}(M,T^{1,0}M)$, we have
\begin{align*}
[i_{A},i_{\xi}]&=i_{A}\circ i_{\xi}-i_{\xi}\circ i_{A}=0,\\
[i_{B},i_{\xi}]&=i_{B}\circ i_{\xi}-i_{\xi}\circ i_{B}=-i_{i_{\xi}B},\\
[i_{C},i_{\xi}]&=i_{C}\circ i_{\xi}-i_{\xi}\circ i_{C}=i_{i_{C}\xi}.
\end{align*}
\end{lemma}
By some simple calculations and applying Lemma \ref{lemma5}, we can prove the following lemma.
\begin{lemma}\label{lemma11}
If $\varphi,\psi,\xi\in A^{0,1}(M,T^{1,0}M)$, then we have
\begin{align}\label{eq51}
[i_{[\varphi,\psi]},i_{\xi}]
=&\,i_{[\varphi,\psi]}\circ i_{\xi}-i_{\xi}\circ i_{[\varphi,\psi]}\nonumber\\
=&\,i_{(i_{[\varphi,\psi]}\xi-i_{\xi}[\varphi,\psi])}\nonumber\\
=&\,i_{(i_{\mathscr{C}(\varphi,\psi)}\xi-i_{\xi}\mathscr{B}(\varphi,\psi))}.
\end{align}
\end{lemma}
By Lemma \ref{lemma11}, we can continue the calculation of (\ref{eq44}).
\begin{align}\label{eq47}
[\nabla,i^{3}_{\phi}]\alpha
=&\,3i^{2}_{\phi}[\nabla,i_{\phi}]\alpha-2i_{\phi}\circ i_{[\phi,\phi]}\alpha
-i_{[\phi,\phi]}i_{\phi}\alpha\nonumber\\
=&\,3i^{2}_{\phi}[\nabla,i_{\phi}]\alpha-2i_{\phi}\circ i_{[\phi,\phi]}\alpha
-(i_{\phi}\circ i_{[\phi,\phi]}+i_{(i_{[\phi,\phi]}\xi-i_{\xi}[\phi,\phi])})\alpha\nonumber\\
=&\,3i^{2}_{\phi}[\nabla,i_{\phi}]\alpha-3i_{\phi}\circ i_{[\phi,\phi]}\alpha
-i_{(i_{[\phi,\phi]}\xi-i_{\xi}[\phi,\phi])})\alpha.\\
[\nabla,i^{4}_{\phi}]\alpha
=&\,\nabla(i^{3}_{\phi}i_{\phi}\alpha)-i^{3}_{\phi}i_{\phi}\nabla\alpha\nonumber\\
=&\,\nabla(i^{3}_{\phi}i_{\phi}\alpha)+i^{3}_{\phi}([\nabla,i_{\phi}]-\nabla\circ i_{\phi})\alpha\nonumber\\
=&\,\nabla\circ i^{3}_{\phi}(i_{\phi}\alpha)-i^{3}_{\phi}\circ\nabla( i_{\phi}\alpha)+i^{3}_{\phi}[\nabla,i_{\phi}]\alpha\nonumber\\
=&\,[\nabla,i^{3}_{\phi}](i_{\phi}\alpha)+i^{3}_{\phi}[\nabla,i_{\phi}]\alpha\nonumber\\
=&\,3i^{2}_{\phi}[\nabla,i_{\phi}]\circ i_{\phi}\alpha-3i_{\phi}i_{[\phi,\phi]}i_{\phi}\alpha
-i_{(i_{[\phi,\phi]}\phi-i_{\phi}[\phi,\phi])}\circ i_{\phi}\alpha+i^{3}_{\phi}[\nabla,i_{\phi}]\alpha\nonumber\\
=&\,3i^{2}_{\phi}\circ(i_{\phi}\circ[\nabla,i_{\phi}]-i_{[\phi,\phi]})\alpha
-3i_{\phi}(i_{\phi}i_{[\phi,\phi]}+i_{(i_{[\phi,\phi]}\phi-i_{\phi}[\phi,\phi])})\alpha\nonumber\\
&\,-i_{(i_{[\phi,\phi]}\phi-i_{\phi}[\phi,\phi])}\circ i_{\phi}\alpha
+i^{3}_{\phi}\circ[\nabla,i_{\phi}]\alpha\nonumber\\
=&\,3i^{3}_{\phi}[\nabla,i_{\phi}]\alpha-3i^{2}_{\phi}i_{[\phi,\phi]}\alpha
-3i^{2}_{\phi}i_{[\phi,\phi]}\alpha
-3i_{\phi}\circ i_{(i_{[\phi,\phi]}\phi-i_{\phi}[\phi,\phi])}\alpha\nonumber\\
&\,-i_{(i_{[\phi,\phi]}\phi-i_{\phi}[\phi,\phi])}\circ i_{\phi}\alpha
+i^{3}_{\phi}[\nabla,i_{\phi}]\alpha\nonumber\\
=&\,4i^{3}_{\phi}[\nabla,i_{\phi}]\alpha-6i^{2}_{\phi}i_{[\phi,\phi]}\alpha
-3i_{\phi}\circ i_{(i_{[\phi,\phi]}\phi-i_{\phi}[\phi,\phi])}\alpha
-i_{(i_{[\phi,\phi]}\phi-i_{\phi}[\phi,\phi])}\circ i_{\phi}\alpha.\label{eq46}
\end{align}
Now we need to calculate
\begin{equation*}
i_{\phi}\circ i_{(i_{[\phi,\phi]}\phi-i_{\phi}[\phi,\phi])}-i_{(i_{[\phi,\phi]}\phi-i_{\phi}[\phi,\phi])}\circ i_{\phi}.
\end{equation*}
We give the following lemma.
\begin{lemma}\label{lemma3}
If $\varphi,\xi\in A^{0,1}_{J}(M,T^{1,0}M)$ and $\psi\in A^{0,2}_{J}(M,T^{1,0}M)\oplus A^{1,1}_{J}(M,T^{1,0}M)\oplus A^{0,2}_{J}(M,T^{0,1}M)$, then
\begin{equation*}%\label{eq17}
i_{\varphi}\circ i_{(i_{\psi}\xi-i_{\xi}\psi)}=i_{(i_{\psi}\xi-i_{\xi}\psi)}\circ i_{\varphi}.
\end{equation*}
\end{lemma}
\begin{proof}
Note that $i_{\psi}\xi-i_{\xi}\psi\in A^{0,2}_{J}(M,T^{1,0}M)$.
\end{proof}
As a direct consequence of Lemma \ref{lemma3}, for $\phi\in A^{0,1}_{J}(M,T^{1,0}M)$ we have
\begin{equation}\label{eq45}
i_{\phi}\circ i_{(i_{[\phi,\phi]}\phi-i_{\phi}[\phi,\phi])}
=i_{(i_{[\phi,\phi]}\phi-i_{\phi}[\phi,\phi])}\circ i_{\phi}.
\end{equation}
By (\ref{eq45}), we continue the calculation of (\ref{eq46}).
\begin{align}\label{eq461}
[\nabla,i^{4}_{\phi}]\alpha
=&\,4i^{3}_{\phi}[\nabla,i_{\phi}]\alpha-6i^{2}_{\phi}i_{[\phi,\phi]}\alpha
-3i_{\phi}\circ i_{(i_{[\phi,\phi]}\phi-i_{\phi}[\phi,\phi])}\alpha
-i_{(i_{[\phi,\phi]}\phi-i_{\phi}[\phi,\phi])}\circ i_{\phi}\alpha\nonumber\\
=&\,4i^{3}_{\phi}[\nabla,i_{\phi}]\alpha-6i^{2}_{\phi}i_{[\phi,\phi]}\alpha
-4i_{\phi}\circ i_{(i_{[\phi,\phi]}\phi-i_{\phi}[\phi,\phi])}\alpha.
\end{align}
By comparing (\ref{eq48}), (\ref{eq47}), (\ref{eq461}) and considering (\ref{eq45}), we introduce the following proposition.
\begin{proposition}\label{proposition1}
Let $(M,J)$ be an almost complex manifold, $\phi\in A^{0,1}_{J}(M,T^{1,0}M)$ be a Beltrami differential, $E$ be a vector bundle on $M$ and $\nabla$ be a connection of $E$. Then we have
\begin{equation}\label{eq49}
[\nabla,i^{k}_{\phi}]
=ki^{k-1}_{\phi}\circ[\nabla,i_{\phi}]
-\frac{k(k-1)}{2}i^{k-2}_{\phi}\circ i_{[\phi,\phi]}
-\frac{k(k-1)(k-2)}{3!}i^{k-3}_{\phi}\circ i_{(i_{[\phi,\phi]}\phi-i_{\phi}[\phi,\phi])}.
\end{equation}
\end{proposition}
\begin{proof}
For $n=2,3,4$, we already have (\ref{eq47}),(\ref{eq48}), (\ref{eq46}). Now we assume that (\ref{eq49}) holds for $n=k$. For the case $n=k+1$, we have
\begin{align*}
[\nabla,i^{k+1}_{\phi}]\alpha
=&\,\nabla\circ i^{k}_{\phi}(i_{\phi}\alpha)-i^{k}_{\phi}(i_{\phi}\circ \nabla)\alpha\\
=&\,\nabla\circ i^{k}_{\phi}(i_{\phi}\alpha)+i^{k}_{\phi}([\nabla,i_{\phi}]-\nabla\circ i_{\phi})\alpha\\
=&\,\nabla\circ i^{k}_{\phi}(i_{\phi}\alpha)-i^{k}_{\phi}\circ\nabla(i_{\phi}\alpha)+i^{k}_{\phi}\circ [\nabla,i_{\phi}]\alpha\\
=&\,[\nabla,i^{k}_{\phi}](i_{\phi}\alpha)+i^{k}_{\phi}\circ [\nabla,i_{\phi}]\alpha\\
=&\,ki^{k-1}_{\phi}\circ[\nabla,i_{\phi}](i_{\phi}\alpha)
-\frac{k(k-1)}{2}i^{k-2}_{\phi}\circ i_{[\phi,\phi]}(i_{\phi}\alpha)\\
&-\frac{k(k-1)(k-2)}{3!}i^{k-3}_{\phi}\circ i_{(i_{[\phi,\phi]\phi}
-i_{\phi}[\phi,\phi])}(i_{\phi}\alpha)
+i^{k}_{\phi}\circ [\nabla,i_{\phi}]\alpha\\
=&\,ki^{k-1}_{\phi}\circ(i_{\phi}\circ[\nabla,i_{\phi}]-i_{[\phi,\phi]})\alpha\quad( \textup{By}\quad(\ref{eq50}))\\
&-\frac{k(k-1)}{2}i^{k-2}_{\phi}\circ\bigl(i_{\phi}i_{[\phi,\phi]}
+i_{(i_{[\phi,\phi]}\phi-i_{\phi}[\phi,\phi])}\bigr)\alpha\quad (\textup{By}\quad(\ref{eq51}))\\
&-\frac{k(k-1)(k-2)}{3!}i^{k-3}_{\phi}\circ i_{\phi}i_{(i_{[\phi,\phi]\phi}
-i_{\phi}[\phi,\phi])}\alpha\quad(\textup{By}\quad (\ref{eq45}))\\
&+i^{k}_{\phi}\circ [\nabla,i_{\phi}]\alpha\\
=&\,ki^{k}_{\phi}\circ[\nabla,i_{\phi}]\alpha-ki^{k-1}_{\phi}i_{[\phi,\phi]}\alpha
-\frac{k(k-1)}{2}i^{k-1}_{\phi}\circ i_{[\phi,\phi]}\alpha
-\frac{k(k-1)}{2}i^{k-2}_{\phi}\circ i_{(i_{[\phi,\phi]}\phi-i_{\phi}[\phi,\phi])}\alpha\\
&-\frac{k(k-1)(k-2)}{3!}i^{k-2}_{\phi}\circ i_{(i_{[\phi,\phi]\phi}
-i_{\phi}[\phi,\phi])}\alpha
+i^{k}_{\phi}\circ [\nabla,i_{\phi}]\alpha\\
=&\,(k+1)i^{k}_{\phi}\circ[\nabla,i_{\phi}]\alpha
-\frac{(k+1)k}{2}i^{k-1}_{\phi}\circ i_{[\phi,\phi]}\alpha
-\frac{(k+1)k(k-1)}{3!}i^{k-2}_{\phi}\circ i_{(i_{[\phi,\phi]}\phi-i_{\phi}[\phi,\phi])}\alpha.
\end{align*}
We have proved the proposition.
\end{proof}
On the base of Proposition \ref{proposition1}, we can prove the following theorem.
\begin{theorem}\label{theorem1}
Let $(M,J)$ be an almost complex manifold, $E\rightarrow M$ be a vector bundle over $M$, $\nabla$ be a connection of $E$ and $\phi\in A^{0,1}_{J}(M,T^{1,0}M)$ be a Beltrami differential that generates a new almost complex structure $J_{\phi}$. Then on $A^{\ast,\ast}_{J}(M,E)$, we have
\begin{equation}\label{eq64}
e^{-i_{\phi}}\circ\nabla\circ e^{i_{\phi}}
=\nabla-\mathcal{L}^{\nabla}_{\phi}-i_{\frac{1}{2}[\phi,\phi]}
-i_{\frac{1}{3!}(i_{[\phi,\phi]}\phi-i_{\phi}[\phi,\phi])}.
\end{equation}
\end{theorem}
\begin{proof}
A direct consequence of $\nabla\circ i^{k}_{\phi}-i^{k}_{\phi}\circ\nabla=[\nabla,i^{k}_{\phi}]$ and Proposition \ref{proposition1}, we have
\begin{align*}
&\nabla\circ \frac{1}{k!}i^{k}_{\phi}-\frac{1}{k!}i^{k}_{\phi}\circ\nabla\\
=&\,\frac{1}{k!}ki^{k-1}_{\phi}\circ[\nabla,i_{\phi}]
-\frac{1}{k!}\frac{k(k-1)}{2}i^{k-2}_{\phi}\circ i_{[\phi,\phi]}
-\frac{1}{k!}\frac{k(k-1)(k-2)}{3!}i^{k-3}_{\phi}\circ i_{(i_{[\phi,\phi]}\phi-i_{\phi}[\phi,\phi])}\\
=&\,\frac{1}{(k-1)!}i^{k-1}_{\phi}\circ[\nabla,i_{\phi}]
-\frac{1}{(k-2)!}i^{k-2}_{\phi}\circ i_{\frac{1}{2}[\phi,\phi]}
-\frac{1}{(k-3)!}i^{k-3}_{\phi}\circ i_{\frac{1}{3!}(i_{[\phi,\phi]}\phi-i_{\phi}[\phi,\phi])}.
\end{align*}
Thus
\begin{equation*}
\nabla\circ e^{i_{\phi}}-e^{i_{\phi}}\circ\nabla
=e^{i_{\phi}}\circ[\nabla,i_{\phi}]-e^{i_{\phi}}\circ i_{\frac{1}{2}[\phi,\phi]}-e^{i_{\phi}}\circ i_{\frac{1}{3!}(i_{[\phi,\phi]}\phi-i_{\phi}[\phi,\phi])}.
\end{equation*}
\end{proof}
If we choose the vector bundle to be the trivial bundle $M\times\mathbb{R}$ and the connection to be the exterior differential $d$, then we have the following corollary.
\begin{corollary}\label{corollary3}
Let $(M,J)$ be an almost complex manifold, $d$ be the ordinary exterior differential operator over $M$. Then
\begin{equation*}
e^{-i_{\phi}}\circ d\circ e^{i_{\phi}}
=d-\mathcal{L}_{\phi}-i_{\frac{1}{2}[\phi,\phi]}
-i_{\frac{1}{3!}(i_{[\phi,\phi]}\phi-i_{\phi}[\phi,\phi])}.
\end{equation*}
\end{corollary}
When $\nabla^{0,1}=\bar\partial$, we have the following useful lemma.
\begin{lemma}\label{nlemma1}
Let $\phi$ be a Beltrami differential, $\bar\partial$ be the (extended) Cauchy-Riemannian operator and $\alpha\in A^{p,q}_{J}(M)$, then
\begin{equation*}
\bar\partial(\phi\lrcorner\,\alpha)=(\bar\partial\phi)\lrcorner\,\alpha+\phi\lrcorner\,\bar\partial\alpha.
\end{equation*}
Equivalently
\begin{equation*}
[\bar\partial,i_{\phi}]=i_{\bar\partial\phi}.
\end{equation*}
\end{lemma}
\begin{proof}
For $\alpha=\alpha_{i_{1}\ldots i_{p}\bar J}\theta^{i_{1}}\wedge\cdots\wedge\theta^{i_{p}}\wedge\theta^{\bar J}$, $|\bar J|=q$,
\begin{small}
\begin{align}
\bar\partial(e_{k}\lrcorner\alpha)
=&\,\bar\partial(\alpha_{i_{1}\ldots i_{p}\bar J}e_{k}\lrcorner(\theta^{i_{1}}\wedge\cdots\wedge\theta^{i_{p}}\wedge\theta^{\bar J}))\nonumber\\
=&\,\bar\partial\alpha_{i_{1}\ldots i_{p}\bar J}\wedge
(e_{k}\lrcorner(\theta^{i_{1}}\wedge\cdots\wedge\theta^{i_{p}}\wedge\theta^{\bar J}))\nonumber\\
&+\alpha_{i_{1}\ldots i_{p}\bar J}
\sum_{\ell=1}^{p}(-1)^{\ell-1}\bar\partial(\theta^{i_{1}}\wedge\cdots\wedge\theta^{i_{\ell-1}}\wedge(e_{k}\lrcorner\theta^{i_{\ell}})
\wedge\theta^{i_{\ell+1}}\wedge\cdots\wedge\theta^{i_{p}}\wedge\theta^{\bar J})\nonumber\\
=&\,\bar\partial\alpha_{i_{1}\ldots i_{p}\bar J}\wedge
e_{k}\lrcorner(\theta^{i_{1}}\wedge\cdots\wedge\theta^{i_{p}}\wedge\theta^{\bar J})\nonumber\\
&+\alpha_{i_{1}\ldots i_{p}\bar J}
\sum_{\ell=1}^{p}(-1)^{\ell-1}\sum_{t=1}^{\ell-1}(-1)^{t-1}\theta^{i_{1}}\wedge\cdots\wedge\bar\partial\theta^{i_{t}}\wedge\cdots\wedge
\theta^{i_{\ell-1}}\wedge(e_{k}\lrcorner\theta^{i_{\ell}})\wedge\theta^{i_{\ell+1}}\wedge\cdots\wedge\theta^{i_{p}}\wedge\theta^{\bar J}\nonumber\\
&+\alpha_{i_{1}\ldots i_{p}\bar J}
\sum_{\ell=1}^{p}(-1)^{\ell-1}\sum_{t=\ell+1}^{p}(-1)^{t}\theta^{i_{1}}\wedge\cdots\wedge\theta^{i_{\ell-1}}\wedge(e_{k}\lrcorner\theta^{i_{\ell}})
\wedge\theta^{i_{\ell+1}}\wedge\cdots\wedge\bar\partial\theta^{i_{t}}\wedge\cdots\wedge\theta^{i_{p}}\wedge\theta^{\bar J}\nonumber\\
&+\alpha_{i_{1}\ldots i_{p}\bar J}
(-1)^{p+1}e_{k}\lrcorner(\theta^{i_{1}}\wedge\cdots\wedge\theta^{i_{p}}\wedge\bar\partial\theta^{\bar J})\nonumber\\
=&\,\bar\partial\alpha_{i_{1}\ldots i_{p}\bar J}\wedge
e_{k}\lrcorner(\theta^{i_{1}}\wedge\cdots\wedge\theta^{i_{p}}\wedge\theta^{\bar J})\nonumber\\
&+\alpha_{i_{1}\ldots i_{p}\bar J}
\sum_{\ell=1}^{p}\sum_{t=1}^{\ell-1}(-1)^{\ell+t}\theta^{i_{1}}\wedge\cdots\wedge\bar\partial\theta^{i_{t}}\wedge\cdots\wedge
\theta^{i_{\ell-1}}\wedge(e_{k}\lrcorner\theta^{i_{\ell}})\wedge\theta^{i_{\ell+1}}\wedge\cdots\wedge\theta^{i_{p}}\wedge\theta^{\bar J}\nonumber\\
&+\alpha_{i_{1}\ldots i_{p}\bar J}
\sum_{\ell=1}^{p}\sum_{t=\ell+1}^{p}(-1)^{\ell+t-1}\theta^{i_{1}}\wedge\cdots\wedge\theta^{i_{\ell-1}}\wedge(e_{k}\lrcorner\theta^{i_{\ell}})
\wedge\theta^{i_{\ell+1}}\wedge\cdots\wedge\bar\partial\theta^{i_{t}}\wedge\cdots\wedge\theta^{i_{p}}\wedge\theta^{\bar J}\nonumber\\
&+\alpha_{i_{1}\ldots i_{p}\bar J}
(-1)^{p+1}e_{k}\lrcorner(\theta^{i_{1}}\wedge\cdots\wedge\theta^{i_{p}}\wedge\bar\partial\theta^{\bar J}).\label{neweq6}
\end{align}
\end{small}
\begin{small}
\begin{align}
\bar\partial e_{k}\lrcorner\alpha
=&\,\bar\partial e_{k}\lrcorner\alpha_{i_{1}\ldots i_{p}\bar J}\theta^{i_{1}}\wedge\cdots\wedge\theta^{i_{p}}\wedge\theta^{\bar J}\nonumber\\
=&\,\alpha_{i_{1}\ldots i_{p}\bar j_{1}\ldots \bar j_{q}}\bar\partial e_{k}\lrcorner(\theta^{i_{1}}\wedge\cdots\wedge\theta^{i_{p}}\wedge\theta^{\bar j_{1}}\wedge\cdots\wedge\theta^{\bar j_{q}})\nonumber\\
=&\,\alpha_{i_{1}\ldots i_{p}\bar J}\sum_{\ell=1}^{p}\theta^{i_{1}}\wedge\cdots\wedge(\bar\partial e_{k}\lrcorner\theta^{i_{\ell}})\wedge\cdots\wedge\theta^{i_{p}}\wedge\theta^{\bar J}\nonumber\\
=&\,\alpha_{i_{1}\ldots i_{p}\bar J}
\sum_{\ell=1}^{p}\theta^{i_{1}}\wedge\cdots\wedge(e_{k}\lrcorner\bar\partial\theta^{i_{\ell}})\wedge\cdots\wedge\theta^{i_{p}}\wedge\theta^{\bar J}\nonumber\\
=&\,\alpha_{i_{1}\ldots i_{p}\bar J}
\sum_{\ell=1}^{p}(-1)^{\ell-1}e_{k}\lrcorner(\theta^{i_{1}}\wedge\cdots\wedge\bar\partial\theta^{i_{\ell}}\wedge\cdots\wedge\theta^{i_{p}}\wedge\theta^{\bar J})\nonumber\\
&-\alpha_{i_{1}\ldots i_{p}\bar J}\sum_{\ell=1}^{p}(-1)^{\ell-1}\sum_{t=1}^{\ell-1}(-1)^{t-1}\theta^{i_{1}}\wedge\cdots\wedge(e_{k}\lrcorner\theta^{i_{t}})\wedge\cdots\wedge
\bar\partial\theta^{i_{\ell}}\wedge\cdots\wedge\theta^{i_{p}}\wedge\theta^{\bar J}\nonumber\\
&-\alpha_{i_{1}\ldots i_{p}\bar J}\sum_{\ell=1}^{p}(-1)^{\ell-1}\sum_{t=\ell+1}^{p}(-1)^{t}\theta^{i_{1}}\wedge\cdots\wedge\bar\partial\theta^{i_{\ell}}\wedge\cdots
(e_{k}\lrcorner\theta^{i_{t}})\wedge\cdots\wedge\theta^{i_{p}}\wedge\theta^{\bar J}\bigr\}\nonumber\\
=&\,\alpha_{i_{1}\ldots i_{p}\bar J}
\sum_{\ell=1}^{p}(-1)^{\ell-1}e_{k}\lrcorner(\theta^{i_{1}}\wedge\cdots\wedge\bar\partial\theta^{i_{\ell}}\wedge\cdots\wedge\theta^{i_{p}}\wedge\theta^{\bar J})\nonumber\\
&-\alpha_{i_{1}\ldots i_{p}\bar J}\sum_{\ell=1}^{p}\sum_{t=1}^{\ell-1}(-1)^{\ell+t}\theta^{i_{1}}\wedge\cdots\wedge(e_{k}\lrcorner\theta^{i_{t}})\wedge\cdots\wedge
\bar\partial\theta^{i_{\ell}}\wedge\cdots\wedge\theta^{i_{p}}\wedge\theta^{\bar J}\nonumber\\
&-\alpha_{i_{1}\ldots i_{p}\bar J}\sum_{\ell=1}^{p}\sum_{t=\ell+1}^{p}(-1)^{\ell+t-1}\theta^{i_{1}}\wedge\cdots\wedge\bar\partial\theta^{i_{\ell}}\wedge\cdots
(e_{k}\lrcorner\theta^{i_{t}})\wedge\cdots\wedge\theta^{i_{p}}\wedge\theta^{\bar J}\nonumber\\
=&\,\alpha_{i_{1}\ldots i_{p}\bar J}
\sum_{\ell=1}^{p}(-1)^{\ell-1}e_{k}\lrcorner(\theta^{i_{1}}\wedge\cdots\wedge\bar\partial\theta^{i_{\ell}}\wedge\cdots\wedge\theta^{i_{p}}\wedge\theta^{\bar J})\nonumber\\
&+\alpha_{i_{1}\ldots i_{p}\bar J}\sum_{\ell=1}^{p}\sum_{t<\ell}(-1)^{\ell+t-1}\theta^{i_{1}}\wedge\cdots\wedge(e_{k}\lrcorner\theta^{i_{t}})\wedge\cdots\wedge
\bar\partial\theta^{i_{\ell}}\wedge\cdots\wedge\theta^{i_{p}}\wedge\theta^{\bar J}\nonumber\\
&+\alpha_{i_{1}\ldots i_{p}\bar J}\sum_{\ell=1}^{p}\sum_{t>\ell}(-1)^{\ell+t}\theta^{i_{1}}\wedge\cdots\wedge\bar\partial\theta^{i_{\ell}}\wedge\cdots
(e_{k}\lrcorner\theta^{i_{t}})\wedge\cdots\wedge\theta^{i_{p}}\wedge\theta^{\bar J}\nonumber\\
=&\,\alpha_{i_{1}\ldots i_{p}\bar J}
\sum_{\ell=1}^{p}(-1)^{\ell-1}e_{k}\lrcorner(\theta^{i_{1}}\wedge\cdots\wedge\bar\partial\theta^{i_{\ell}}\wedge\cdots\wedge\theta^{i_{p}}\wedge\theta^{\bar J})\nonumber\\
&+\alpha_{i_{1}\ldots i_{p}\bar J}\sum_{\ell=1}^{p}\sum_{t>\ell}(-1)^{\ell+t-1}\theta^{i_{1}}\wedge\cdots\wedge(e_{k}\lrcorner\theta^{i_{\ell}})\wedge\cdots\wedge
\bar\partial\theta^{i_{t}}\wedge\cdots\wedge\theta^{i_{p}}\wedge\theta^{\bar J}\quad\textup{(By direct observation.)}\nonumber\\
&+\alpha_{i_{1}\ldots i_{p}\bar J}\sum_{\ell=1}^{p}\sum_{t<\ell}(-1)^{\ell+t}\theta^{i_{1}}\wedge\cdots\wedge\bar\partial\theta^{i_{t}}\wedge\cdots
(e_{k}\lrcorner\theta^{i_{\ell}})\wedge\cdots\wedge\theta^{i_{p}}\wedge\theta^{\bar J}.\label{neweq7}
\end{align}
\end{small}
\begin{small}
\begin{align}
e_{k}\lrcorner\bar\partial\alpha
=&\,e_{k}\lrcorner\bar\partial(\alpha_{i_{1}\ldots i_{p}\bar J}\theta^{i_{1}}\wedge\cdots\wedge\theta^{i_{p}}\wedge\theta^{\bar J})\nonumber\\
=&\,e_{k}\lrcorner(\bar\partial\alpha_{i_{1}\ldots i_{p}\bar J}\wedge
\theta^{i_{1}}\wedge\cdots\wedge\theta^{i_{p}}\wedge\theta^{\bar J})
+\alpha_{i_{1}\ldots i_{p}\bar J}
e_{k}\lrcorner\bar\partial(\theta^{i_{1}}\wedge\cdots\wedge\theta^{i_{p}}\wedge\theta^{\bar J})\nonumber\\
=&\,-\bar\partial\alpha_{i_{1}\ldots i_{p}\bar J}\wedge e_{k}\lrcorner(\theta^{i_{1}}\wedge\cdots\wedge\theta^{i_{p}}\wedge\theta^{\bar J})
+\alpha_{i_{1}\ldots i_{p}\bar J}e_{k}\lrcorner
\sum_{\ell=1}^{p}(-1)^{\ell-1}\theta^{i_{1}}\wedge\cdots\wedge\bar\partial\theta^{i_{\ell}}\wedge\cdots\wedge\theta^{i_{p}}\wedge\theta^{\bar J}\nonumber\\
&+\alpha_{i_{1}\ldots i_{p}\bar J}e_{k}\lrcorner
(-1)^{p+1}(\theta^{i_{1}}\wedge\cdots\wedge\theta^{i_{p}}\wedge\bar\partial\theta^{\bar J})\nonumber\\
=&\,-\bar\partial\alpha_{i_{1}\ldots i_{p}\bar J}\wedge e_{k}\lrcorner(\theta^{i_{1}}\wedge\cdots\wedge\theta^{i_{p}}\wedge\theta^{\bar J})
+\alpha_{i_{1}\ldots i_{p}\bar J}e_{k}\lrcorner
\sum_{\ell=1}^{p}(-1)^{\ell-1}\theta^{i_{1}}\wedge\cdots\wedge\bar\partial\theta^{i_{\ell}}\wedge\cdots\wedge\theta^{i_{p}}\wedge\theta^{\bar J}\nonumber\\
&+\alpha_{i_{1}\ldots i_{p}\bar J}(-1)^{p}e_{k}\lrcorner
(\theta^{i_{1}}\wedge\cdots\wedge\theta^{i_{p}}\wedge\bar\partial\theta^{\bar J}).\label{neweq8}
\end{align}
\end{small}
Comparing (\ref{neweq6}), (\ref{neweq7}) and (\ref{neweq8}) we get
\begin{equation}\label{neweq9}
\bar\partial(e_{k}\lrcorner\,\alpha)=(\bar\partial e_{k})\lrcorner\,\alpha-e_{k}\lrcorner(\bar\partial\alpha).
\end{equation}
For $\phi=\phi^{r}\otimes e_{r}\in A^{0,1}_{J}(M,T^{1,0}M)$, we have
\begin{align*}
\bar\partial(\phi\lrcorner\,\alpha)
=&\,\bar\partial(\phi^{r}\wedge(e_{r}\lrcorner\alpha))\\
=&\,\bar\partial\phi^{r}\wedge(e_{r}\lrcorner\alpha)-\phi^{r}\wedge\bar\partial(e_{r}\lrcorner\alpha)\\
=&\,\bar\partial\phi^{r}\wedge(e_{r}\lrcorner\alpha)-\phi^{r}\wedge((\bar\partial e_{r})\lrcorner\alpha-e_{r}\lrcorner\bar\partial\alpha)&\textup{By (\ref{neweq9})}\\
=&\,\bar\partial\phi^{r}\wedge(e_{r}\lrcorner\alpha)-\phi^{r}\wedge(\bar\partial e_{r})\lrcorner\alpha+\phi^{r}\wedge(e_{r}\lrcorner\bar\partial\alpha)\\
=&\,\bar\partial\phi\lrcorner\,\alpha+\phi\lrcorner\,\bar\partial\alpha.
\end{align*}
We have proved
\begin{equation*}
\bar\partial(\phi\lrcorner\,\alpha)=\bar\partial\phi\lrcorner\,\alpha+\phi\lrcorner\,\bar\partial\alpha.
\end{equation*}
\end{proof}
By Lemma \ref{nlemma1} and $[\bar\mu,i_{\phi}]=0$, we can rewrite Corollary \ref{corollary3} as
\begin{corollary}\label{corollary4}
Let $(M,J)$ be an almost complex manifold, $d$ be the ordinary exterior differential operator over $M$. Then
\begin{equation}\label{eq89}
e^{-i_{\phi}}\circ d\circ e^{i_{\phi}}=d+[\mu,i_{\phi}]+[\partial,i_{\phi}]
-i_{\frac{1}{2}(\mathscr{B}(\phi,\phi)+\mathscr{C}(\phi,\phi))}+i_{\mathrm{MC}(\phi)},
\end{equation}
where
\begin{equation*}
\mathrm{MC}(\phi):=\bar\partial\phi-\frac{1}{2}\mathscr{A}(\phi,\phi)-\frac{1}{3!}(i_{[\phi,\phi]}\phi-i_{\phi}[\phi,\phi]).
\end{equation*}
\end{corollary}
\begin{proof}
\begin{align*}
e^{-i_{\phi}}\circ d\circ e^{i_{\phi}}
=&\,d-\mathcal{L}^{\mu}_{\phi}-\mathcal{L}^{\partial}_{\phi}+i_{\bar\partial\phi-\frac{1}{2}[\phi,\phi]}
-i_{\frac{1}{3!}(i_{[\phi,\phi]}\phi-i_{\phi}[\phi,\phi])}\\
=&\,d+[\mu,i_{\phi}]+[\partial,i_{\phi}]+i_{\bar\partial\phi-\frac{1}{2}[\phi,\phi]}
-i_{\frac{1}{3!}(i_{[\phi,\phi]}\phi-i_{\phi}[\phi,\phi])}\\
=&\,d+[\mu,i_{\phi}]+[\partial,i_{\phi}]-i_{\frac{1}{2}(\mathscr{B}(\phi,\phi)+\mathscr{C}(\phi,\phi))}
+i_{\mathrm{MC}(\phi)}.
\end{align*}
\end{proof}
Now we introduce the \emph{almost complex Maurer-Cartan equation} as
\begin{equation}\label{mcequation}
\mathrm{MC}(\phi)=0.
\end{equation}
We would like to point out that the almost complex Maurer-Cartan equation reduces to the well-known Maurer-Cartan equation on complex manifolds, i.e.
\begin{equation*}
\bar\partial\phi-\frac{1}{2}[\phi,\phi]=0.
\end{equation*}
\subsection{Extended exponential operator and extension formula}
 \cite[Definition 1.1]{zhao2015extension} and \cite{rao2016power} introduce an extension of $e^{i_{\phi}}$,
\begin{equation*}
e^{i_{\phi}|i_{\bar\phi}}:A_{J}^{p,q}(M)\rightarrow A^{p,q}_{\phi}(M).
\end{equation*}
For any $\alpha=\frac{1}{p!q!}\alpha_{i_{1}\cdots i_{p}\bar j_{1}\cdots\bar j_{q}}
\theta^{i_{1}}\wedge\cdots\wedge\theta^{i_{p}}\wedge\theta^{\bar j_{1}}\wedge\cdots\wedge\theta^{\bar j_{q}}$,
\begin{equation*}
e^{i_{\phi}|i_{\bar\phi}}(\alpha)
:=\frac{1}{p!q!}\alpha_{i_{1}\cdots i_{p}\bar j_{1}\cdots\bar j_{q}}
e^{i_{\phi}}(\theta^{i_{1}})\wedge\cdots\wedge e^{i_{\phi}}(\theta^{i_{p}})\wedge e^{i_{\bar\phi}}(\theta^{\bar j_{1}})\wedge\cdots\wedge e^{i_{\bar\phi}}(\theta^{\bar j_{q}}).
\end{equation*}
Since $i_{\phi}$ and $i_{\bar\phi}$ are contractions, this definition is independent of the choice of basis. There are several properties of the map $e^{i_{\phi}|i_{\bar\phi}}$, which have been proved in the complex case in \cite{rao2018several}. The proofs in this paper are basicly the same as \cite{rao2018several}. We give them here just for the readers' convenience and completeness of this manuscript.
\begin{lemma}\label{lemma18}\cite[Lemma 2.9]{rao2018several}
Let $(M,J)$ be an almost complex manifold and $\phi\in A^{0,1}_{J}(M,T^{1,0}M)$ that induces a new almost complex structure $J_{\phi}$ (i.e. $\phi$ is a Beltrami differential), then
$$
e^{i_{\phi}|i_{\bar\phi}}:A_{J}^{p,q}(M)\rightarrow A^{p,q}_{\phi}(M),
$$
is a linear isomorphism.
\end{lemma}
\begin{proof}
Note that
\begin{displaymath}
e^{i_{\phi}|i_{\bar\phi}}:
\left\{\begin{array}{cc}
A_{J}^{1,0}\rightarrow A_{\phi}^{1,0}&:\theta^{i}\mapsto\theta^{i}_{\phi}\\
A_{J}^{0,1}\rightarrow A_{\phi}^{0,1}&:\theta^{\bar i}\mapsto\theta^{\bar i}_{\phi}
\end{array}\right..
\end{displaymath}
\end{proof}
\begin{lemma}\label{lemma6}
\begin{equation*}
e^{i_{\phi}}((I-\bar\phi\cdot\phi+\bar\phi)\lrcorner\,\theta^{k})
=(I+\phi)\lrcorner\,\theta^{k}=e^{i_{\phi}}(\theta^{k}),
\end{equation*}
and
\begin{equation*}
e^{i_{\phi}}((I-\bar\phi\cdot\phi+\bar\phi)\lrcorner\,\theta^{\bar k})
=(I+\bar\phi)\lrcorner\,\theta^{\bar k}=e^{i_{\bar\phi}}(\theta^{\bar k}).
\end{equation*}
\end{lemma}
\begin{proof}
Put $\phi^{i}_{\bar j}\theta^{\bar j}\otimes e_{i}$, $\bar\phi=\phi^{\bar i}_{j}\theta^{j}\otimes e_{\bar i}$. By direct calculation we get
\begin{align*}
e^{i_{\phi}}((I-\bar\phi\cdot\phi+\bar\phi)\lrcorner\,\theta^{k})
=&\,e^{i_{\phi}}((I-\phi^{i}_{\bar j}\theta^{\bar j}\wedge e_{i}\lrcorner (\phi^{\bar p}_{q}\theta^{q}\otimes e_{\bar p})+\phi^{\bar i}_{j}\theta^{j}\otimes e_{\bar i})\lrcorner\theta^{k})\\
=&\,e^{i_{\phi}}((I-\phi^{i}_{\bar j}\phi^{\bar p}_{i}\theta^{\bar j}\otimes e_{\bar p}+\phi^{\bar i}_{j}\theta^{j}\otimes e_{\bar i})\lrcorner\theta^{k})\\
=&\,e^{i_{\phi}}(\theta^{k}),
\end{align*}
and
\begin{align*}
e^{i_{\phi}}((I-\bar\phi\cdot\phi+\bar\phi)\lrcorner\,\theta^{\bar k})
=&\,e^{i_{\phi}}((I-\phi^{i}_{\bar j}\phi^{\bar p}_{i}\theta^{\bar j}\otimes e_{\bar p}+\phi^{\bar i}_{j}\theta^{j}\otimes e_{\bar i})\lrcorner\theta^{\bar k})\\
=&\,(I+\phi)\lrcorner(\theta^{\bar k}-\phi^{i}_{\bar j}\phi^{\bar k}_{i}\theta^{\bar j}
+\phi^{\bar k}_{j}\theta^{j})\\
=&\,\theta^{\bar k}-\phi^{i}_{\bar j}\phi^{\bar k}_{i}\theta^{\bar j}+\phi^{\bar k}_{j}\theta^{j}
+\phi^{p}_{\bar q}\theta^{\bar q}\wedge e_{p}\lrcorner(\phi^{\bar k}_{j}\theta^{j})\\
=&\,\theta^{\bar k}-\phi^{i}_{\bar j}\phi^{\bar k}_{i}\theta^{\bar j}+\phi^{\bar k}_{j}\theta^{j}
+\phi^{j}_{\bar q}\phi^{\bar k}_{j}\theta^{\bar q}\\
=&\,\theta^{\bar k}+\phi^{\bar k}_{j}\theta^{j}\\
=&\,(I+\bar\phi)\lrcorner\theta^{\bar k}\\
=&\,e^{i_{\bar\phi}}(\theta^{\bar k}).
\end{align*}
\end{proof}
\begin{lemma}\cite[Lemma 2.10]{rao2018several}
$e^{i_{\phi}|i_{\bar\phi}}$ is a real operator.
\end{lemma}
\begin{proof}
Check that for any $\alpha\in A^{p,q}_{J}(M)$,
\begin{equation*}
\overline{e^{i_{\phi}|i_{\bar\phi}}(\alpha)}=e^{i_{\phi}|i_{\bar\phi}}(\overline{\alpha}).
\end{equation*}
\end{proof}
\cite[Equation 2.8]{rao2016power} introduces the \emph{simultaneous contraction} (denoted ``$\Finv\,$") on each component of a complex differential form. For example, if $\varphi$ is a contraction operator,
\begin{equation*}
\alpha=\alpha_{I_{p},\bar J_{q}}\theta^{i_{1}}\wedge\cdots\wedge\theta^{i_{p}}\wedge\theta^{\bar j_{1}}\wedge\cdots\wedge\theta^{\bar j_{q}},
\end{equation*}
then
\begin{equation}\label{neweq12}
\varphi\Finv\,\alpha:=\alpha_{i_{1}\cdots i_{p},\bar j_{1}\cdots \bar j_{q}}
(\varphi\lrcorner\,\theta^{i_{1}})\wedge\cdots\wedge(\varphi\lrcorner\,\theta^{i_{p}})\wedge(\varphi\lrcorner\,\theta^{\bar j_{1}})\wedge\cdots\wedge(\varphi\lrcorner\,\theta^{\bar j_{q}}).
\end{equation}
\begin{remark}\label{remark2}
By the definition of the simultaneous contraction $\Finv\,$, one can check that generally
\begin{equation*}
(\varphi+\psi)\Finv\,\alpha\neq\varphi\Finv\,\alpha+\psi\Finv\,\alpha.
\end{equation*}
When $\alpha\in A^{1}(M)$, we have
\begin{equation*}
(\varphi+\psi)\Finv\,\alpha=(\varphi+\psi)\lrcorner\,\alpha=\varphi\lrcorner\,\alpha+\psi\lrcorner\,\alpha.
\end{equation*}
\end{remark}
\begin{lemma}\label{lemma9}
\begin{equation*}
\theta^{\bar j}=e^{i_{\phi}|i_{\bar\phi}}((-\bar\phi\cdot(I^{\prime}-\phi\cdot\bar\phi)^{-1}
+(I^{\prime\prime}-\bar\phi\cdot\phi)^{-1})\lrcorner\,\theta^{\bar j}).
\end{equation*}
\end{lemma}
\begin{proof}
\begin{align*}
\theta^{\bar j}=&(\Phi^{-1})^{\bar j}_{\lambda}\theta_{\phi}^{\lambda}\\
=&\,(\Phi^{-1})^{\bar j}_{k}\theta_{\phi}^{k}+(\Phi^{-1})^{\bar j}_{\bar\ell}\theta_{\phi}^{\bar\ell}\\
=&\,(-\bar\phi\cdot(I^{\prime}-\phi\cdot\bar\phi)^{-1})^{\bar j}_{k}\theta_{\phi}^{k}
+((I^{\prime\prime}-\bar\phi\cdot\phi)^{-1})^{\bar j}_{\bar\ell}\theta_{\phi}^{\bar\ell}\\
=&\,(-\bar\phi\cdot(I^{\prime}-\phi\cdot\bar\phi)^{-1})^{\bar j}_{k}e^{i_{\phi}}(\theta^{k})
+((I^{\prime\prime}-\bar\phi\cdot\phi)^{-1})^{\bar j}_{\bar\ell}e^{i_{\bar\phi}}(\theta^{\bar\ell})\\
=&\,e^{i_{\phi}}((-\bar\phi\cdot(I^{\prime\prime}-\phi\cdot\bar\phi)^{-1})^{\bar j}_{k}\theta^{k})
+e^{i_{\bar\phi}}(((I^{\prime\prime}-\bar\phi\cdot\phi)^{-1})^{\bar j}_{\bar\ell}\theta^{\bar\ell})\\
=&\,e^{i_{\phi}}(-\bar\phi\cdot(I^{\prime}-\phi\cdot\bar\phi)^{-1}\lrcorner\theta^{\bar j})
+e^{i_{\bar\phi}}((I^{\prime\prime}-\bar\phi\cdot\phi)^{-1}\lrcorner\theta^{\bar j})\\
=&\,e^{i_{\phi}|i_{\bar\phi}}(-\bar\phi\cdot(I^{\prime}-\phi\cdot\bar\phi)^{-1}\lrcorner\theta^{\bar j}+(I^{\prime\prime}-\bar\phi\cdot\phi)^{-1}\lrcorner\theta^{\bar j})\\
=&\,e^{i_{\phi}|i_{\bar\phi}}\bigl((-\bar\phi\cdot(I^{\prime}-\phi\cdot\bar\phi)^{-1}
+(I^{\prime\prime}-\bar\phi\cdot\phi)^{-1})\lrcorner\,\theta^{\bar j}\bigr).
\end{align*}
\end{proof}
As in \cite[p.3003, line -2]{rao2018several}, via $\Finv\,$, the operator $e^{i_{\phi}|i_{\bar\phi}}$ could be expressed as
\begin{equation*}
e^{i_{\phi}|i_{\bar\phi}}=(I+\phi+\bar\phi)\Finv\, .
\end{equation*}
\begin{lemma}\label{lemma12}\cite[Lemma 2.12]{rao2018several}
\begin{equation*}
e^{-i_{\phi}}\circ e^{i_{\phi}|i_{\bar\phi}}=(I-\bar\phi\cdot\phi+\bar\phi)\Finv\,,
\end{equation*}
\begin{equation*}
e^{-i_{\phi}|-i_{\bar\phi}}\circ e^{i_{\phi}}
=(I^{\prime}+(I^{\prime\prime}-\bar\phi\cdot\phi)^{-1}
-\bar\phi\cdot(I^{\prime}-\phi\cdot\bar\phi)^{-1})\Finv\,.
\end{equation*}
\end{lemma}
\begin{proof}
For any $\alpha\in A^{p,q}_{J}(M)$,
\begin{align*}
&e^{i_{\phi}}((I-\bar\phi\cdot\phi+\bar\phi)\Finv\,\alpha)\\
=&\,e^{i_{\phi}}(\alpha_{I_{p},\bar J_{ q}}\cdots\wedge(I-\bar\phi\cdot\phi+\bar\phi)\lrcorner\theta^{i_{p}}\wedge\cdots\wedge (I-\bar\phi\cdot\phi+\bar\phi)\lrcorner\theta^{\bar j_{q}}\wedge\cdots)\\
=&\,\alpha_{I_{p},\bar J_{q}}\cdots\wedge e^{i_{\phi}}((I-\bar\phi\cdot\phi+\bar\phi)\lrcorner\theta^{i_{p}})\wedge\cdots\wedge e^{i_{\phi}}\bigl((I-\bar\phi\cdot\phi+\bar\phi)\lrcorner\theta^{\bar j_{q}}\bigr)\wedge\cdots
\quad (\textup{by Lemma}\,\, \ref{lemma6})\\
=&\,\alpha_{I_{p},\bar J_{q}}\cdots\wedge e^{i_{\phi}}(\theta^{i_{p}})\wedge\cdots
\wedge e^{i_{\bar\phi}}(\theta^{\bar j_{q}})\wedge\cdots\\
=&\,e^{i_{\phi}|i_{\bar\phi}}(\alpha).
\end{align*}
By Lemma \ref{lemma9}, we have
\begin{align*}
e^{i_{\phi}}(\alpha)
=&\,\alpha_{I_{p},\bar J_{q}}e^{i_{\phi}}(\theta^{I_{p}})\wedge\theta^{\bar J_{q}}\\
=&\,\alpha_{I_{p},\bar J_{q}}e^{i_{\phi}}(\theta^{I_{p}})\wedge\cdots
\wedge(e^{i_{\phi}|i_{\bar\phi}}(((I^{\prime\prime}-\bar\phi\cdot\phi)^{-1}
-\bar\phi(I-\phi\cdot\bar\phi)^{-1})\lrcorner\theta^{\bar j_{q}}))\wedge\cdots
\quad (\textup{by Lemma}\,\, \ref{lemma9})\\
=&\,\alpha_{I_{p},\bar J_{q}}e^{i_{\phi}}(\theta^{I_{p}})
\wedge e^{i_{\phi}|i_{\phi}}(((I^{\prime\prime}-\bar\phi\cdot\phi)^{-1}
-\bar\phi\cdot(I^{\prime}-\phi\cdot\bar\phi)^{-1})\Finv\,\theta^{\bar J_{q}})\\
=&\,\alpha_{I_{p},\bar J_{q}}e^{i_{\phi}|i_{\bar\phi}}(\theta^{I_{p}}\wedge
((I^{\prime\prime}-\bar\phi\cdot\phi)^{-1}
-\bar\phi\cdot(I^{\prime}-\phi\cdot\bar\phi)^{-1})\Finv\,\theta^{\bar J_{q}}).
\end{align*}
Thus we get
\begin{align*}
e^{-i_{\phi}|-i_{\bar\phi}}\circ e^{i_{\phi}}(\alpha)
=&\,\alpha_{I_{p},\bar J_{q}}\theta^{I_{p}}\wedge
((I^{\prime\prime}-\bar\phi\cdot\phi)^{-1}-\bar\phi\cdot(I^{\prime}-\phi\cdot\bar\phi)^{-1})\Finv\,\theta^{\bar J_{q}}\\
=&\,\alpha_{I_{p},\bar J_{q}}(I^{\prime}\Finv\,\theta^{I_{p}})\wedge
(I^{\prime}+(I^{\prime\prime}-\bar\phi\cdot\phi)^{-1}-\bar\phi\cdot(I^{\prime}-\phi\cdot\bar\phi)^{-1})\Finv\,\theta^{\bar J_{q}}\\
=&\,\alpha_{I_{p},\bar J_{q}}
(I^{\prime}+(I^{\prime\prime}-\bar\phi\cdot\phi)^{-1}-\bar\phi\cdot(I^{\prime}-\phi\cdot\bar\phi)^{-1})\Finv\,
(\theta^{I_{p}}\wedge\theta^{\bar J_{q}})\\
=&\,(I^{\prime}+(I^{\prime\prime}-\bar\phi\cdot\phi)^{-1}-\bar\phi\cdot(I^{\prime}-\phi\cdot\bar\phi)^{-1})\Finv\,\alpha.
\end{align*}
\end{proof}
Using Lemma \ref{lemma6} and Lemma \ref{lemma12}, we get the following theorem.
\begin{theorem}\label{theorem3}
Let $e^{i_{\phi}|i_{\bar\phi}}$ be the extended exponential operator. For any $\alpha\in A^{p,q}_{J}(M)$, we have
\begin{align}\label{eq65}
d(e^{i_{\phi}|i_{\bar\phi}}(\alpha))
=&\,e^{i_{\phi}|i_{\bar\phi}}
\bigl\{((I^{\prime}+(I^{\prime\prime}-\bar\phi\cdot\phi)^{-1}-\bar\phi\cdot(I^{\prime}-\phi\cdot\bar\phi)^{-1})\nonumber\\
&\Finv\,\bigl(d+[\mu,i_{\phi}]+[\partial,i_{\phi}]-i_{\frac{1}{2}(\mathscr{B}(\phi,\phi)+\mathscr{C}(\phi,\phi))}+i_{\mathrm{MC}(\phi)}\bigr)
\circ(I-\bar\phi\cdot\phi+\bar\phi)\Finv\,\alpha\bigr\}.
\end{align}
\end{theorem}
\begin{proof}
Considering (\ref{eq89}), we check by direct calculations.
\begin{align}
&d(e^{i_{\phi}|i_{\bar\phi}}(\alpha))\nonumber\\
=&\,d\circ e^{i_{\phi}}\circ e^{-i_{\phi}}\circ e^{i_{\phi}|i_{\bar\phi}}(\alpha)\nonumber\\
=&\,e^{i_{\phi}}\circ (d+[\mu,i_{\phi}]+[\partial,i_{\phi}]-i_{\frac{1}{2}(\mathscr{B}(\phi,\phi)+\mathscr{C}(\phi,\phi))}
+i_{\mathrm{MC}(\phi)})\circ e^{-i_{\phi}}\circ e^{i_{\phi}|i_{\bar\phi}}(\alpha)\nonumber\\
=&\,e^{i_{\phi}|i_{\bar\phi}}\circ e^{-i_{\phi}|-i_{\bar\phi}}\circ e^{i_{\phi}}\circ (d+[\mu,i_{\phi}]+[\partial,i_{\phi}]-i_{\frac{1}{2}(\mathscr{B}(\phi,\phi)+\mathscr{C}(\phi,\phi))}
+i_{\mathrm{MC}(\phi)})\circ e^{-i_{\phi}}\circ e^{i_{\phi}|i_{\bar\phi}}(\alpha)\nonumber\\
=&\,e^{i_{\phi}|i_{\bar\phi}}
((I^{\prime}+(I^{\prime\prime}-\bar\phi\cdot\phi)^{-1}-\bar\phi\cdot(I^{\prime}-\phi\cdot\bar\phi)^{-1})\Finv\,
(d+[\mu,i_{\phi}]+[\partial,i_{\phi}]-i_{\frac{1}{2}(\mathscr{B}(\phi,\phi)+\mathscr{C}(\phi,\phi))}
+i_{\mathrm{MC}(\phi)})\nonumber\\
&\circ(I-\bar\phi\cdot\phi+\bar\phi)\Finv\,\alpha).\quad(\textup{By Lemma \ref{lemma12}})\nonumber
\end{align}
\end{proof}
\section{Decomposition of the extension formula}\label{sec4}
\subsection{Decomposition type one}
Note that (\ref{eq52})
\begin{equation*}
[\nabla,i_{\varphi}]\circ i_{\psi}=i_{\psi}\circ[\nabla,i_{\varphi}]-i_{[\varphi,\psi]},
\end{equation*}
could be decomposed via types into the following four equations.
\begin{align*}
[\mu,i_{\varphi}]\circ i_{\psi}&=i_{\psi}\circ[\mu,i_{\varphi}]-i_{\mathscr{B}(\varphi,\psi)+\mathscr{C}(\varphi,\psi)},\\
[\nabla^{1,0},i_{\varphi}]\circ i_{\psi}&=i_{\psi}\circ[\nabla^{1,0},i_{\varphi}]-i_{\mathscr{A}(\varphi,\psi)},\\
[\nabla^{0,1},i_{\varphi}]\circ i_{\psi}&=i_{\psi}\circ[\nabla^{0,1},i_{\varphi}],\\
[\bar\mu,i_{\varphi}]\circ i_{\psi}&=i_{\psi}\circ[\bar\mu,i_{\varphi}].
\end{align*}
Using these equations and calculating as the former section, we can prove the following lemma.
\begin{lemma}\label{lemma8}
Let $(M,J)$ be an almost complex manifold, $\phi\in A^{0,1}_{J}(M,T^{1,0}M)$ be a Beltrami differential, $E\rightarrow M$ be a vector bundle on $M$ and $\nabla$ be a connection of $E$. Then on $A^{\ast,\ast}_{J}(M,E)$ we have
\begin{align*}
[\mu,i^{k}_{\phi}]=&\,ki^{k-1}_{\phi}\circ[\mu,i_{\phi}]
-\frac{k(k-1)}{2}i^{k-2}_{\phi}\circ i_{(\mathscr{B}(\phi,\phi)+\mathscr{C}(\phi,\phi))}\\
&-\frac{k(k-1)(k-2)}{3!}i^{k-3}_{\phi}\circ
i_{(i_{\mathscr{C}(\phi,\phi)}\phi-i_{\phi}\mathscr{B}(\phi,\phi))},\\
[\nabla^{1,0},i^{k}_{\phi}]
=&\,ki^{k-1}_{\phi}\circ[\nabla^{1,0},i_{\phi}]
-\frac{k(k-1)}{2}i^{k-2}_{\phi}\circ i_{\mathscr{A}(\phi,\phi)},\\
[\nabla^{0,1},i^{k}_{\phi}]
=&\,ki^{k-1}_{\phi}\circ[\nabla^{0,1},i_{\phi}],\\
[\bar\mu,i^{k}_{\phi}]
=&\,ki^{k-1}_{\phi}\circ[\bar\mu,i_{\phi}].
\end{align*}
\end{lemma}
\begin{proof}
 By Lemma \ref{lemma5}, Lemma \ref{lemma11}, the fact that
\begin{equation*}
\mathscr{A}(\phi,\phi)\in A_{J}^{0,2}(M,T^{1,0}M),
\end{equation*}
 and
\begin{equation*}
\mathscr{B}(\phi,\phi)\in A^{1,1}_{J}(M,T^{1,0}M),\quad \mathscr{C}(\phi,\phi)\in A^{0,2}_{J}(M,T^{0,1}M),
\end{equation*}
we have
\begin{equation}\label{eq61}
i_{\mathscr{A}(\phi,\phi)}\circ i_{\phi}=i_{\phi}\circ i_{\mathscr{A}(\phi,\phi)},
\end{equation}
and
\begin{align}\label{eq62}
i_{(\mathscr{B}(\phi,\phi)+\mathscr{C}(\phi,\phi))}\circ i_{\phi}-i_{\phi}\circ i_{(\mathscr{B}(\phi,\phi)+\mathscr{C}(\phi,\phi))}
=&\,i_{i_{(\mathscr{B}(\phi,\phi)+\mathscr{C}(\phi,\phi))}\phi
-i_{\phi}(\mathscr{B}(\phi,\phi)+\mathscr{C}(\phi,\phi))}\nonumber\\
=&\,i_{(i_{\mathscr{C}(\phi,\phi)}\phi-i_{\phi}\mathscr{B}(\phi,\phi))}.
\end{align}
By Lemma \ref{lemma3} we have
\begin{equation}\label{eq63}
i_{(i_{\mathscr{C}(\phi,\phi)}\phi-i_{\phi}\mathscr{B}(\phi,\phi))}\circ i_{\phi}
=i_{\phi}\circ i_{(i_{\mathscr{C}(\phi,\phi)}\phi-i_{\phi}\mathscr{B}(\phi,\phi))}.
\end{equation}
Comparing the proof of Proposition \ref{proposition1}, the key point of the proof is the commutators (\ref{eq61}), (\ref{eq62}) and (\ref{eq63}). Following the proof of Proposition \ref{proposition1}, we can prove the lemma.
\end{proof}
Using Lemma \ref{lemma8}, we can prove the following theorem, which is a decomposition of (\ref{eq64}) according to types.
\begin{theorem}\label{theorem4}
Let $(M,J)$ be an almost complex manifold, $\phi\in A^{0,1}_{J}(M,T^{1,0}M)$ be a Beltrami differential, $E\rightarrow M$ be a vector bundle on $M$ and $\nabla$ be a connection of $E$. We have
\begin{align*}
e^{-i_{\phi}}\circ\mu\circ e^{i_{\phi}}
&=\mu-\mathcal{L}^{\mu}_{\phi}-i_{\frac{1}{2}(\mathscr{B}(\phi,\phi)+\mathscr{C}(\phi,\phi))}
-i_{\frac{1}{3!}(i_{\mathscr{C}(\phi,\phi)}\phi-i_{\phi}\mathscr{B}(\phi,\phi))},\\
e^{-i_{\phi}}\circ\nabla^{1,0}\circ e^{i_{\phi}}
&=\nabla^{1,0}-\mathcal{L}^{\nabla^{1,0}}_{\phi}-i_{\frac{1}{2}\mathscr{A}(\phi,\phi)},\\
e^{-i_{\phi}}\circ\nabla^{0,1}\circ e^{i_{\phi}}
&=\nabla^{0,1}-\mathcal{L}^{\nabla^{0,1}}_{\phi},\\
e^{-i_{\phi}}\circ\bar\mu\circ e^{i_{\phi}}
&=\bar\mu-\mathcal{L}^{\bar\mu}_{\phi}.
\end{align*}
\end{theorem}
For the exterior differential $d$ we also have the following corollary.
\begin{corollary}\label{corollary5}
Let $(M,J)$ be an almost complex manifold and $\phi\in A_{J}^{0,1}(M,T^{1,0}M)$ be a Beltrami differential. We have
\begin{align*}
e^{-i_{\phi}}\circ\mu\circ e^{i_{\phi}}
&=\mu-\mathcal{L}^{\mu}_{\phi}-i_{\frac{1}{2}(\mathscr{B}(\phi,\phi)+\mathscr{C}(\phi,\phi))}
-i_{\frac{1}{3!}(i_{\mathscr{C}(\phi,\phi)}\phi-i_{\phi}\mathscr{B}(\phi,\phi))},\\
e^{-i_{\phi}}\circ\partial\circ e^{i_{\phi}}
&=\partial-\mathcal{L}^{\partial}_{\phi}-i_{\frac{1}{2}\mathscr{A}(\phi,\phi)},\\
e^{-i_{\phi}}\circ\bar\partial\circ e^{i_{\phi}}
&=\bar\partial-\mathcal{L}^{\bar\partial}_{\phi},\\
e^{-i_{\phi}}\circ\bar\mu\circ e^{i_{\phi}}
&=\bar\mu-\mathcal{L}^{\bar\mu}_{\phi}.
\end{align*}
\end{corollary}
\subsection{Decomposition type two}
In this subsection, we decompose (\ref{eq65}) according to the types of both sides. One should keep in mind Remark \ref{remark1} and the fact that contractions are pointwisely linear.
\begin{equation*}
\bar\phi\cdot(I^{\prime}-\phi\cdot\bar\phi)^{-1}\in A_{J}^{1,0}(M,T^{0,1}M),\quad
(I^{\prime\prime}-\bar\phi\cdot\phi)^{-1}\in A_{J}^{0,1}(M,T^{0,1}M).
\end{equation*}
Denote
\begin{align*}
O_{1}&=I^{\prime}+(I^{\prime\prime}-\bar\phi\cdot\phi)^{-1}-\bar\phi\cdot(I^{\prime}-\phi\cdot\bar\phi)^{-1},\\
O_{2}&=d+[\mu,i_{\phi}]+[\partial,i_{\phi}]-i_{\frac{1}{2}(\mathscr{B}(\phi,\phi)+\mathscr{C}(\phi,\phi))}+i_{\mathrm{MC}(\phi)},\\
O_{3}&=I-\bar\phi\cdot\phi+\bar\phi.
\end{align*}
We list the types of the operators in Table \ref{tub1}. The leftmost column in the tabulation is the type of the corresponding operator.
\renewcommand\arraystretch{2} %set the height of the rows of the table
\begin{table}\footnotesize\tabcolsep 16pt
\centering
\caption{Types of operators}
\begin{tabular}{|c|c|c|c|c|}
  \hline
  % after \\: \hline or \cline{col1-col2} \cline{col3-col4} ...
  Type &$O_{1}$ & $O_{2}$ & $O_{3}$ & Total type  \\
  \hline
  $(2,-1)$&                      &$\mu$                                                  &                      &$(3,-2)$\\
  \hline
  $(1,0)$ &                      &$\partial+[\mu,i_{\phi}]$                                             &                      &$(2,-1)$\\
  \hline
  $(1,-1)$&$-\bar\phi\cdot(I^{\prime}-\phi\cdot\bar\phi)^{-1}$ &                  &  $\bar\phi$                     &$(1,0)$  \\
  \hline
  $(0,1)$ &                      &$\bar\partial+[\partial,i_{\phi}]-i_{\frac{1}{2}(\mathscr{B}(\phi,\phi)+\mathscr{C}(\phi,\phi))}$    &&$(0,1)$\\
  \hline
  $(0,0)$ &$I^{\prime}+(I^{\prime\prime}-\bar\phi\cdot\phi)^{-1}$ &                           & $I-\bar\phi\cdot\phi$                &$(-1,2)$ \\
  \hline
  $(-1,2)$&                      &$\bar\mu+i_{\mathrm{MC}(\phi)}$ & & $(4,-3)$\\
  \hline
\end{tabular}
\label{tub1}
\end{table}
The rightmost column in the tabulation lists all the possible types of the composition operator $O_{1}\circ O_{2}\circ O_{3}$. Considering Remark \ref{remark1}, we should be careful with the types of the decomposition of $O_{1}\circ O_{2}\circ O_{3}$, or similarly of
\begin{equation*}
d(e^{i_{\phi}|i_{\bar\phi}}(\alpha))
=e^{i_{\phi}|i_{\bar\phi}}(O_{1}\Finv\,(O_{2}(O_{3}\Finv\,\alpha))).
\end{equation*}
Since $I\in A^{0,1}(M,T^{0,1})\oplus A^{1,0}(M,T^{1,0})$, $\bar\phi\cdot\phi\in A^{0,1}(M,T^{0,1})$ and $\bar\phi\in A^{1,0}(M,T^{0,1})$, we know that
\begin{small}
\begin{align*}
&(I-\bar\phi\cdot\phi+\bar\phi)\Finv\,\alpha\\
=&\,\alpha_{i_{1}\ldots i_{p}\bar j_{1}\ldots \bar j_{q}}(I-\bar\phi\cdot\phi+\bar\phi)\lrcorner\theta^{i_{1}}\wedge\cdots\wedge
(I-\bar\phi\cdot\phi+\bar\phi)\lrcorner\theta^{i_{p}}\wedge(I-\bar\phi\cdot\phi+\bar\phi)\lrcorner\theta^{\bar j_{1}}\wedge\cdots\wedge(I-\bar\phi\cdot\phi+\bar\phi)\lrcorner\theta^{\bar j_{q}}\\
=&\,\alpha_{i_{1}\ldots i_{p}\bar j_{1}\ldots \bar j_{q}}\theta^{i_{1}}\wedge\cdots\wedge
\theta^{i_{p}}\wedge(I-\bar\phi\cdot\phi+\bar\phi)\lrcorner\theta^{\bar j_{1}}\wedge\cdots\wedge(I-\bar\phi\cdot\phi+\bar\phi)\lrcorner\theta^{\bar j_{q}}\\
=&\,\alpha_{i_{1}\ldots i_{p}\bar j_{1}\ldots \bar j_{q}}\theta^{i_{1}}\wedge\cdots\wedge\theta^{i_{p}}
\wedge((I-\bar\phi\cdot\phi)^{\bar j_{1}}_{\bar \ell_{1}}\theta^{\bar\ell_{1}}+\phi^{\bar j_{1}}_{k_{1}}\theta^{k_{1}})\wedge\cdots\wedge
((I-\bar\phi\cdot\phi)^{\bar j_{q}}_{\bar \ell_{q}}\theta^{\bar\ell_{q}}+\phi^{\bar j_{q}}_{k_{q}}\theta^{k_{q}}).
\end{align*}
\end{small}
So we get
\begin{equation*}
(I-\bar\phi\cdot\phi+\bar\phi)\Finv\,\alpha
=\sum_{s=0}^{q}P_{J}^{p+s,q-s}((I-\bar\phi\cdot\phi+\bar\phi)\Finv\,\alpha),
\end{equation*}
where
\begin{align*}
&P_{J}^{p+s,q-s}((I-\bar\phi\cdot\phi+\bar\phi)\Finv\,\alpha)\\
=&\alpha_{i_{1}\ldots i_{p}\bar j_{1}\ldots \bar j_{q}}\sum_{1\leq r_{1}<\cdots r_{s}\leq q}\delta^{1\cdots q}_{r_{1}\cdots r_{s}r^{c}_{1}\cdots r^{c}_{q-s}}
\phi^{\bar j_{r_{1}}}_{k_{1}}\cdots\phi^{\bar j_{r_{s}}}_{k_{s}}(I-\bar\phi\cdot\phi)^{\bar j_{r^{c}_{1}}}_{\bar\ell_{q-s}}\cdots(I-\bar\phi\cdot\phi)^{\bar j_{r^{c}_{q-s}}}_{\bar\ell_{q}}\\
&\theta^{i_{1}}\wedge\cdots\wedge\theta^{i_{p}}\wedge
\theta^{k_{1}}\wedge\cdots\wedge\theta^{k_{s}}\wedge\theta^{\bar\ell_{q-s}}\wedge\cdots\wedge\theta^{\bar\ell_{q}},
\end{align*}
where $r^{c}_{1}<\ldots<r^{c}_{q-s}$ and $\{r_{1},\cdots,r_{s},r^{c}_{1},\cdots r^{c}_{q-s}\}=\{1,\ldots,q\}$.
Now we need to calculate
\begin{equation*}
(I^{\prime}+(I^{\prime\prime}-\bar\phi\cdot\phi)^{-1}-\bar\phi\cdot(I^{\prime}-\phi\cdot\bar\phi)^{-1}\bigr)\Finv\,.
\end{equation*}
Note that
\begin{equation*}
(I^{\prime\prime}-\bar\phi\cdot\phi)^{-1}\in A_{J}^{0,1}(M,T^{0,1}),\quad
\bar\phi\cdot(I^{\prime}-\phi\cdot\bar\phi)^{-1}\in A_{J}^{1,0}(M,T^{0,1}).
\end{equation*}
Formally we have
\begin{align*}
(I^{\prime}+(I^{\prime\prime}-\bar\phi\cdot\phi)^{-1}-\bar\phi\cdot(I^{\prime}-\phi\cdot\bar\phi)^{-1})\Finv\,\alpha
=\sum_{s=0}^{q}P_{J}^{p+s,q-s}((I^{\prime}+(I^{\prime\prime}-\bar\phi\cdot\phi)^{-1}-\bar\phi\cdot(I^{\prime}-\phi\cdot\bar\phi)^{-1})\Finv\,\alpha).
\end{align*}
The precise calculation is
\begin{small}
\begin{align*}
&(I^{\prime}+(I^{\prime\prime}-\bar\phi\cdot\phi)^{-1}-\bar\phi\cdot(I^{\prime}-\phi\cdot\bar\phi)^{-1})\Finv\,\alpha\\
=&\,\alpha_{i_{1}\ldots i_{p}\bar j_{1}\ldots\bar j_{q}}
(I^{\prime}+(I^{\prime\prime}-\bar\phi\cdot\phi)^{-1}-\bar\phi\cdot(I^{\prime}-\phi\cdot\bar\phi)^{-1})\lrcorner\theta^{i_{1}}\wedge\cdots
\wedge(I^{\prime}+(I^{\prime\prime}-\bar\phi\cdot\phi)^{-1}-\bar\phi\cdot(I^{\prime}-\phi\cdot\bar\phi)^{-1})\lrcorner\theta^{i_{p}}\\
&\wedge(I^{\prime}+(I^{\prime\prime}-\bar\phi\cdot\phi)^{-1}-\bar\phi\cdot(I^{\prime}-\phi\cdot\bar\phi)^{-1})
\lrcorner\,\theta^{\bar j_{1}}
\wedge\cdots\wedge(I^{\prime}+(I^{\prime\prime}-\bar\phi\cdot\phi)^{-1}-\bar\phi\cdot(I^{\prime}-\phi\cdot\bar\phi)^{-1})
\lrcorner\,\theta^{\bar j_{q}}\\
=&\,\alpha_{i_{1}\ldots i_{p}\bar j_{1}\ldots\bar j_{q}}
\theta^{i_{1}}\wedge\cdots\wedge\theta^{i_{p}}\\
&\wedge(((I^{\prime\prime}-\bar\phi\cdot\phi)^{-1})^{\bar j_{1}}_{\bar \ell_{1}}\theta^{\bar\ell_{1}}
-(\bar\phi\cdot(I^{\prime}-\phi\cdot\bar\phi)^{-1})^{\bar j_{1}}_{k_{1}}\theta^{k_{1}})
\wedge\cdots\wedge
(((I^{\prime\prime}-\bar\phi\cdot\phi)^{-1})^{\bar j_{q}}_{\bar \ell_{q}}\theta^{\bar\ell_{q}}
-(\bar\phi\cdot(I^{\prime}-\phi\cdot\bar\phi)^{-1})^{\bar j_{q}}_{k_{q}}\theta^{k_{q}}),\\
\end{align*}
\end{small}
and
\begin{align*}
&P_{J}^{p+s,q-s}((I^{\prime}+(I^{\prime\prime}-\bar\phi\cdot\phi)^{-1}-\bar\phi\cdot(I^{\prime}-\phi\cdot\bar\phi)^{-1})\Finv\,\alpha)\\
=&\,(-1)^{s}\alpha_{i_{1}\ldots i_{p}\bar j_{1}\ldots \bar j_{q}}\sum_{1\leq r_{1}<\cdots<r_{s}\leq q}\delta^{1\cdots q}_{r_{1}\cdots r_{s}r^{c}_{1}\cdots r^{c}_{q-s}}
(\bar\phi\cdot(I^{\prime}-\phi\cdot\bar\phi)^{-1})^{\bar j_{r_{1}}}_{k_{1}}\cdots
(\bar\phi\cdot(I^{\prime}-\phi\cdot\bar\phi)^{-1})^{\bar j_{r_{s}}}_{k_{s}}\\
&\cdot(I^{\prime\prime}-\bar\phi\cdot\phi)^{\bar j_{r^{c}_{1}}}_{\bar\ell_{q-s}}
\cdots(I^{\prime\prime}-\bar\phi\cdot\phi)^{\bar j_{r^{c}_{q-s}}}_{\bar\ell_{q}}
\theta^{i_{1}}\wedge\cdots\wedge\theta^{i_{p}}\wedge
\theta^{k_{1}}\wedge\cdots\wedge\theta^{k_{s}}\wedge\theta^{\bar\ell_{q-s}}\wedge\cdots\wedge\theta^{\bar\ell_{q}}.
\end{align*}
where $r^{c}_{1}<\ldots<r^{c}_{q-s}$ and $\{r_{1},\cdots,r_{s},r^{c}_{1},\cdots r^{c}_{q-s}\}=\{1,\ldots,q\}$. Now we can compute $O_{1}\circ O_{2}\circ O_{3}$.
\begin{small}
\begin{align*}
&O_{1}\circ O_{2}\circ O_{3}(\alpha)\\
=&\,(I^{\prime}+(I^{\prime\prime}-\bar\phi\cdot\phi)^{-1}-\bar\phi\cdot(I^{\prime}-\phi\cdot\bar\phi)^{-1})\Finv\,
\bigl(d+[\mu,i_{\phi}]+[\partial,i_{\phi}]-i_{\frac{1}{2}(\mathscr{B}(\phi,\phi)+\mathscr{C}(\phi,\phi))}+i_{\mathrm{MC}(\phi)}\bigr)
\circ(I-\bar\phi\cdot\phi+\bar\phi)\Finv\,\alpha\\
=&\,(I^{\prime}+(I^{\prime\prime}-\bar\phi\cdot\phi)^{-1}-\bar\phi\cdot(I^{\prime}-\phi\cdot\bar\phi)^{-1})\Finv\,
\bigl(\underbrace{\mu}_{2,-1}+\underbrace{\partial+[\mu,i_{\phi}]}_{1,0}
+\underbrace{\bar\partial+[\partial,i_{\phi}]-i_{\frac{1}{2}(\mathscr{B}(\phi,\phi)+\mathscr{C}(\phi,\phi))}}_{0,1}
+\underbrace{\bar\mu+i_{\mathrm{MC}(\phi)}}_{-1,2}\bigr)\\
&\circ\sum_{s=0}^{q}P_{J}^{p+s,q-s}((I-\bar\phi\cdot\phi+\bar\phi)\Finv\,\alpha)\\
=&\,(I^{\prime}+(I^{\prime\prime}-\bar\phi\cdot\phi)^{-1}-\bar\phi\cdot(I^{\prime}-\phi\cdot\bar\phi)^{-1})\Finv
\underbrace{\mu}_{2,-1}
\sum_{s=0}^{q}P_{J}^{p+s,q-s}((I^{\prime\prime}-\bar\phi\cdot\phi+\bar\phi)\Finv\,\alpha)\\
&+(I^{\prime}+(I^{\prime\prime}-\bar\phi\cdot\phi)^{-1}-\bar\phi\cdot(I^{\prime}-\phi\cdot\bar\phi)^{-1})\Finv
\underbrace{\partial+[\mu,i_{\phi}]}_{1,0}
\sum_{s=0}^{q}P_{J}^{p+s,q-s}((I-\bar\phi\cdot\phi+\bar\phi)\Finv\,\alpha)\\
&+(I^{\prime}+(I^{\prime\prime}-\bar\phi\cdot\phi)^{-1}-\bar\phi\cdot(I-\phi\cdot\bar\phi)^{-1})\Finv\,
(\underbrace{\bar\partial+[\partial,i_{\phi}]-i_{\frac{1}{2}(\mathscr{B}(\phi,\phi)+\mathscr{C}(\phi,\phi))}}_{0,1})
\sum_{s=0}^{q}P_{J}^{p+s,q-s}((I-\bar\phi\cdot\phi+\bar\phi)\Finv\,\alpha)\\
&+(I^{\prime}+(I^{\prime\prime}-\bar\phi\cdot\phi)^{-1}-\bar\phi\cdot(I^{\prime}-\phi\cdot\bar\phi)^{-1})\Finv
(\underbrace{\bar\mu+i_{\mathrm{MC}(\phi)}}_{-1,2})
\sum_{s=0}^{q}P_{J}^{p+s,q-s}((I-\bar\phi\cdot\phi+\bar\phi)\Finv\,\alpha)\\
=&\,\sum_{s=0}^{q}(I^{\prime}+(I^{\prime\prime}-\bar\phi\cdot\phi)^{-1}-\bar\phi\cdot(I^{\prime}-\phi\cdot\bar\phi)^{-1})\Finv
\underbrace{\mu P_{J}^{p+s,q-s}((I-\bar\phi\cdot\phi+\bar\phi)\Finv\,\alpha)}_{(p+s+2,q-s-1)}\\
&+\sum_{s=0}^{q}(I^{\prime}+(I^{\prime\prime}-\bar\phi\cdot\phi)^{-1}-\bar\phi\cdot(I^{\prime}-\phi\cdot\bar\phi)^{-1})\Finv
\underbrace{(\partial+[\mu,i_{\phi}]) P_{J}^{p+s,q-s}((I-\bar\phi\cdot\phi+\bar\phi)\Finv\,\alpha)}_{(p+s+1,q-s)}\\
&+\sum_{s=0}^{q}(I^{\prime}+(I^{\prime\prime}-\bar\phi\cdot\phi)^{-1}-\bar\phi\cdot(I^{\prime}-\phi\cdot\bar\phi)^{-1})\Finv
\underbrace{(\bar\partial+[\partial,i_{\phi}]-i_{\frac{1}{2}(\mathscr{B}(\phi,\phi)+\mathscr{C}(\phi,\phi))})
P_{J}^{p+s,q-s}((I-\bar\phi\cdot\phi+\bar\phi)\Finv\,\alpha)}_{p+s,q-s+1}\\
&+\sum_{s=0}^{q}(I^{\prime}+(I^{\prime\prime}-\bar\phi\cdot\phi)^{-1}-\bar\phi\cdot(I^{\prime}-\phi\cdot\bar\phi)^{-1})\Finv
\underbrace{(\bar\mu+i_{\mathrm{MC}(\phi)})
P_{J}^{p+s,q-s}((I-\bar\phi\cdot\phi+\bar\phi)\Finv\,\alpha)}_{p+s-1,q-s+2}\\
=&\,\sum_{s=0}^{q}\sum_{t=0}^{q-s-1}P_{J}^{p+s+2+t,q-s-1-t}(I^{\prime}+(I^{\prime\prime}-\bar\phi\cdot\phi)^{-1}
-\bar\phi\cdot(I^{\prime}-\phi\cdot\bar\phi)^{-1})\Finv
\underbrace{\mu P_{J}^{p+s,q-s}((I-\bar\phi\cdot\phi+\bar\phi)\Finv\,\alpha)}_{(p+s+2,q-s-1)}\\
&+\sum_{s=0}^{q}\sum_{t=0}^{q-s}P_{J}^{p+s+1+t,q-s-t}(I^{\prime}+(I^{\prime\prime}-\bar\phi\cdot\phi)^{-1}-\bar\phi\cdot(I^{\prime}-\phi\cdot\bar\phi)^{-1})\Finv
\underbrace{(\partial+[\mu,i_{\phi}]) P_{J}^{p+s,q-s}((I-\bar\phi\cdot\phi+\bar\phi)\Finv\,\alpha)}_{(p+s+1,q-s)}\\
&+\sum_{s=0}^{q}\sum_{t=0}^{q-s+1}P_{J}^{p+s+t,q-s+1-t}(I^{\prime}+(I^{\prime\prime}-\bar\phi\cdot\phi)^{-1}
-\bar\phi\cdot(I^{\prime}-\phi\cdot\bar\phi)^{-1})\Finv\\
&\underbrace{(\bar\partial+[\partial,i_{\phi}]-i_{\frac{1}{2}(\mathscr{B}(\phi,\phi)+\mathscr{C}(\phi,\phi))})
P_{J}^{p+s,q-s}((I-\bar\phi\cdot\phi+\bar\phi)\Finv\,\alpha)}_{p+s,q-s+1}\\
&+\sum_{s=0}^{q}\sum_{t=0}^{q-s+2}P_{J}^{p+s-1+t,q-s+2-t}(I^{\prime}+(I^{\prime\prime}-\bar\phi\cdot\phi)^{-1}
-\bar\phi\cdot(I^{\prime}-\phi\cdot\bar\phi)^{-1})\Finv\\
&\underbrace{(\bar\mu+i_{\mathrm{MC}(\phi)})
P_{J}^{p+s,q-s}((I-\bar\phi\cdot\phi+\bar\phi)\Finv\,\alpha)}_{p+s-1,q-s+2}.
\end{align*}
\end{small}
By the above equation, we know that the $J^{2,-1}$-term is
\begin{small}
\begin{align*}
&(O_{1}\circ O_{2}\circ O_{3}(\alpha))_{J}^{2,-1}\\
=&\,\sum_{s+t=0}P_{J}^{p+s+2+t,q-s-1-t}(I^{\prime}+(I^{\prime\prime}-\bar\phi\cdot\phi)^{-1}-\bar\phi\cdot(I^{\prime}-\phi\cdot\bar\phi)^{-1})\Finv\,
\underbrace{\mu P_{J}^{p+s,q-s}((I-\bar\phi\cdot\phi+\bar\phi)\Finv\,\alpha)}_{(p+s+2,q-s-1)}\\
&+\sum_{s+t=1}P_{J}^{p+s+1+t,q-s-t}(I^{\prime}+(I^{\prime\prime}-\bar\phi\cdot\phi)^{-1}-\bar\phi\cdot(I^{\prime}-\phi\cdot\bar\phi)^{-1})\Finv\,
\underbrace{(\partial+[\mu,i_{\phi}]) P_{J}^{p+s,q-s}((I-\bar\phi\cdot\phi+\bar\phi)\Finv\,\alpha)}_{(p+s+1,q-s)}\\
&+\sum_{s+t=2}P_{J}^{p+s+t,q-s+1-t}(I^{\prime}+(I^{\prime\prime}-\bar\phi\cdot\phi)^{-1}-\bar\phi\cdot(I^{\prime}-\phi\cdot\bar\phi)^{-1})\Finv\,
\underbrace{(\bar\partial+[\partial,i_{\phi}]-i_{\frac{1}{2}(\mathscr{B}(\phi,\phi)+\mathscr{C}(\phi,\phi))})
P_{J}^{p+s,q-s}((I-\bar\phi\cdot\phi+\bar\phi)\Finv\,\alpha)}_{p+s,q-s+1}\\
&+\sum_{s+t=3}P_{J}^{p+s-1+t,q-s+2-t}(I^{\prime}+(I^{\prime\prime}-\bar\phi\cdot\phi)^{-1}-\bar\phi\cdot(I^{\prime}-\phi\cdot\bar\phi)^{-1})\Finv\,
\underbrace{(\bar\mu+i_{\mathrm{MC}(\phi)})
P_{J}^{p+s,q-s}((I-\bar\phi\cdot\phi+\bar\phi)\Finv\,\alpha)}_{p+s-1,q-s+2}\bigr\}\\
=&\,\sum_{s+t=0}P_{J}^{p+2,q-1}(I^{\prime}+(I^{\prime\prime}-\bar\phi\cdot\phi)^{-1}-\bar\phi\cdot(I^{\prime}-\phi\cdot\bar\phi)^{-1})\Finv\,
\underbrace{\mu P_{J}^{p+s,q-s}((I-\bar\phi\cdot\phi+\bar\phi)\Finv\,\alpha)}_{(p+s+2,q-s-1)}\\
&+\sum_{s+t=1}P_{J}^{p+2,q-1}(I^{\prime}+(I^{\prime\prime}-\bar\phi\cdot\phi)^{-1}-\bar\phi\cdot(I^{\prime}-\phi\cdot\bar\phi)^{-1})\Finv\,
\underbrace{(\partial+[\mu,i_{\phi}]) P_{J}^{p+s,q-s}((I-\bar\phi\cdot\phi+\bar\phi)\Finv\,\alpha)}_{(p+s+1,q-s)}\\
&+\sum_{s+t=2}P_{J}^{p+2,q-1}(I^{\prime}+(I^{\prime\prime}-\bar\phi\cdot\phi)^{-1}-\bar\phi\cdot(I^{\prime}-\phi\cdot\bar\phi)^{-1})\Finv\,
\underbrace{(\bar\partial+[\partial,i_{\phi}]-i_{\frac{1}{2}(\mathscr{B}(\phi,\phi)+\mathscr{C}(\phi,\phi))})
P_{J}^{p+s,q-s}((I-\bar\phi\cdot\phi+\bar\phi)\Finv\,\alpha)}_{p+s,q-s+1}\\
&+\sum_{s+t=3}P_{J}^{p+2,q-1}(I^{\prime}+(I^{\prime\prime}-\bar\phi\cdot\phi)^{-1}-\bar\phi\cdot(I^{\prime}-\phi\cdot\bar\phi)^{-1})\Finv\,
\underbrace{(\bar\mu+i_{\mathrm{MC}(\phi)})
P_{J}^{p+s,q-s}((I-\bar\phi\cdot\phi+\bar\phi)\Finv\,\alpha)}_{p+s-1,q-s+2}\bigr\}\\
=&\,P_{J}^{p+2,q-1}(I^{\prime}+(I^{\prime\prime}-\bar\phi\cdot\phi)^{-1}-\bar\phi\cdot(I^{\prime}-\phi\cdot\bar\phi)^{-1})\Finv\,
\underbrace{\mu P_{J}^{p,q}((I-\bar\phi\cdot\phi+\bar\phi)\Finv\,\alpha)}_{(p+2,q-1)}\\
&+P_{J}^{p+2,q-1}(I^{\prime}+(I^{\prime\prime}-\bar\phi\cdot\phi)^{-1}-\bar\phi\cdot(I^{\prime}-\phi\cdot\bar\phi)^{-1})\Finv\,
\underbrace{(\partial+[\mu,i_{\phi}]) P_{J}^{p+1,q-1}((I-\bar\phi\cdot\phi+\bar\phi)\Finv\,\alpha)}_{(p+2,q-1)}\\
&+P_{J}^{p+2,q-1}(I^{\prime}+(I^{\prime\prime}-\bar\phi\cdot\phi)^{-1}-\bar\phi\cdot(I^{\prime}-\phi\cdot\bar\phi)^{-1})\Finv\,
\underbrace{(\partial+[\mu,i_{\phi}]) P_{J}^{p,q}((I-\bar\phi\cdot\phi+\bar\phi)\Finv\,\alpha)}_{(p+1,q)}\\
&+P_{J}^{p+2,q-1}(I^{\prime}+(I^{\prime\prime}-\bar\phi\cdot\phi)^{-1}-\bar\phi\cdot(I^{\prime}-\phi\cdot\bar\phi)^{-1})\Finv\,
\underbrace{(\bar\partial+[\partial,i_{\phi}]-i_{\frac{1}{2}(\mathscr{B}(\phi,\phi)+\mathscr{C}(\phi,\phi))})
P_{J}^{p+2,q-2}((I-\bar\phi\cdot\phi+\bar\phi)\Finv\,\alpha)}_{p+2,q-1}\\
&+P_{J}^{p+2,q-1}(I^{\prime}+(I^{\prime\prime}-\bar\phi\cdot\phi)^{-1}-\bar\phi\cdot(I^{\prime}-\phi\cdot\bar\phi)^{-1})\Finv\,
\underbrace{(\bar\partial+[\partial,i_{\phi}]-i_{\frac{1}{2}(\mathscr{B}(\phi,\phi)+\mathscr{C}(\phi,\phi))})
P_{J}^{p+1,q-1}((I-\bar\phi\cdot\phi+\bar\phi)\Finv\,\alpha)}_{p+1,q}\\
&+P_{J}^{p+2,q-1}(I^{\prime}+(I^{\prime\prime}-\bar\phi\cdot\phi)^{-1}-\bar\phi\cdot(I^{\prime}-\phi\cdot\bar\phi)^{-1})\Finv\,
\underbrace{(\bar\partial+[\partial,i_{\phi}]-i_{\frac{1}{2}(\mathscr{B}(\phi,\phi)+\mathscr{C}(\phi,\phi))})
P_{J}^{p,q}((I-\bar\phi\cdot\phi+\bar\phi)\Finv\,\alpha)}_{p,q+1}\\
&+P_{J}^{p+2,q-1}(I^{\prime}+(I^{\prime\prime}-\bar\phi\cdot\phi)^{-1}-\bar\phi\cdot(I^{\prime}-\phi\cdot\bar\phi)^{-1})\Finv\,
\underbrace{(\bar\mu+i_{\mathrm{MC}(\phi)})
P_{J}^{p+3,q-3}((I-\bar\phi\cdot\phi+\bar\phi)\Finv\,\alpha)}_{p+2,q}\\
&+P_{J}^{p+2,q-1}(I^{\prime}+(I^{\prime\prime}-\bar\phi\cdot\phi)^{-1}-\bar\phi\cdot(I^{\prime}-\phi\cdot\bar\phi)^{-1})\Finv\,
\underbrace{(\bar\mu+i_{\mathrm{MC}(\phi)})
P_{J}^{p+2,q-2}((I-\bar\phi\cdot\phi+\bar\phi)\Finv\,\alpha)}_{p+1,q}\\
&+P_{J}^{p+2,q-1}(I^{\prime}+(I^{\prime\prime}-\bar\phi\cdot\phi)^{-1}-\bar\phi\cdot(I^{\prime}-\phi\cdot\bar\phi)^{-1})\Finv\,
\underbrace{(\bar\mu+i_{\mathrm{MC}(\phi)})
P_{J}^{p+1,q-1}((I-\bar\phi\cdot\phi+\bar\phi)\Finv\,\alpha)}_{p,q+1}\\
&+P_{J}^{p+2,q-1}(I^{\prime}+(I^{\prime\prime}-\bar\phi\cdot\phi)^{-1}-\bar\phi\cdot(I^{\prime}-\phi\cdot\bar\phi)^{-1})\Finv\,
\underbrace{(\bar\mu+i_{\mathrm{MC}(\phi)})
P_{J}^{p,q}((I-\bar\phi\cdot\phi+\bar\phi)\Finv\,\alpha)}_{p-1,q+2}.
\end{align*}
\end{small}
The $J^{1,0}$-term is
\begin{small}
\begin{align*}
&(O_{1}\circ O_{2}\circ O_{3}(\alpha))_{J}^{1,0}\\
=&\,\sum_{s+t=0}P_{J}^{p+s+1+t,q-s-t}(I^{\prime}+(I^{\prime\prime}-\bar\phi\cdot\phi)^{-1}-\bar\phi\cdot(I^{\prime}-\phi\cdot\bar\phi)^{-1})\Finv\,
\underbrace{(\partial+[\mu,i_{\phi}]) P_{J}^{p+s,q-s}((I-\bar\phi\cdot\phi+\bar\phi)\Finv\,\alpha)}_{(p+s+1,q-s)}\\
&+\sum_{s+t=1}P_{J}^{p+s+t,q-s+1-t}(I^{\prime}+(I^{\prime\prime}-\bar\phi\cdot\phi)^{-1}-\bar\phi\cdot(I^{\prime}-\phi\cdot\bar\phi)^{-1})\Finv\\
&\underbrace{(\bar\partial+[\partial,i_{\phi}]-i_{\frac{1}{2}(\mathscr{B}(\phi,\phi)+\mathscr{C}(\phi,\phi))})
P_{J}^{p+s,q-s}((I-\bar\phi\cdot\phi+\bar\phi)\Finv\,\alpha)}_{p+s,q-s+1}\\
&+\sum_{s+t=2}P_{J}^{p+s-1+t,q-s+2-t}(I^{\prime}+(I^{\prime\prime}-\bar\phi\cdot\phi)^{-1}-\bar\phi\cdot(I^{\prime}-\phi\cdot\bar\phi)^{-1})\Finv\,
\underbrace{(\bar\mu+i_{\mathrm{MC}(\phi)})
P_{J}^{p+s,q-s}((I-\bar\phi\cdot\phi+\bar\phi)\Finv\,\alpha)}_{p+s-1,q-s+2}\\
=&\,P_{J}^{p+1,q}(I^{\prime}+(I^{\prime\prime}-\bar\phi\cdot\phi)^{-1}-\bar\phi\cdot(I^{\prime}-\phi\cdot\bar\phi)^{-1})\Finv\,
\underbrace{(\partial+[\mu,i_{\phi}]) P_{J}^{p,q}((I-\bar\phi\cdot\phi+\bar\phi)\Finv\,\alpha)}_{(p+1,q)}\\
&+P_{J}^{p+1,q}(I^{\prime}+(I^{\prime\prime}-\bar\phi\cdot\phi)^{-1}-\bar\phi\cdot(I^{\prime}-\phi\cdot\bar\phi)^{-1})\Finv\,
\underbrace{(\bar\partial+[\partial,i_{\phi}]-i_{\frac{1}{2}(\mathscr{B}(\phi,\phi)+\mathscr{C}(\phi,\phi))})
P_{J}^{p+1,q-1}((I-\bar\phi\cdot\phi+\bar\phi)\Finv\,\alpha)}_{p+1,q}\\
&+P_{J}^{p+1,q}(I^{\prime}+(I^{\prime\prime}-\bar\phi\cdot\phi)^{-1}-\bar\phi\cdot(I^{\prime}-\phi\cdot\bar\phi)^{-1})\Finv\,
\underbrace{(\bar\partial+[\partial,i_{\phi}]-i_{\frac{1}{2}(\mathscr{B}(\phi,\phi)+\mathscr{C}(\phi,\phi))})
P_{J}^{p,q}((I-\bar\phi\cdot\phi+\bar\phi)\Finv\,\alpha)}_{p,q+1}\\
&+P_{J}^{p+1,q}(I^{\prime}+(I^{\prime\prime}-\bar\phi\cdot\phi)^{-1}-\bar\phi\cdot(I^{\prime}-\phi\cdot\bar\phi)^{-1})\Finv\,
\underbrace{(\bar\mu+i_{\mathrm{MC}(\phi)})
P_{J}^{p+2,q-2}((I-\bar\phi\cdot\phi+\bar\phi)\Finv\,\alpha)}_{p+1,q}\\
&+P_{J}^{p+1,q}(I^{\prime}+(I^{\prime\prime}-\bar\phi\cdot\phi)^{-1}-\bar\phi\cdot(I^{\prime}-\phi\cdot\bar\phi)^{-1})\Finv\,
\underbrace{(\bar\mu+i_{\mathrm{MC}(\phi)})
P_{J}^{p+1,q-1}((I-\bar\phi\cdot\phi+\bar\phi)\Finv\,\alpha)}_{p,q+1}\\
&+P_{J}^{p+1,q}(I^{\prime}+(I^{\prime\prime}-\bar\phi\cdot\phi)^{-1}-\bar\phi\cdot(I^{\prime}-\phi\cdot\bar\phi)^{-1})\Finv\,
\underbrace{(\bar\mu+i_{\mathrm{MC}(\phi)})
P_{J}^{p,q}((I-\bar\phi\cdot\phi+\bar\phi)\Finv\,\alpha)}_{p-1,q+2}.
\end{align*}
\end{small}
The $J^{0,1}$-term is
\begin{small}
\begin{align*}
&(O_{1}\circ O_{2}\circ O_{3}(\alpha))_{J}^{0,1}\\
=&\,\sum_{s=0}^{q}\sum_{t=0}^{q-s+1}P_{J}^{p+s+t,q-s+1-t}(I^{\prime}+(I^{\prime\prime}-\bar\phi\cdot\phi)^{-1}-\bar\phi\cdot(I^{\prime}-\phi\cdot\bar\phi)^{-1})\Finv\\
&\underbrace{(\bar\partial+[\partial,i_{\phi}]-i_{\frac{1}{2}(\mathscr{B}(\phi,\phi)+\mathscr{C}(\phi,\phi))})
P_{J}^{p+s,q-s}((I-\bar\phi\cdot\phi+\bar\phi)\Finv\,\alpha)}_{p+s,q-s+1}\\
&+\sum_{s=0}^{q}\sum_{t=0}^{q-s+2}P_{J}^{p+s-1+t,q-s+2-t}(I^{\prime}+(I^{\prime\prime}-\bar\phi\cdot\phi)^{-1}-\bar\phi\cdot(I^{\prime}-\phi\cdot\bar\phi)^{-1})\Finv\,
\underbrace{(\bar\mu+i_{\mathrm{MC}(\phi)})
P_{J}^{p+s,q-s}((I-\bar\phi\cdot\phi+\bar\phi)\Finv\,\alpha)}_{p+s-1,q-s+2}\\
=&\,\sum_{s+t=0}P_{J}^{p+s+t,q-s+1-t}(I^{\prime}+(I^{\prime\prime}-\bar\phi\cdot\phi)^{-1}-\bar\phi\cdot(I^{\prime}-\phi\cdot\bar\phi)^{-1})\Finv\,
\underbrace{(\bar\partial+[\partial,i_{\phi}]-i_{\frac{1}{2}(\mathscr{B}(\phi,\phi)+\mathscr{C}(\phi,\phi))})
P_{J}^{p+s,q-s}((I-\bar\phi\cdot\phi+\bar\phi)\Finv\,\alpha)}_{p+s,q-s+1}\\
&+\sum_{s+t=1}^{q-s+2}P_{J}^{p+s-1+t,q-s+2-t}(I^{\prime}+(I^{\prime\prime}-\bar\phi\cdot\phi)^{-1}-\bar\phi\cdot(I^{\prime}-\phi\cdot\bar\phi)^{-1})\Finv\,
\underbrace{(\bar\mu+i_{\mathrm{MC}(\phi)})
P_{J}^{p+s,q-s}((I-\bar\phi\cdot\phi+\bar\phi)\Finv\,\alpha)}_{p+s-1,q-s+2}\\
=&\,P_{J}^{p,q+1}(I^{\prime}+(I^{\prime\prime}-\bar\phi\cdot\phi)^{-1}-\bar\phi\cdot(I^{\prime}-\phi\cdot\bar\phi)^{-1})\Finv\,
\underbrace{(\bar\partial+[\partial,i_{\phi}]-i_{\frac{1}{2}(\mathscr{B}(\phi,\phi)+\mathscr{C}(\phi,\phi))})
P_{J}^{p,q}((I-\bar\phi\cdot\phi+\bar\phi)\Finv\,\alpha)}_{p,q+1}\\
&+P_{J}^{p,q+1}(I^{\prime}+(I^{\prime\prime}-\bar\phi\cdot\phi)^{-1}-\bar\phi\cdot(I^{\prime}-\phi\cdot\bar\phi)^{-1})\Finv\,
\underbrace{(\bar\mu+i_{\mathrm{MC}(\phi)})
P_{J}^{p+1,q-1}((I-\bar\phi\cdot\phi+\bar\phi)\Finv\,\alpha)}_{p,q+1}\\
&+P_{J}^{p,q+1}(I^{\prime}+(I^{\prime\prime}-\bar\phi\cdot\phi)^{-1}-\bar\phi\cdot(I^{\prime}-\phi\cdot\bar\phi)^{-1})\Finv\,
\underbrace{(\bar\mu+i_{\mathrm{MC}(\phi)})
P_{J}^{p,q}((I-\bar\phi\cdot\phi+\bar\phi)\Finv\,\alpha)}_{p-1,q+2}.
\end{align*}
\end{small}
The $J^{-1,2}$-term is
\begin{equation*}
(O_{1}\circ O_{2}\circ O_{3}(\alpha))_{J}^{-1,2}
=P_{J}^{p-1,q+2}(I^{\prime}+(I^{\prime\prime}-\bar\phi\cdot\phi)^{-1}-\bar\phi\cdot(I^{\prime}-\phi\cdot\bar\phi)^{-1})\Finv\,
\underbrace{(\bar\mu+i_{\mathrm{MC}(\phi)})
P_{J}^{p,q}((I-\bar\phi\cdot\phi+\bar\phi)\Finv\,\alpha)}_{p-1,q+2}.
\end{equation*}
The LHS of (\ref{eq65}) is decomposed as
\begin{equation*}
d(e^{i_{\phi}|i_{\bar\phi}}(\sigma))
=\mu_{\phi}(e^{i_{\phi}|i_{\bar\phi}}(\sigma))+\partial_{\phi}(e^{i_{\phi}|i_{\bar\phi}}(\sigma))
+\partial_{\phi}(e^{i_{\phi}|i_{\bar\phi}}(\sigma))+\bar\mu_{\phi}(e^{i_{\phi}|i_{\bar\phi}}(\sigma)).
\end{equation*}
By comparing types of both sides, we have the following theorem.
\begin{theorem}\label{theorem2}
Let $(M,J)$ be an almost complex manifold and $d=\mu+\partial+\bar\partial+\bar\mu$ be the decomposition of the exterior differential operator with respect to $J$. Let $\phi\in A^{0,1}_{J}(M,T^{1,0}M)$ be a Beltrami differential on $M$ that generates a new almost complex structure $J_{\phi}$ on $M$ and $d=\mu_{\phi}+\partial_{\phi}+\bar\partial_{\phi}+\bar\mu_{\phi}$ be the decomposition of the exterior differential operator with respect to $J_{\phi}$. Then we have the following equations.
\begin{align*}
\mu_{\phi}\circ e^{i_{\phi}|i_{\bar\phi}}
=&\,e^{i_{\phi}|i_{\bar\phi}}((O_{1}\circ O_{2}\circ O_{3})_{J}^{2,-1}),\\
\partial_{\phi}\circ e^{i_{\phi}|i_{\bar\phi}}
=&\,e^{i_{\phi}|i_{\bar\phi}}((O_{1}\circ O_{2}\circ O_{3})_{J}^{1,0}),\\
\bar\partial_{\phi}\circ e^{i_{\phi}|i_{\bar\phi}}
=&\,e^{i_{\phi}|i_{\bar\phi}}((O_{1}\circ O_{2}\circ O_{3})_{J}^{0,1}),\\
\bar\mu_{\phi}\circ e^{i_{\phi}|i_{\bar\phi}}
=&\,e^{i_{\phi}|i_{\bar\phi}}((O_{1}\circ O_{2}\circ O_{3})_{J}^{-1,2}).
\end{align*}
\end{theorem}
Due to the limitation of the space of this paper and for the readers' convenience, we omit the explicit expressions of the terms of $O_{1}\circ O_{2}\circ O_{3}$ in this theorem.
\section{Applications}\label{sec6}
\subsection{\texorpdfstring{$(n,0)$}{(n,0)}-forms}
Let $\Omega\in A^{n,0}_{J}(M)$. The $J^{0,1}$-term of $O_{1}\circ O_{2}\circ O_{3}(\Omega)$ is
\begin{small}
\begin{align*}
&(O_{1}\circ O_{2}\circ O_{3}(\Omega))_{J}^{0,1}\\
=&\,P_{J}^{n,1}(I^{\prime}+(I^{\prime\prime}-\bar\phi\cdot\phi)^{-1}-\bar\phi\cdot(I^{\prime}-\phi\cdot\bar\phi)^{-1})\Finv\,
\underbrace{(\bar\partial+[\partial,i_{\phi}]-i_{\frac{1}{2}(\mathscr{B}(\phi,\phi)+\mathscr{C}(\phi,\phi))})
P_{J}^{n,0}((I-\bar\phi\cdot\phi+\bar\phi)\Finv\,\Omega)}_{(n,1)}\\
&+P_{J}^{n,1}(I^{\prime}+(I^{\prime\prime}-\bar\phi\cdot\phi)^{-1}-\bar\phi\cdot(I^{\prime}-\phi\cdot\bar\phi)^{-1})\Finv\,
\underbrace{(\bar\mu+i_{\mathrm{MC}(\phi)})
P_{J}^{n+1,-1}((I-\bar\phi\cdot\phi+\bar\phi)\Finv\,\Omega)}_{(n,1)}\\
&+P_{J}^{n,1}(I^{\prime}+(I^{\prime\prime}-\bar\phi\cdot\phi)^{-1}-\bar\phi\cdot(I^{\prime}-\phi\cdot\bar\phi)^{-1})\Finv\,
\underbrace{(\bar\mu+i_{\mathrm{MC}(\phi)})
P_{J}^{n,0}((I^{\prime\prime}-\bar\phi\cdot\phi+\bar\phi)\Finv\,\Omega)}_{(n-1,2)}\\
%=&P_{J}^{n,1}(I^{\prime}+(1-\bar\phi\cdot\phi)^{-1}-\bar\phi\cdot(I-\phi\cdot\bar\phi)^{-1})\Finv\,
%(\bar\partial+[\partial,i_{\phi}]-i_{\frac{1}{2}(\mathscr{B}+\mathscr{C})})
%P_{J}^{n,0}(\Omega)\\
%&+P_{J}^{n,1}(I^{\prime}+(1-\bar\phi\cdot\phi)^{-1}-\bar\phi\cdot(I-\phi\cdot\bar\phi)^{-1})\Finv\,
%(\bar\mu+i_{\bar\partial\phi-\frac{1}{2}\mathscr{A}}
%-i_{\frac{1}{3!}(i_{[\phi,\phi]}\phi-i_{\phi}[\phi,\phi])})
%P_{J}^{n+1,-1}(\Omega)\\
%&+P_{J}^{n,1}(I^{\prime}+(1-\bar\phi\cdot\phi)^{-1}-\bar\phi\cdot(I-\phi\cdot\bar\phi)^{-1})\Finv\,
%(\bar\mu+i_{\bar\partial\phi-\frac{1}{2}\mathscr{A}}
%-i_{\frac{1}{3!}(i_{[\phi,\phi]}\phi-i_{\phi}[\phi,\phi])})
%P_{J}^{n,0}(\Omega)\\
=&\,P_{J}^{n,1}(I^{\prime}+(I^{\prime\prime}-\bar\phi\cdot\phi)^{-1}-\bar\phi\cdot(I^{\prime}-\phi\cdot\bar\phi)^{-1})\Finv\,
\underbrace{(\bar\partial+[\partial,i_{\phi}]-i_{\frac{1}{2}(\mathscr{B}(\phi,\phi)+\mathscr{C}(\phi,\phi))})(\Omega)}_{(n,1)}\\
&+P_{J}^{n,1}(I^{\prime}+(I^{\prime\prime}-\bar\phi\cdot\phi)^{-1}-\bar\phi\cdot(I^{\prime}-\phi\cdot\bar\phi)^{-1})\Finv\,
\underbrace{(\bar\mu+i_{\mathrm{MC}(\phi)})(\Omega)}_{(n-1,2)}\\
%=&P_{J}^{n,1}(I^{\prime}+(1-\bar\phi\cdot\phi)^{-1}-\bar\phi\cdot(I-\phi\cdot\bar\phi)^{-1})\Finv\,
%\bigl((\bar\partial+[\partial,i_{\phi}]-i_{\frac{1}{2}(\mathscr{B}+\mathscr{C})})\Omega
%+(\bar\mu+i_{\bar\partial\phi-\frac{1}{2}\mathscr{A}}-i_{\frac{1}{3!}(i_{[\phi,\phi]}\phi-i_{\phi}[\phi,\phi])})\Omega\bigr)\\
=&\,P_{J}^{n,1}(I^{\prime}+(I^{\prime\prime}-\bar\phi\cdot\phi)^{-1}-\bar\phi\cdot(I^{\prime}-\phi\cdot\bar\phi)^{-1})\Finv\,
(\underbrace{\bar\partial\Omega+\partial(\phi\lrcorner\Omega)-\frac{1}{2}\mathscr{B}(\phi,\phi)\lrcorner\Omega}_{(n,1)}
+\underbrace{\bar\mu\Omega+\mathrm{MC}(\phi)\lrcorner\Omega}_{(n-1,2)}).
\end{align*}
\end{small}
The $J^{-1,2}$-term of $O_{1}\circ O_{2}\circ O_{3}(\Omega)$ is
\begin{align*}
&(O_{1}\circ O_{2}\circ O_{3}(\Omega))_{J}^{-1,2}\\
=&\,P_{J}^{n-1,2}(I^{\prime}+(I^{\prime\prime}-\bar\phi\cdot\phi)^{-1}-\bar\phi\cdot(I^{\prime}-\phi\cdot\bar\phi)^{-1})\Finv\,
\underbrace{(\bar\mu+i_{\mathrm{MC}(\phi)})
P_{J}^{n,0}((I-\bar\phi\cdot\phi+\bar\phi)\Finv\,\Omega)}_{(n-1,2)}\\
=&\,P_{J}^{n-1,2}(I^{\prime}+(I^{\prime\prime}-\bar\phi\cdot\phi)^{-1}-\bar\phi\cdot(I^{\prime}-\phi\cdot\bar\phi)^{-1})\Finv\,
\underbrace{(\bar\mu+i_{\mathrm{MC}(\phi)})(\Omega)}_{(n-1,2)}\\
=&\,P_{J}^{n-1,2}(I^{\prime}+(I^{\prime\prime}-\bar\phi\cdot\phi)^{-1}-\bar\phi\cdot(I^{\prime}-\phi\cdot\bar\phi)^{-1})\Finv\,
\underbrace{(\bar\mu\Omega+\mathrm{MC}(\phi)\lrcorner\Omega}_{(n-1,2)}).
\end{align*}
For $\Theta=\Theta_{1\cdots n\bar j}\theta^{1}\wedge\cdots\wedge\theta^{n}\wedge\theta^{\bar j}\in A_{J}^{n,1}(M)$ we have
\begin{align}
&P_{J}^{n,1}((I^{\prime}+(I^{\prime\prime}-\bar\phi\cdot\phi)^{-1}
-\bar\phi\cdot(I^{\prime}-\phi\cdot\bar\phi)^{-1})\Finv\,\Theta)\nonumber\\
=&\,\alpha_{1\ldots n\bar j}(I^{\prime\prime}-\bar\phi\cdot\phi)^{\bar j}_{\bar\ell}
\theta^{1}\wedge\cdots\wedge\theta^{n}\wedge\theta^{\bar\ell}\nonumber\\
=&\,(-1)^{n}(I^{\prime\prime}-\bar\phi\cdot\phi)^{-1}\lrcorner\Theta.\label{equ51}
\end{align}
For $\Xi=\Xi_{i_{1}\cdots i_{n-1}\bar j_{1}\bar j_{2}}\theta^{i_{1}}\wedge\cdots\wedge\theta^{i_{n-1}}\wedge\theta^{\bar j_{1}}\wedge\theta^{\bar j_{2}}\in A_{J}^{n-1,2}(M)$, by direct calculation we have
\begin{align*}
&(I^{\prime\prime}-\bar\phi\cdot\phi)^{-1}\lrcorner(\bar\phi\cdot(I^{\prime}-\phi\cdot\bar\phi)^{-1}\lrcorner\Xi)\\
=&\,((I^{\prime\prime}-\bar\phi\cdot\phi)^{-1})_{\bar j}^{\bar k}\theta^{\bar j}\otimes e_{\bar k}\lrcorner
\bigl((\bar\phi\cdot(I^{\prime}-\phi\cdot\bar\phi)^{-1})^{\bar p}_{q}\theta^{q}\otimes e_{\bar p}\lrcorner
(\Xi_{i_{1}\cdots i_{n-1}\bar j_{1}\bar j_{2}}\theta^{i_{1}}\wedge\theta^{i_{n-1}}\wedge\theta^{\bar j_{1}}\wedge\theta^{\bar j_{2}})\bigr)\\
=&\,((I^{\prime\prime}-\bar\phi\cdot\phi)^{-1})_{\bar j}^{\bar k}\theta^{\bar j}\otimes e_{\bar k}\lrcorner
\bigl((\bar\phi\cdot(I^{\prime}-\phi\cdot\bar\phi)^{-1})^{\bar p}_{q}(-1)^{n-1} \delta^{\bar j_{1}}_{\bar p}
\Xi_{i_{1}\cdots i_{n-1}\bar j_{1}\bar j_{2}}\theta^{q}\wedge\theta^{i_{1}}\wedge\theta^{i_{n-1}}\wedge\theta^{\bar j_{2}})\\
&+(\bar\phi\cdot(I^{\prime}-\phi\cdot\bar\phi)^{-1})^{\bar p}_{q}(-1)^{n} \delta_{\bar p}^{\bar j_{2}}
\Xi_{i_{1}\cdots i_{n-1}\bar j_{1}\bar j_{2}}\theta^{q}\wedge\theta^{i_{1}}\wedge\theta^{i_{n-1}}\wedge\theta^{\bar j_{1}})\bigr)\\
=&\,((I^{\prime\prime}-\bar\phi\cdot\phi)^{-1})_{\bar j}^{\bar k}\theta^{\bar j}\otimes e_{\bar k}\lrcorner
\bigl((\bar\phi\cdot(I^{\prime}-\phi\cdot\bar\phi)^{-1})^{\bar j_{1}}_{q}
\Xi_{i_{1}\cdots i_{n-1}\bar j_{1}\bar j_{2}}
\theta^{i_{1}}\wedge\theta^{i_{n-1}}\wedge\theta^{q}\wedge\theta^{\bar j_{2}})\\
&-(\bar\phi\cdot(I^{\prime}-\phi\cdot\bar\phi)^{-1})^{\bar j_{2}}_{q}
\Xi_{i_{1}\cdots i_{n-1}\bar j_{1}\bar j_{2}}
\theta^{i_{1}}\wedge\theta^{i_{n-1}}\wedge\theta^{q}\wedge\theta^{\bar j_{1}}\bigr)\\
=&\,((I^{\prime\prime}-\bar\phi\cdot\phi)^{-1})_{\bar j}^{\bar k}\theta^{\bar j}\otimes e_{\bar k}\lrcorner
\bigl((\bar\phi\cdot(I^{\prime}-\phi\cdot\bar\phi)^{-1})^{\bar j_{1}}_{q}
\Xi_{i_{1}\cdots i_{n-1}\bar j_{1}\bar j_{2}}
\theta^{i_{1}}\wedge\theta^{i_{n-1}}\wedge\theta^{q}\wedge\theta^{\bar j_{2}})\\
&-((I^{\prime\prime}-\bar\phi\cdot\phi)^{-1})_{\bar j}^{\bar k}\theta^{\bar j}\otimes e_{\bar k}\lrcorner
\bigl((\bar\phi\cdot(I^{\prime}-\phi\cdot\bar\phi)^{-1})^{\bar j_{2}}_{q}
\Xi_{i_{1}\cdots i_{n-1}\bar j_{1}\bar j_{2}}
\theta^{i_{1}}\wedge\theta^{i_{n-1}}\wedge\theta^{q}\wedge\theta^{\bar j_{1}}\bigr)\\
=&\,(-1)^{n}((I^{\prime\prime}-\bar\phi\cdot\phi)^{-1})_{\bar j}^{\bar k}\delta^{\bar j_{2}}_{\bar k}(\bar\phi\cdot(I-\phi\cdot\bar\phi)^{-1})^{\bar j_{1}}_{q}
\Xi_{i_{1}\cdots i_{n-1}\bar j_{1}\bar j_{2}}
\theta^{\bar j}\wedge\theta^{i_{1}}\wedge\theta^{i_{n-1}}\wedge\theta^{q}\\
&-(-1)^{n}((I^{\prime\prime}-\bar\phi\cdot\phi)^{-1})_{\bar j}^{\bar k} \delta_{\bar k}^{j_{1}}
((\bar\phi\cdot(I-\phi\cdot\bar\phi)^{-1})^{\bar j_{2}}_{q}
\Xi_{i_{1}\cdots i_{n-1}\bar j_{1}\bar j_{2}}
\theta^{\bar j}\wedge\theta^{i_{1}}\wedge\theta^{i_{n-1}}\wedge\theta^{q}\\
=&\,(\bar\phi\cdot(I^{\prime}-\phi\cdot\bar\phi)^{-1})^{\bar j_{1}}_{q}
((I^{\prime\prime}-\bar\phi\cdot\phi)^{-1})_{\bar j}^{\bar j_{2}}
\Xi_{i_{1}\cdots i_{n-1}\bar j_{1}\bar j_{2}}
\theta^{i_{1}}\wedge\theta^{i_{n-1}}\wedge\theta^{q}\wedge\theta^{\bar j}\\
&-((I^{\prime\prime}-\bar\phi\cdot\phi)^{-1})_{\bar j}^{\bar j_{1}}
((\bar\phi\cdot(I^{\prime}-\phi\cdot\bar\phi)^{-1})^{\bar j_{2}}_{q}
\Xi_{i_{1}\cdots i_{n-1}\bar j_{1}\bar j_{2}}
\wedge\theta^{i_{1}}\wedge\theta^{i_{n-1}}\wedge\theta^{q}\wedge\theta^{\bar j},
\end{align*}
and consequently
\begin{align}
&P_{J}^{n,1}((I^{\prime}+(I^{\prime\prime}-\bar\phi\cdot\phi)^{-1}
-\bar\phi\cdot(I^{\prime}-\phi\cdot\bar\phi)^{-1})\Finv\,\Xi)\nonumber\\
=&\,-\Xi_{i_{1}\ldots i_{n-1}\bar j_{1}\bar j_{2}}(\bar\phi\cdot(I^{\prime}-\phi\cdot\bar\phi)^{-1})^{\bar j_{1}}_{q}(I^{\prime\prime}-\bar\phi\cdot\phi)^{\bar j_{2}}_{\bar j}
\theta^{i_{1}}\wedge\cdots\wedge\theta^{i_{n-1}}\wedge\theta^{q}\wedge\theta^{\bar j}\nonumber\\
&-\Xi_{i_{1}\ldots i_{n-1}\bar j_{1}\bar j_{2}}
(I^{\prime\prime}-\bar\phi\cdot\phi)^{\bar j_{1}}_{\bar j}
(\bar\phi\cdot(I^{\prime}-\phi\cdot\bar\phi)^{-1})^{\bar j_{2}}_{q}
\theta^{i_{1}}\wedge\cdots\wedge\theta^{i_{n-1}}\wedge\theta^{\bar j}\wedge\theta^{q}\nonumber\\
=&\,-\Xi_{i_{1}\ldots i_{n-1}\bar j_{1}\bar j_{2}}(\bar\phi\cdot(I^{\prime}-\phi\cdot\bar\phi)^{-1})^{\bar j_{1}}_{q}
(I-\bar\phi\cdot\phi)^{\bar j_{2}}_{\bar j}
\theta^{i_{1}}\wedge\cdots\wedge\theta^{i_{n-1}}\wedge\theta^{q}\wedge\theta^{\bar j}\nonumber\\
&+\Xi_{i_{1}\ldots i_{n-1}\bar j_{1}\bar j_{2}}
(I^{\prime\prime}-\bar\phi\cdot\phi)^{\bar j_{1}}_{\bar\ell}(\bar\phi\cdot(I^{\prime}-\phi\cdot\bar\phi)^{-1})^{\bar j_{2}}_{q}
\theta^{i_{1}}\wedge\cdots\wedge\theta^{i_{n-1}}\wedge\theta^{k}\wedge\theta^{\bar j}\nonumber\\
=&\,-(I-\bar\phi\cdot\phi)^{-1}\lrcorner(\bar\phi\cdot(I^{\prime}-\phi\cdot\bar\phi)^{-1}\lrcorner\,\Xi).\label{equ52}
\end{align}
We can calculate the $P^{n-1,2}_{J}$ projection of $\Xi$ directly as follows.
\begin{align}
&P_{J}^{n-1,2}((I^{\prime}+(I^{\prime\prime}-\bar\phi\cdot\phi)^{-1}
-\bar\phi\cdot(I^{\prime}-\phi\cdot\bar\phi)^{-1})\Finv\,\Xi)\nonumber\\
=&\,P_{J}^{n-1,2}(\Xi_{i_{1}\cdots i_{n-1}\bar j_{1}\bar j_{2}}
(I^{\prime}+(I^{\prime\prime}-\bar\phi\cdot\phi)^{-1}
-\bar\phi\cdot(I^{\prime}-\phi\cdot\bar\phi)^{-1})\lrcorner\,\theta^{i_{1}}
\wedge\cdots\nonumber\\
&\wedge(I^{\prime}+(I^{\prime\prime}-\bar\phi\cdot\phi)^{-1}
-\bar\phi\cdot(I^{\prime}-\phi\cdot\bar\phi)^{-1})\lrcorner\,\theta^{i_{n-1}}
\wedge(I^{\prime}+(I^{\prime\prime}-\bar\phi\cdot\phi)^{-1}
-\bar\phi\cdot(I^{\prime}-\phi\cdot\bar\phi)^{-1})\lrcorner\,\theta^{\bar j_{1}}\nonumber\\
&\wedge(I^{\prime}+(I^{\prime\prime}-\bar\phi\cdot\phi)^{-1}
-\bar\phi\cdot(I^{\prime}-\phi\cdot\bar\phi)^{-1})\lrcorner\,\theta^{\bar j_{2}})\nonumber\\
=&\,\Xi_{i_{1}\cdots i_{n-1}\bar j_{1}\bar j_{2}}
\theta^{i_{1}}\wedge\cdots\wedge\theta^{i_{n-1}}
\wedge((I^{\prime\prime}-\bar\phi\cdot\phi)^{-1})\lrcorner\,\theta^{\bar j_{1}}
\wedge((I^{\prime\prime}-\bar\phi\cdot\phi)^{-1})\lrcorner\,\theta^{\bar j_{2}}\nonumber\\
=&\,(I^{\prime}+(I^{\prime\prime}-\bar\phi\cdot\phi)^{-1})\Finv\,\Xi.\label{equ55}
\end{align}
Thus
\begin{align}
\bar\partial_{\phi}e^{i_{\phi}|i_{\bar\phi}}(\Omega)
=&\,e^{i_{\phi}|i_{\bar\phi}}\circ P_{J}^{n,1}(I^{\prime}+(I^{\prime\prime}-\bar\phi\cdot\phi)^{-1}
-\bar\phi\cdot(I^{\prime}-\phi\cdot\bar\phi)^{-1})\Finv\,
(\underbrace{\bar\partial\Omega+\partial(\phi\lrcorner\Omega)
-\frac{1}{2}\mathscr{B}(\phi,\phi)\lrcorner\Omega}_{(n,1)})\nonumber\\
&+e^{i_{\phi}|i_{\bar\phi}}\circ P_{J}^{n,1}(I^{\prime}+(I^{\prime\prime}-\bar\phi\cdot\phi)^{-1}
-\bar\phi\cdot(I^{\prime}-\phi\cdot\bar\phi)^{-1})\Finv\,
(\underbrace{\bar\mu\Omega+\mathrm{MC}(\phi)\lrcorner\Omega}_{(n-1,2)})\nonumber\\
=&\,e^{i_{\phi}|i_{\bar\phi}}\circ(-1)^{n}(I^{\prime\prime}-\bar\phi\cdot\phi)^{-1}\lrcorner\,
(\underbrace{\bar\partial\Omega+\partial(\phi\lrcorner\Omega)
-\frac{1}{2}\mathscr{B}(\phi,\phi)\lrcorner\Omega}_{(n,1)})\quad\textup{by (\ref{equ51})}\nonumber\\
&-e^{i_{\phi}|i_{\bar\phi}}\circ(I^{\prime\prime}-\bar\phi\cdot\phi)^{-1}\lrcorner\bigl(
(\bar\phi\cdot(I^{\prime}-\phi\cdot\bar\phi)^{-1})\lrcorner\,
(\underbrace{\bar\mu\Omega+\mathrm{MC}(\phi)\lrcorner\Omega}_{(n-1,2)})\bigr),\quad
\textup{by (\ref{equ52})}\label{equ53}\\
\bar\mu_{\phi}e^{i_{\phi}|i_{\bar\phi}}(\Omega)
=&\,e^{i_{\phi}|i_{\bar\phi}}\circ P_{J}^{n-1,2}(I^{\prime}+(I^{\prime\prime}-\bar\phi\cdot\phi)^{-1}-\bar\phi\cdot(I^{\prime}-\phi\cdot\bar\phi)^{-1})\Finv\,
(\underbrace{\bar\mu\Omega+\mathrm{MC}(\phi)\lrcorner\Omega}_{n-1,2})\nonumber\\
=&\,e^{i_{\phi}|i_{\bar\phi}}\circ(I^{\prime}+(I^{\prime\prime}-\bar\phi\cdot\phi)^{-1})\Finv\,
\underbrace{(\bar\mu\Omega+\mathrm{MC}(\phi)\lrcorner\Omega}_{n-1,2}.\quad\textup{by (\ref{equ55})}\label{equ50}
\end{align}
If $\phi$ satisfies the almost complex Maurer-Cartan equation, (\ref{equ53}) and (\ref{equ50}) reduce to
\begin{align}
\bar\partial_{\phi}e^{i_{\phi}|i_{\bar\phi}}(\Omega)
=&\,e^{i_{\phi}|i_{\bar\phi}}\circ(-1)^{n}(I^{\prime\prime}-\bar\phi\cdot\phi)^{-1}\lrcorner\,
(\bar\partial\Omega+\partial(\phi\lrcorner\Omega)
-\frac{1}{2}\mathscr{B}(\phi,\phi)\lrcorner\Omega)\nonumber\\
&-e^{i_{\phi}|i_{\bar\phi}}\circ(I^{\prime\prime}-\bar\phi\cdot\phi)^{-1}\lrcorner\bigl(
(\bar\phi\cdot(I^{\prime}-\phi\cdot\bar\phi)^{-1})\lrcorner\,(\bar\mu\Omega)\bigr),\label{equ40}\\
\bar\mu_{\phi}e^{i_{\phi}|i_{\bar\phi}}(\Omega)
=&\,e^{i_{\phi}|i_{\bar\phi}}\circ(I^{\prime}+(I^{\prime\prime}-\bar\phi\cdot\phi)^{-1})\Finv\,
(\bar\mu\Omega).\label{equ41}
\end{align}
\begin{proposition}\label{proposition2}
Let $(M,J)$ be an almost complex manifold and $\phi$ be a Beltrami differential that satisfies the almost complex Maurer-Cartan equation. For any smooth $J^{n,0}$-form $\Omega$, $e^{i_{\phi}|i_{\bar\phi}}(\Omega)$ is a $\bar\partial_{\phi}$-closed $J_{\phi}^{n,0}$-form if and only if
\begin{equation*}
(-1)^{n}(\bar\partial\Omega+\partial(\phi\lrcorner\Omega)-\frac{1}{2}\mathscr{B}(\phi,\phi)\lrcorner\Omega)
-\bar\phi\cdot(I^{\prime}-\phi\cdot\bar\phi)^{-1}\lrcorner\bar\mu\Omega=0.
\end{equation*}
\end{proposition}
\begin{proof}
By (\ref{equ40}) the fact that $(I^{\prime}-\phi\cdot\bar\phi)$ (and equivalently $(I^{\prime\prime}-\bar\phi\cdot\phi)$) is pointwisely invertible we know that $\bar\partial_{\phi}e^{i_{\phi}|i_{\bar\phi}}(\Omega)=0$ is equivalent to
\begin{equation*}
(-1)^{n}(\bar\partial\Omega+\partial(\phi\lrcorner\Omega)-\frac{1}{2}\mathscr{B}(\phi,\phi)\lrcorner\Omega)
-\bar\phi\cdot(I^{\prime}-\phi\cdot\bar\phi)^{-1}\lrcorner\bar\mu\Omega=0.
\end{equation*}
\end{proof}
In the integrable case, Proposition \ref{proposition2} reduces to \cite[Proposition 5.1]{MR3302118}.
\begin{proposition}\cite[Proposition 5.1]{MR3302118}
Let $(M,J,N_{J}=0)$ be a complex manifold and $\phi$ be a Beltrami differential that satisfies the Maurer-Cartan equation. For any smooth $J^{n,0}$-form $\Omega$, $e^{i_{\phi}|i_{\bar\phi}}(\Omega)$ is a $\bar\partial_{\phi}$-closed $J_{\phi}^{n,0}$-form if and only if
\begin{equation*}
\bar\partial\Omega+\partial(\phi\lrcorner\Omega)=0.
\end{equation*}
\end{proposition}
\begin{proof}
Note that in the integrable case we have $\mu=0$ (equivalently $\bar\mu=0$) and $\mathscr{B}(\phi,\phi)=0$.
\end{proof}
By the definition of the Dolbeault cohomology on almost complex manifolds, $[[\Omega]]\in H^{n,0}_{\bar\partial}(M)$ implies that $\bar\mu\Omega=0$ and
\begin{equation*}
\bar\partial[\Omega]=[\bar\partial\Omega]=0 \Leftrightarrow \bar\partial\Omega=\bar\mu\rho=0.
\end{equation*}
Thus by (\ref{equ40}) and (\ref{equ41}) we have
\begin{align}
\bar\partial_{\phi}e^{i_{\phi}|i_{\bar\phi}}(\Omega)
=&\,(-1)^{n}e^{i_{\phi}|i_{\bar\phi}}\circ (I^{\prime\prime}-\bar\phi\cdot\phi)^{-1}
\lrcorner(\partial(\phi\lrcorner\Omega)-\frac{1}{2}\mathscr{B}(\phi,\phi)\lrcorner\Omega),\label{equ38}\\
\bar\mu_{\phi}e^{i_{\phi}|i_{\bar\phi}}(\Omega)
=&\,e^{i_{\phi}|i_{\bar\phi}}\circ(I^{\prime}+(I^{\prime\prime}-\bar\phi\cdot\phi)^{-1})\Finv\,
(\bar\mu\Omega)=0.\label{equ39}
\end{align}
By (\ref{equ38}) and (\ref{equ39}) we have the following theorem.
\begin{theorem}\label{theorem5}
Let $(M,J)$ be an almost complex manifold, $\phi\in A^{0,1}_{J}(M,T^{1,0}M)$ be a Beltrami differential which generates a new almost complex structure $J_{\phi}$ on $M$. Assume that $\phi$ satisfies the almost complex Maurer-Cartan equation. Then for any $[[\Omega]]\in H^{n,0}_{\mathrm{Dol},J}(M)$, $[[e^{i_{\phi}|i_{\bar\phi}}(\Omega)]]\in H^{n,0}_{\mathrm{Dol},J_{\phi}}(M)$ if and only if
\begin{equation*}
\partial(\phi\lrcorner\Omega)-\frac{1}{2}\mathscr{B}(\phi,\phi)\lrcorner\Omega=0.
\end{equation*}
\end{theorem}
\begin{proof}
By (\ref{equ38}) and (\ref{equ39}) we know that $[[e^{i_{\phi}|i_{\bar\phi}}(\Omega)]]\in H^{n,0}_{\mathrm{Dol},J_{\phi}}(M)$ if and only if
\begin{equation}\label{equ54}
(I^{\prime\prime}-\bar\phi\cdot\phi)^{-1}\lrcorner(\partial(\phi\lrcorner\Omega)-\frac{1}{2}\mathscr{B}(\phi,\phi)\lrcorner\Omega)=0.
\end{equation}
By the fact that $(I^{\prime}-\phi\cdot\bar\phi)$ is pointwisely invertible we know that (\ref{equ54}) is equivalent to
\begin{equation*}
\partial(\phi\lrcorner\Omega)-\frac{1}{2}\mathscr{B}(\phi,\phi)\lrcorner\Omega=0.
\end{equation*}
\end{proof}
\subsection{\texorpdfstring{$(n,q)$}{(n,q)}-forms}
Let $\Xi\in A_{J}^{n,q}(M)$. The $J^{0,1}$-term of $O_{1}\circ O_{2}\circ O_{3}(\Xi)$ is
\begin{small}
\begin{align}
&(O_{1}\circ O_{2}\circ O_{3}(\Xi))_{J}^{0,1}\nonumber\\
=&P_{J}^{n,q+1}(I^{\prime}+(I^{\prime\prime}-\bar\phi\cdot\phi)^{-1}-\bar\phi\cdot(I^{\prime}-\phi\cdot\bar\phi)^{-1})\Finv
\underbrace{(\bar\partial+[\partial,i_{\phi}]-i_{\frac{1}{2}(\mathscr{B}(\phi,\phi)+\mathscr{C}(\phi,\phi))})
P_{J}^{n,q}((I-\bar\phi\cdot\phi+\bar\phi)\Finv\Xi)}_{n,q+1}\nonumber\\
&+P_{J}^{n,q+1}(I^{\prime}+(I^{\prime\prime}-\bar\phi\cdot\phi)^{-1}-\bar\phi\cdot(I^{\prime}-\phi\cdot\bar\phi)^{-1})\Finv
\underbrace{(\bar\mu+i_{\mathrm{MC}(\phi)})
P_{J}^{n+1,q-1}((I-\bar\phi\cdot\phi+\bar\phi)\Finv\Xi)}_{n,q+1}\nonumber\\
&+P_{J}^{n,q+1}(I^{\prime}+(I^{\prime\prime}-\bar\phi\cdot\phi)^{-1}-\bar\phi\cdot(I^{\prime}-\phi\cdot\bar\phi)^{-1})\Finv
\underbrace{(\bar\mu+i_{\mathrm{MC}(\phi)})
P_{J}^{n,q}((I-\bar\phi\cdot\phi+\bar\phi)\Finv\Xi)}_{n-1,q+2}\nonumber\\
=&P_{J}^{n,q+1}(I^{\prime}+(I^{\prime\prime}-\bar\phi\cdot\phi)^{-1}-\bar\phi\cdot(I^{\prime}-\phi\cdot\bar\phi)^{-1})\Finv
\underbrace{(\bar\partial+[\partial,i_{\phi}]-i_{\frac{1}{2}(\mathscr{B}(\phi,\phi)+\mathscr{C}(\phi,\phi))})
P_{J}^{n,q}((I-\bar\phi\cdot\phi+\bar\phi)\Finv\Xi)}_{n,q+1}\nonumber\\
&+P_{J}^{n,q+1}(I^{\prime}+(I^{\prime\prime}-\bar\phi\cdot\phi)^{-1}-\bar\phi\cdot(I^{\prime}-\phi\cdot\bar\phi)^{-1})\Finv
\underbrace{(\bar\mu+i_{\mathrm{MC}(\phi)})
P_{J}^{n,q}((I-\bar\phi\cdot\phi+\bar\phi)\Finv\Xi)}_{n-1,q+2}\nonumber\\
=&P_{J}^{n,q+1}(I^{\prime}+(I^{\prime\prime}-\bar\phi\cdot\phi)^{-1}-\bar\phi\cdot(I^{\prime}-\phi\cdot\bar\phi)^{-1})\Finv
\underbrace{(\bar\partial+[\partial,i_{\phi}]-i_{\frac{1}{2}(\mathscr{B}(\phi,\phi)+\mathscr{C}(\phi,\phi))})
((I-\bar\phi\cdot\phi)\Finv\Xi)}_{n,q+1}\nonumber\\
&+P_{J}^{n,q+1}(I^{\prime}+(I^{\prime\prime}-\bar\phi\cdot\phi)^{-1}-\bar\phi\cdot(I^{\prime}-\phi\cdot\bar\phi)^{-1})\Finv
\underbrace{(\bar\mu+i_{\mathrm{MC}(\phi)})((I-\bar\phi\cdot\phi)\Finv\Xi)}_{n-1,q+2}\nonumber\\
=&(I^{\prime}+(I^{\prime\prime}-\bar\phi\cdot\phi)^{-1})\Finv
(\bar\partial+[\partial,i_{\phi}]-i_{\frac{1}{2}(\mathscr{B}(\phi,\phi)+\mathscr{C}(\phi,\phi))})(I-\bar\phi\cdot\phi)\Finv\Xi\nonumber\\
&-(I^{\prime}+(I^{\prime\prime}-\bar\phi\cdot\phi)^{-1})\Finv(\bar\phi\cdot(I^{\prime}-\phi\cdot\bar\phi)^{-1})\lrcorner
(\bar\mu+i_{\mathrm{MC}(\phi)})(I-\bar\phi\cdot\phi)\Finv\Xi\nonumber\\
=&(I^{\prime}+(I^{\prime\prime}-\bar\phi\cdot\phi)^{-1})\Finv\nonumber\\
&\bigl((\bar\partial+[\partial,i_{\phi}]-i_{\frac{1}{2}(\mathscr{B}(\phi,\phi)+\mathscr{C}(\phi,\phi))})(I-\bar\phi\cdot\phi)\Finv\Xi
-(\bar\phi\cdot(I^{\prime}-\phi\cdot\bar\phi)^{-1})\lrcorner(\bar\mu+i_{\mathrm{MC}(\phi)})(I-\bar\phi\cdot\phi)\Finv\Xi\bigr)\nonumber\\
=&(I^{\prime}+(I^{\prime\prime}-\bar\phi\cdot\phi)^{-1})\Finv
\bigl((\bar\partial+[\partial,i_{\phi}]-i_{\frac{1}{2}(\mathscr{B}(\phi,\phi)+\mathscr{C}(\phi,\phi))}
-i_{\bar\phi\cdot(I^{\prime}-\phi\cdot\bar\phi)^{-1}}\circ(\bar\mu+i_{\mathrm{MC}(\phi)}))(I-\bar\phi\cdot\phi)\Finv\Xi\bigr).\label{equa1}
\end{align}
\end{small}
Now we can prove following theorem for smooth $(n,q)$-forms.
\begin{theorem}\label{theorem6}
Let $(M,J)$ be an almost complex manifold and $\phi$ be a Beltrami differential that induces a new almost complex structure. For any smooth $\Xi\in A^{n,q}_{J}(M)$, $e^{i_{\phi}|i_{\bar\phi}}(\Xi)\in A^{n,q}_{\phi}(M)$ is $\bar\partial_{\phi}$-closed if and only if
\begin{small}
\begin{equation*}
(I^{\prime}+(I^{\prime\prime}-\bar\phi\cdot\phi)^{-1})\Finv
\bigl((\bar\partial+[\partial,i_{\phi}]-i_{\frac{1}{2}(\mathscr{B}(\phi,\phi)+\mathscr{C}(\phi,\phi))}
-i_{\bar\phi\cdot(I^{\prime}-\phi\cdot\bar\phi)^{-1}}\circ(\bar\mu+i_{\mathrm{MC}(\phi)}))(I-\bar\phi\cdot\phi)\Finv\Xi\bigr)=0.
\end{equation*}
\end{small}
\end{theorem}
\begin{proof}
Direct consequence of (\ref{equa1}), the fact that $e^{i_{\phi}|i_{\bar\phi}}$ is an isomorphism and
\begin{equation*}
\bar\partial_{\phi}e^{i_{\phi}|i_{\bar\phi}}(\Xi)=e^{i_{\phi}|i_{\bar\phi}}((O_{1}\circ O_{2}\circ O_{3}(\Xi))_{J}^{0,1}).
\end{equation*}
\end{proof}
\section{Appendix: Dolbeault cohomology on almost complex manifolds}\label{sec7}

This part is basically taken form \cite{cirici2018dolbeault} for the self-containment of this manuscript and the readers' convenience.

Let $(M,J)$ be an almost complex manifold. The exterior differential operator is
$$
d=\mu+\partial+\bar\partial+\bar\mu.
$$
\begin{displaymath}
\xymatrix{
A^{p-1,q-1}&A^{p,q} \ar@{=>}[ddl]^{\mu} & A^{p,q+1}\ar@{=>}[ddl]^{\mu} & A^{p,q+2}\ar@{=>}[ddl]^{\mu}\\
A^{p+1,q-1}\ar[rru]^{\bar\mu} &A^{p+1,q} \ar[rru]^{\bar\mu} \ar@{=>}[ddl]^{\mu} & A^{p+1,q+1}\ar@{=>}[ddl]^{\mu}&  A^{p+1,q+2} \ar@{=>}[ddl]^{\mu}\\
A^{p+2,q-1}\ar[rru]^{\bar\mu}&A^{p+2,q}\ar[rru]^{\bar\mu} & A^{p+2,q+1} & A^{p+2,q+2}\\
A^{p+2,q-1}\ar[rru]^{\bar\mu}&A^{p+2,q}\ar[rru]^{\bar\mu} & A^{p+2,q+1} & A^{p+2,q+2}
}.
\end{displaymath}
By $d^{2}=0$, and splitting via types, we have
\begin{align*}
\mu^{2}&=0\quad (4,-2),\\
\mu\partial+\partial\mu&=0\quad (3,-1),\\
\mu\bar\partial+\bar\partial\mu+\partial^{2}&=0\quad (2,0),\\
\mu\bar\mu+\partial\bar\partial+\bar\partial\partial+\bar\mu\mu&=0\quad (1,1),\\
\bar\mu\partial+\partial\bar\mu+\bar\partial^{2}&=0 \quad (0,2),\\
\bar\mu\bar\partial+\bar\partial\bar\mu&=0 \quad (-1,3),\\
\bar\mu^{2}&=0\quad (-2,4).
\end{align*}
By $\bar\mu^{2}=0$, we get a chain
\begin{equation*}
\cdots\rightarrow A^{p+1,q-2}\stackrel{\bar\mu}{\longrightarrow} A^{p,q}\stackrel{\bar\mu}{\longrightarrow} A^{p-1,q+2}\rightarrow\cdots.
\end{equation*}
We define the $\bar\mu$-cohomology of $(M,J)$ as
\begin{equation*}
H^{p,q}_{\bar\mu}(M):=\frac{\mathrm{Ker}\,(\bar\mu:A^{p,q}\rightarrow A^{p-1,q+2})}{\mathrm{Im}\,(\bar\mu:A^{p+1,q-2}\rightarrow A^{p,q})}.
\end{equation*}
By $\bar\partial\bar\mu+\bar\mu\bar\partial=0$, we know that
\begin{equation*}
\bar\partial:H^{p,q}_{\bar\mu}(M)\rightarrow H^{p,q+1}_{\bar\mu}(M),
\end{equation*}
is a well-defined map. Explicitly, if $[\varphi]\in H^{p,q}_{\bar\mu}(M)$, $\varphi\in A^{p,q}(M)$ such that $\bar\mu\varphi=0$, and $\psi=\varphi+\bar\mu\rho$, then
\begin{equation*}
\bar\mu\bar\partial\varphi=-\bar\partial\bar\mu\varphi=0,
\end{equation*}
i.e. $\bar\partial\varphi\in\mathrm{Ker}\,\bar\mu$, and
\begin{equation*}
\bar\partial\psi=\bar\partial\varphi+\bar\partial\bar\mu\rho=\bar\partial\varphi-\bar\mu\bar\partial\rho.
\end{equation*}
Hence we know that
\begin{equation*}
\bar\partial\psi-\bar\partial\varphi=-\bar\mu\bar\partial\rho\in\mathrm{Im}\,\bar\mu,
\end{equation*}
or $[\bar\partial\psi]=[\bar\partial\varphi]$. The equation $\bar\mu\partial+\partial\bar\mu+\bar\partial^{2}=0$ implies that $-\bar\partial^{2}$ is homotopic to zero (or null homotopic), with respect to the (chain) differential $\bar\mu$ and the chain homotopy $\partial$, precisely
\begin{displaymath}
\xymatrix{
\ar[r] &A^{p+1,q-2}\ar[r]^{\bar\mu} & A^{p,q}\ar[dl]_{\partial} \ar[r]^{\bar\mu} \ar@{.>}[d]^{f}& A^{p-1,q+2}\ar[dl]_{\partial}\ar[r] & \\
\ar[r] &A^{p+1,q} \ar[r]^{\bar\mu} & A^{p,q+2} \ar[r]^{\bar\mu} & A^{p-1,q+4}\ar[r]&
}.
\end{displaymath}
By this graph, we see that the map $f:=\bar\mu\partial+\partial\bar\mu$ is a null-homotopic chain map, i.e.
\begin{equation*}
f:(A^{p,\ast},\bar\mu)\rightarrow(A^{p+1,\ast},\bar\mu),
\end{equation*}
and then we know that the induced maps of the cohomology groups
\begin{equation*}
f_{\ast}:H^{p,\ast}_{\bar\mu}(M)\rightarrow H^{p+1,\ast}_{\bar\mu}(M),
\end{equation*}
are zero, i.e.
\begin{equation*}
-\bar\partial^{2}:H^{p,q}_{\bar\mu}(M)\rightarrow H^{p,q+1}_{\bar\mu}(M)\rightarrow H^{p,q+2}_{\bar\mu}(M),
\end{equation*}
is zero. This means that
\begin{equation*}
\cdots\rightarrow H^{p,q-1}_{\bar\mu}(M)\stackrel{\bar\partial}{\longrightarrow} H^{p,q}_{\bar\mu}(M)\stackrel{\bar\partial}{\longrightarrow} H^{p,q+1}_{\bar\mu}(M)\rightarrow\cdots,
\end{equation*}
is a chain complex. Using this chain complex, we define the Dolbeault cohomology of $(M,J)$ as
\begin{equation*}
H^{p,q}_{\mathrm{Dol}}(M):=H^{q}(H^{p,\ast}_{\bar\mu}(M),\bar\partial)=\frac{\mathrm{Ker}\,(\bar\partial :H^{p,q}_{\bar\mu}(M)\rightarrow H^{p,q+1}_{\bar\mu}(M))}{\mathrm{Im}\,(\bar\partial: H^{p,q-1}_{\bar\mu}(M)\rightarrow H^{p,q}_{\bar\mu}(M))}.
\end{equation*}
To emphasise the almost complex structure $J$, we also denote $H^{p,q}_{\mathrm{Dol},J}(M)\triangleq H^{p,q}_{\mathrm{Dol}}(M)$. By \cite[Porposition 4.12]{cirici2018dolbeault} we know that $H^{p,0}_{\mathrm{Dol}}(M),H^{p,n}_{\mathrm{Dol}}(M)$ ($0\leq p\leq n$) are finite dimensional when $M$ is compact.

\section*{Acknowledgement}
Hai-Sheng Liu would like to express his great gratitude to Sheng Rao and Quan-Ting Zhao for their sincerely advices and kindhearted help on both mathematical and non-mathematical aspects of this manuscript. Hai-Sheng Liu would also like to thank Wei Xia for his useful discussions. We are grateful to the two anonymous referees for useful comments and suggestions.

Fu is supported by NSFC grants 11871016.

%\bibliographystyle{plain}
%\bibliography{selfbib}
\end{document}